\documentclass[12pt]{amsart}

\usepackage{hyperref}
\usepackage{amsmath, amssymb, amscd, paralist,tabularx,supertabular,amsmath, verbatim, amsthm, mathrsfs, manfnt,times,latexsym,graphicx, color}
\usepackage[all]{xy}

\usepackage{fullpage}

\DeclareMathOperator{\area}{area}

\DeclareMathOperator{\Cl}{Cl}

\DeclareMathOperator{\covol}{covol}

\DeclareMathOperator{\disc}{disc}

\DeclareMathOperator{\opdiv}{div}

\DeclareMathOperator{\Gal}{Gal}

\DeclareMathOperator{\GL}{GL}

\DeclareMathOperator{\grad}{grad}

\DeclareMathOperator{\impart}{Im}

\DeclareMathOperator{\M}{M}

\DeclareMathOperator{\N}{N}

\DeclareMathOperator{\nr}{nr}
\DeclareMathOperator{\nrd}{nrd}

\DeclareMathOperator{\ord}{ord}
\DeclareMathOperator{\opP}{P}
\DeclareMathOperator{\PGL}{PGL}

\DeclareMathOperator{\PSL}{PSL}

\DeclareMathOperator{\repart}{Re}

\DeclareMathOperator{\rk}{rk}

\DeclareMathOperator{\sgn}{sgn}

\DeclareMathOperator{\Tr}{Tr}

\DeclareMathOperator{\trd}{trd}

\DeclareMathOperator{\vol}{vol}

\newcommand{\C}{\mathbb C}
\newcommand{\F}{\mathbb F}
\newcommand{\HH}{\mathbb H}
\newcommand{\PP}{\mathbb P}
\newcommand{\Q}{\mathbb Q}
\newcommand{\R}{\mathbb R}
\newcommand{\Z}{\mathbb Z}

\newcommand{\fraka}{\mathfrak{a}}
\newcommand{\frakb}{\mathfrak{b}}
\newcommand{\frakc}{\mathfrak{c}}
\newcommand{\frakd}{\mathfrak{d}}

\newcommand{\frakp}{\mathfrak{p}}
\newcommand{\frakq}{\mathfrak{q}}

\newcommand{\frakD}{\mathfrak{D}}

\newcommand{\frakM}{\mathfrak{M}}
\newcommand{\frakN}{\mathfrak{N}}

\newcommand{\calH}{\mathcal{H}}

\newcommand{\calL}{\mathcal{L}}

\newcommand{\calO}{\mathcal{O}}

\newcommand{\la}{\langle}
\newcommand{\ra}{\rangle}

\newcommand{\psmod}[1]{~(\textup{\text{mod}}~{#1})}

\newcommand{\quat}[2]{\displaystyle{\biggl(\frac{#1}{#2}\biggr)}}

\def\bs{\backslash}

\newcommand{\bhat}{\widehat{\mathfrak{b}}}
\newcommand{\betahat}{\widehat{\beta}}

\newcommand{\muhat}{\widehat{\mu}}
\newcommand{\Bhat}{\widehat{B}}
\newcommand{\Ehat}{\widehat{E}}
\newcommand{\Fhat}{\widehat{F}}
\newcommand{\calOhat}{\widehat{\mathcal{O}}}
\newcommand{\Zhat}{\widehat{\mathbb{Z}}}
\newcommand{\Qhat}{\widehat{\mathbb{Q}}}
\newcommand{\ZFhat}{\widehat{\mathbb{Z}}_F}
\newcommand{\Khat}{\widehat{K}}
\newcommand{\Gammahat}{\widehat{\Gamma}}

\newcommand{\defi}[1]{{{\fontfamily{phv}\fontsize{10.5pt}{1em}\selectfont \textbf{#1}}}}

\numberwithin{equation}{section}

\theoremstyle{remark}

\newtheorem{exm}[equation]{Example}

\newtheorem{rmk}[equation]{Remark}

\theoremstyle{plain}
\newtheorem{thm}[equation]{Theorem}
\newtheorem{theorem}[equation]{Theorem}
\newtheorem*{thma}{Theorem A}
\newtheorem*{thmb}{Theorem B}
\newtheorem*{thmc}{Theorem C}
\newtheorem*{thmd}{Theorem D}
\newtheorem{prop}[equation]{Proposition}

\newtheorem{lem}[equation]{Lemma}

\theoremstyle{definition}
\newtheorem{defn}[equation]{Definition}

\title{Small isospectral and nonisometric orbifolds \\ of dimension 2 and 3}

\author{Benjamin Linowitz}
\address{Department of Mathematics\\ 
530 Church Street\\
University of Michigan\\
Ann Arbor, MI 48109, USA}
\email{linowitz@umich.edu}

\author{John Voight}
\address{Department of Mathematics, Dartmouth College, 6188 Kemeny Hall, Hanover, NH 03755, USA; Department of Mathematics and Statistics, University of Vermont, 16 Colchester Ave, Burlington, VT 05401, USA}
\email{jvoight@gmail.com}

\begin{document}

\begin{abstract} 
Revisiting a construction due to Vign\'eras, we exhibit small pairs of orbifolds and manifolds of dimension $2$ and $3$ arising from arithmetic Fuchsian and Kleinian groups that are Laplace isospectral (in fact, representation equivalent) but nonisometric. 
\end{abstract}

\maketitle

\subsection*{Introduction}

In 1966, Kac \cite{Kac-Canyouhear} famously posed the question: ``Can one hear the shape of a drum?''  In other words, if you know the frequencies at which a drum vibrates, can you determine its shape?  Since this question was asked, hundreds of articles have been written on this general topic, and it remains a subject of considerable interest \cite{GiraudThas}.

Let $(M,g)$ be a connected, compact Riemannian manifold (with or without boundary).  Associated to $M$ is the Laplace operator $\Delta$, defined by $\Delta(f)=-\opdiv(\grad(f))$ for $f \in L^2(M,g)$ a square-integrable function on $M$.  The eigenvalues of $\Delta$ on the space $L^2(M,g)$ form an infinite, discrete sequence of nonnegative real numbers $0=\lambda_0 < \lambda_1 \leq \lambda_2 \leq \dots$, called the \defi{spectrum} of $M$.  In the case that $M$ is a planar domain, the eigenvalues in the spectrum of $M$ are essentially the frequencies produced by a drum shaped like $M$ and fixed at its boundary.  Two Riemannian manifolds are said to be \defi{Laplace isospectral} if they have the same spectra.  Inverse spectral geometry asks the extent to which the geometry and topology of $M$ are determined by its spectrum. For example, volume, dimension and scalar curvature can all be shown to be spectral invariants.

The problem of whether the isometry class itself is a spectral invariant received a considerable amount of attention beginning in the early 1960s.  In 1964, Milnor \cite{Milnor-Tori} gave the first negative example, exhibiting a pair of Laplace isospectral and nonisometric $16$-dimensional flat tori.  In 1980, Vign\'eras \cite{vigneras-isospectral} constructed non-positively curved, Laplace isospectral and nonisometric manifolds of dimension $n$ for every $n\geq 2$ (hyperbolic for $n=2,3$).  The manifolds considered by Vign\'eras are locally symmetric spaces arising from arithmetic: they are quotients by discrete groups of isometries obtained from unit groups of orders in quaternion algebras over number fields.  A few years later, Sunada \cite{Sunada} developed a general algebraic method for constructing Laplace isospectral manifolds arising from \emph{almost conjugate} subgroups of a finite group of isometries acting on a manifold, inspired by the existence of number fields with the same zeta function where the group and its subgroups form what is known as a \emph{Gassmann triple} \cite{BosmadeSmit}.  Sunada's method is extremely versatile and, along with its variants, accounts for the majority of known constructions of Laplace isospectral non-isometric manifolds \cite{DeturckGordon}.  Indeed, in 1992, Gordon, Webb and Wolpert \cite{GordonWebbWolpert} used an extension of Sunada's method in order to answer Kac's question in the negative for planar surfaces: ``One cannot hear the shape of a drum.''  For expositions of Sunada's method, the work it has influenced, and the many other constructions of Laplace isospectral manifolds, we refer the reader to surveys of Gordon \cite{Gordon-IsospectralSurvey, Gordon-SunadaSurvey} and the references therein.

The examples of Vign\'eras are distinguished as they cannot arise from Sunada's method: in particular, they do not cover a common orbifold \cite{Chen}. They are particularly beautiful as they come from `pure arithmetic' and appear quite naturally in the language of quadratic forms: the ``audible invariants'' of a lattice \cite{Conway-sensual} in a quadratic space were developed by Conway to study the genus of a quadratic form in this language, and Laplace isospectral but nonisometric  orbifolds arise when the (spinor) genus of an indefinite ternary quadratic form contains several distinct isometry classes.  This description allows further generalizations to hyperbolic $n$-orbifolds and $n$-manifolds for $n \geq 4$: see, for example, work of Prasad and Rapinchuk \cite{Prasad-Rapinchuk}.
 
Consequently, it is unfortunate that the literature surrounding these examples is plagued with errors.  The original explicit example given by Vign\'eras is incorrect due to a computational error: it is erroneously claimed \cite[p.~24]{vigneras-isospectral} that $3$ generates a prime ideal in $\Z[\sqrt{10}]$.  This error was subsequently corrected by Vign\'eras in her book \cite[IV.3.E]{vigneras-book}, however, the genus $g$ of the Riemann surfaces ($2$-manifolds) obtained from this construction are enormous: we compute them to be $g=100801$.  The method of Vign\'eras was later taken up by Maclachlan and Reid \cite[Example 12.4.8]{mac-reid-book}, who modified it to produce a pair of surfaces with genus $29$; due to a typo, it is incorrectly claimed that the genus of these $2$-manifolds is $19$.  Maclachlan and Reid also claim that the manifolds produced by this method tend to have large volume, and they note that there is no estimate on the smallest volume of a pair of Laplace isospectral but nonisometric hyperbolic $3$-manifolds \cite[p.~392]{mac-reid-book}.  Finally, in 1994, Maclachlan and Rosenberger \cite{MaclachlanRosenberger} claimed to have produced a pair of hyperbolic $2$-orbifolds of genus $0$ with orbifold fundamental groups of signature $(0;2,2,3,3)$.  However, Buser, Flach, and Semmler \cite{Buser2233} later showed that these examples were too good to be true and that such orbifolds could not be Laplace isospectral.  In an appendix, Maclachlan and Rosenberger \cite{Buser2233} give an arithmetic explanation of why their orbifolds were not Laplace isospectral: there is a subtlety in the embedding theory of orders in quaternion algebras called \emph{selectivity}, first studied by Chinburg and Friedman \cite{Chinburg-Friedman}.

In this article, we correct these errors and exhibit pairs of orbifolds and manifolds of dimensions $2$ and $3$ with minimal area (genus) and volume within certain nice classes of arithmetic groups.

\subsection*{Dimension 2}

We begin with the case of dimension $2$.  A Fuchsian group $\Gamma \leq \PSL_2(\R)$  acts properly discontinuously by isometries on the upper half-plane $\calH_2=\{x+yi \in \C : y>0\}$.  A Fuchsian group $\Gamma \leq \PSL_2(\R)$ is \defi{arithmetic} if it is commensurable with the image of the unit group of a maximal order in a quaternion algebra over a totally real field split at a unique real place.  A Fuchsian group is \defi{maximal} if it so under inclusion, and by a result of Borel \cite{borel-commensurability}, a maximal group is the normalizer of the group of totally positive units in an Eichler order.  If $\Gamma$ is an arithmetic Fuchsian group, then the orientable Riemann $2$-orbifold $X(\Gamma)=\Gamma \backslash \calH_2$ has finite area and is compact unless $\Gamma$ is commensurable with $\PSL_2(\Z)$, in which case we instead compactify and take $X(\Gamma)=\Gamma \backslash \calH_2^*$ where $\calH_2^*=\calH_2 \cup \PP^1(\Q)$ is the completed upper half-plane.  The orbifold $X(\Gamma)$ is a $2$-manifold if and only if $\Gamma$ is torsion free.  (The orbifold $X(\Gamma)$ can also be given the structure of a Riemann $2$-manifold or Riemann surface with a different choice of atlas in a neighborhood of a point with nontrivial stabilizer, `smoothing the corners' of the original orbifold.  However, this procedure changes the spectrum, and indeed ``one can hear the orders of the cone points of orientable hyperbolic $2$-orbifold'' \cite{DoyleRossetti}.  For further discussion on this point, see also work of Dryden \cite{Dryden}, Dryden, Gordon, Greenwald, and Webb \cite[Section 5]{Drydenetal}, and Gordon and Rossetti \cite[Proposition 3.4]{GordonRossetti}.)

If $\Gamma,\Gamma'$ are arithmetic Fuchsian groups such that the $2$-orbifolds $X(\Gamma)$, $X(\Gamma')$ are Laplace isospectral but not isometric, then we say that they are an \defi{isospectral-nonisometric pair}.  The \defi{area} of such a pair is the common area of $X(\Gamma)$ and $X(\Gamma')$.  


The Fuchsian group with the smallest co-area is the arithmetic triangle group with co-area $\pi/21$ and signature $(0;2,3,7)$: it arises as the group of totally positive (equivalently, norm $1$) units in a maximal order of a quaternion algebra defined over the totally real subfield $F=\Q(\zeta_7)^+$ of the cyclotomic field $\Q(\zeta_7)$.  Quotients by torsion-free normal subgroups of this group correspond to Hurwitz curves, those Riemann surfaces with largest possible automorphism group for their genus.  

Our first main result is as follows (Theorem \ref{theorem:bestorbifold}).  

\begin{thma} 
The minimal area of an isospectral-nonisometric pair of $2$-orbifolds associated to maximal arithmetic Fuchsian groups is $23\pi/6$, and this bound is achieved by exactly three pairs, up to isomorphism.
\end{thma} 

Since $(23\pi/6)/(\pi/21) = 80.5$, the three pairs in Theorem A indeed have relatively small area; they all have signature $(0;2,2,2,2,2,3,4)$, so in particular the underlying Riemann surfaces have genus zero.  These pairs are described in detail in Section 3: they arise from the method of Vign\'eras and hence are non-Sunada.  The groups fall naturally into two Galois orbits: each Galois orbit arises from unit groups of maximal orders in a quaternion algebra defined over a totally real field $F$ and the pairs correspond to the possible choices of the split real place: the two totally real fields $F$ are $\Q(\sqrt{34})$ (either choice of split real place gives an isomorphic pair) and the primitive totally real quartic field of discriminant $21200$, generated by a root of the polynomial $x^4-6x^2-2x+2$ (two different choices of four possible split real places yield two nonisometric pairs).  

Our second result concerns manifolds (Theorem \ref{theorem:manifoldexamples}).  Recall that the area $A$ of a $2$-manifold is determined by its genus $g$ according the formula $A=2\pi (2g-2)$.  For this theorem, we turn to a different class of groups to suitably enrich the set of interesting examples.  We say that an arithmetic Fuchsian group is \defi{unitive} if it contains the group of units of reduced norm $1$ and is contained in the group of totally positive units of an Eichler order, with equality in either case allowed.  Unitive Fuchsian groups are not necessarily maximal---however, a maximal group properly containing a unitive Fuchsian group very often has torsion and so rarely gives rise to a $2$-manifold. 

\begin{thmb} 
The minimal genus of an isospectral-nonisometric pair of $2$-manifolds associated to unitive torsion-free Fuchsian groups is $6$, and this bound is achieved by exactly two pairs, up to isomorphism.
\end{thmb}

We exhibit the pairs of genus $6$ manifolds achieving the bound in Theorem B in Example  \ref{exm:manifoldexample}; they fall into one Galois orbit, defined over the totally real quartic field of discriminant $5744$.

\subsection*{Dimension 3}

We now turn to analogous results in dimension $3$, looking at Kleinian groups $\Gamma \leq \PSL_2(\C)$ acting on hyperbolic $3$-space $\calH_3=\C \times \R_{>0}$ and their quotients $X(\Gamma)=\Gamma \backslash \calH_3^{(*)}$, which are complete, orientable hyperbolic $3$-orbifolds.  As opposed to the case of $2$-orbifolds, by the Mostow rigidity theorem \cite{Mostow,Prasad}, hyperbolic volume is a topological invariant.

Gehring, Marshall, and Martin \cite{GM,MM}, after a long series of prior results, have identified the smallest volume hyperbolic $3$-orbifold: it is the orbifold associated to a maximal arithmetic Kleinian group associated to a maximal order in the quaternion algebra over the quartic field $F$ (generated by a root of $x^4+6x^3+12x^2+9x+1$) of discriminant $-275$ ramified at the two real places.  This orbifold has volume 
\[ \frac{275^{3/2}\zeta_F(2)}{2^{7} \pi^{6}}=0.03905\ldots \]
where $\zeta_F(s)$ is the Dedekind zeta function of $F$; its uniformizing group is a two-generated group obtained as a $\Z/2\Z$-extension of the orientation-preserving subgroup of the group generated by reflection in the faces of the $3$-$5$-$3$-hyperbolic Coxeter tetrahedron.  Chinburg and Friedman \cite{chinburg-smallestorbifold} had previously shown that it is the \emph{arithmetic} hyperbolic $3$-orbifold with smallest volume.  

The method of Vign\'eras, and in particular all of the examples of isospectral-nonisometric pairs produced in this paper, are more than just isospectral: they are in fact representation equivalent (see \S 2) and consequently strongly isospectral (isospectral with respect to all natural strongly elliptic operators, such as the Laplacian acting on $p$-forms for each $p$).  Laplace isospectrality and representation equivalence are known to be equivalent for $2$-orbifolds (see Remark \ref{rmk:repequiv}) but it is already unknown for $3$-manifolds.  For the purposes of our algebraic techniques, it is preferable to work with the stronger condition of representation equivalence; accordingly, we say that a pair of arithmetic (Fuchsian or) Kleinian groups $\Gamma,\Gamma'$ are a \defi{representation equivalent-nonisometric pair} if the orbifolds $X(\Gamma)$, $X(\Gamma')$ are 
representation equivalent but not isometric.

We define a \defi{unitive} Kleinian group analogously as for a Fuchsian group.  To simplify our search in dimension $3$, we restrict to unitive groups (see Remark \ref{rmk:nonunitive3orbifold} for the scope of the calculation required to extend this to maximal groups).  

\begin{thmc} 
The smallest volume of a representation equivalent-nonisometric pair of $3$-orbifolds associated to unitive Kleinian groups is $2.8340\ldots,$ and this bound is attained by a unique pair, up to isomorphism.  
\end{thmc}

The pair achieving the bound in Theorem C is arises from a quaternion algebra defined over the sextic field $F$ of discriminant $-974528$ generated by a root of $x^6-2x^5-x^4+4x^3-4x^2+1$ and ramified only at the four real places of $F$: the volume is given exactly by the expression
\[ \frac{15227^{3/2}\zeta_F(2)}{8\pi^{10}}=2.8340\ldots. \]
The ratio $2.8340/0.03905 = 72.57\ldots$ shows that this pair is indeed comparably small as in the case of $2$-orbifolds, again contrary to common belief that the examples of Vign\'eras have large volume.  For a complete description of this pair, see Example \ref{example:3orbifolds}.

Finally, we consider $3$-manifolds.  Gabai, Meyerhoff, and Milley identified the Weeks manifold as the lowest volume $3$-manifold \cite{GMM1,GMM2}, following work of Chinburg, Friedman, Jones, and Reid \cite{Chinburg-Friedman-Jones-Reid-smallestmanifold} who showed it was the smallest arithmetic $3$-manifold.  It arises from a subgroup of index $12$ in a maximal group associated to a maximal order in the quaternion algebra over the cubic field of discriminant $-23$ (generated by a root of $x^3-x+1$) ramified at the real place and the prime of norm $5$, or equivalently is obtained by $(5, 1)$, $(5, 2)$ surgery on the two components of the Whitehead link.  The Weeks manifold has volume 
\[ \frac{3 \cdot 23^{3/2} \zeta_F(2)}{4\pi^4} = 0.9427\ldots. \]
(Minimal volume hyperbolic $3$-orbifolds and manifolds were previously known to exist, but not explicitly, by a theorem of J\o{}rgenson and Thurston \cite{DunbarMeyerhoff,Gromov}.)

\begin{thmd} 
There exist representation equivalent-nonisometric $3$-manifolds associated to torsion-free arithmetic Kleinian groups of covolume $39.2406\ldots$\,.
\end{thmd}

This pair of manifolds is described in Example \ref{example:3manifolds}.  This example is the smallest we found looking among $3$-manifolds, but to check if it is indeed the smallest would require significant effort (because of its large volume; see Remark \ref{rmk:nonunitive3manifold}).

\subsection*{Summary}

Our contributions in this article are as follows.  First, we give a criterion which allows us to establish isospectrality for quaternionic Shimura varieties arising from maximal arithmetic groups, building on the foundational work of Vign\'eras \cite{vigneras-isospectral} and generalizing work of Rajan \cite{Rajan}.  This criterion allows more flexibility in considering Fuchsian and Kleinian groups to find the smallest examples.  Second, following the strategy employed by many authors \cite{Agol, chinburg-smallestorbifold, Chinburg-Friedman-Jones-Reid-smallestmanifold, GMMR, mac-reid-book, Voight-shim}, we apply the volume formula due to Borel and discriminant bounds due to Odlyzko to reduce the question to a finite calculation.  However, the simplest bounds one might obtain this way are far too large to be of practical use; we combine several additional constraints (some new, some old) that conspire together to whittle the bounds down to a still large but computable range.  These methods allow us to list all arithmetic groups of bounded volume, in particular reproving the results of Chinburg and Friedman and Chinburg, Friedman, Jones, and Reid cited above, replacing technical lemmas with a sequence of algorithmically checkable bounds.  We expect that this blend of ideas will be useful in other domains: they have already been adapted to the case of arithmetic reflection groups by Belolipetsky and Linowitz \cite{MishaBen}.  Finally, we use algorithms for quaternion algebras and arithmetic Fuchsian and Kleinian groups, implemented in \textsc{Magma} \cite{Magma}, to carry out the remaining significant enumeration problem: we list the possible pairs and identify the ones of smallest area under congruence hypotheses for both orbifolds and manifolds in dimensions $2$ and $3$. 

This paper is organized as follows.  In section 1, we set up the basic notation and review the method of Vign\'eras.  In section 2, we consider the subjects isospectrality and selectivity.  In section 3, we exhibit the examples described in Theorem A.  In section 4, we prove Theorem A and describe the computations involved.  In section 5, we prove Theorem B.  Finally, in section 6, we prove Theorems C and D.

We leave several avenues open for further work.  First and foremost, it would be interesting to extend our results from maximal or unitive groups to the full class of \emph{congruence} (arithmetic) groups; our general setup allows this more general class, and we prove some results in this direction, but they are not included in our Theorems A--D because they require some rather more involved group-theoretic arguments as well as a treatment of selectivity that falls outside of the scope of this paper.  Second, considering nonorientable isospectral orbifolds also promises to provide some intriguing small examples.  Third, there is the case of higher dimension: one could continue with the extension of the method of Vign\'eras using quaternion algebras, but it is perhaps more natural to consider quotients of hyperbolic $n$-space for $n \geq 4$, working explicitly instead with the genus of a quadratic form.  Finally, it would be interesting to seek a common refinement of our work with the work of Prasad and Rapinchuk on weak commensurability \cite{Prasad-Rapinchuk} and its relationship to isospectrality.

The authors would like to thank Peter Doyle for his initial efforts and enduring support on this project, as well as Peter Buser, Richard Foote, Carolyn Gordon, Dubi Kelmer, Markus Kirschmer, Aurel Page, and the anonymous referee for helpful comments.  The first author was partially supported by an NSF RTG Grant (DMS-1045119) and an NSF Postdoctoral Fellowship (DMS-1304115); the second author was partially supported by an NSF CAREER Award (DMS-1151047).

\section{Quaternionic Shimura varieties} \label{section:shimuravarieties}

In this first section, we set up notation used in the construction of Vign\'eras \cite{vigneras-isospectral,vigneras-book}; see also the exposition by Maclachlan and Reid \cite{mac-reid-book}.  We construct orbifolds as quotients of products of upper half-planes and half-spaces by arithmetic groups arising from quaternion orders; for more on orbifolds, see e.g.\ Scott \cite{Scott} or Thurston \cite{Thurston}.  To be as self-contained as possible, and because they will become important later, we describe some of this standard material in detail.  (For more detail on orders see Reiner \cite{Reiner} as well as work of Kirschmer and Voight for an algorithmic perspective on orders and ideals in quaternion algebras \cite{KirschmerVoight}.)

\subsection*{Quaternion algebras and orders}

Let $F$ be a number field and let $\Z_F$ be its ring of integers.  A \defi{quaternion algebra} $B$ over $F$ is a central simple algebra of dimension $4$ over $F$, or equivalently an $F$-algebra with generators $i,j \in B$ such that 
\begin{equation} \label{quateq}
i^2=a, \quad j^2=b, \quad ji=-ij
\end{equation}
with $a,b \in F^*$; such an algebra is denoted $B=\quat{a,b}{F}$.

Let $B$ be a quaternion algebra over $F$.  Then $B$ has a unique (anti-)involution $\overline{\phantom{x}}:B \to B$ called \defi{conjugation} such that the \defi{reduced trace} $\trd(\alpha)=\alpha+\overline{\alpha}$ and the \defi{reduced norm} $\nrd(\alpha)=\alpha\overline{\alpha}$ belong to $F$ for all $\alpha \in B$.  

Let $K$ be a field containing $F$ as a subfield.  Then $B_K=B \otimes_F K$ is a quaternion algebra over $K$, and we say $K$ \defi{splits} $B$ if $B_K \cong M_2(K)$.  If $[K:F]=2$, then $K$ splits $B$ if and only if there exists an $F$-embedding $K \hookrightarrow B$.  Let $v$ be a place of $F$, and let $F_v$ denote the completion of $F$ at $v$.  We say $B$ is \defi{split} at $v$ if $F_v$ splits $B$, and otherwise we say that $B$ is \defi{ramified} at $v$.  The set of ramified places of $B$ is of finite (even) cardinality and uniquely characterizes $B$ up to $F$-algebra isomorphism.  We define the \defi{discriminant} $\frakD$ of $B$ to be the ideal of $\Z_F$ given by the product of all finite ramified primes of $B$.  

An \defi{order} $\calO \subset B$ is a subring of $B$ (with $1 \in \calO$) with $\calO F =B$ that is finitely generated as a $\Z_F$-submodule.  A sub-$\Z_F$-algebra $\Lambda$ with $\Lambda F=B$ is a $\Z_F$-order if and only if each $\alpha \in \Lambda$ is integral ($\trd(\alpha),\nrd(\alpha) \in \Z_F$ for all $\alpha \in \Lambda$).  An order is \defi{maximal} if it is maximal under inclusion.  By the Skolem-Noether theorem, two orders $\calO,\calO'$ are isomorphic if and only if there exists $\nu \in B^\times$ such that $\calO' = \nu \calO \nu^{-1}$.  The \defi{(reduced) discriminant} $\frakd$ of an order $\calO$ is the square root of the $\Z_F$-ideal generated by
$\det(\trd(\gamma_i\gamma_j))_{i,j=1,\dots,4}$ for all $\gamma_1,\dots,\gamma_4 \in \calO$.  An order $\calO$ of reduced discriminant $\frakd$ is maximal if and only if $\frakd=\frakD$.


\subsection*{Quaternionic Shimura varieties}

Let $r$ and $c$ be the number of real and complex places of $F$, respectively, so that $[F:\Q]=r+2c=n$.  Suppose that $B$ is split at $s$ real places and is ramified at the remaining $r-s$ places, so that 
\begin{equation} \label{BR}
B \hookrightarrow B \otimes_{\Q} \R \xrightarrow{\sim} \M_2(\R)^s \times \HH^{r-s} \times \M_2(\C)^c
\end{equation}
Let $\iota=(\iota_1,\dots,\iota_{s+c})$ denote the projection $B \to \M_2(\R)^s \times \M_2(\C)^c$.  Let 
\[ F_+^\times=\{x \in F^\times: v(x)>0 \text{ for all real places $v$}\} \] 
be the group of totally positive elements of $F$, let $\Z_{F,+}^\times = \Z_F^\times \cap F_+^\times$, and let
\[ B_+^\times = \{ \alpha \in B^\times : \nrd(\gamma)\in F_+^\times\} \]
be the group of units in $B$ of totally positive reduced norm.  

Let $\calH_2=\{z=x+iy \in \C : \impart(z)>0\}$ and $\calH_3=\C \times \R_{>0}=\{(z,t) : z \in \C, t \in \R_{>0}\}$ be the hyperbolic spaces of dimensions $2$ and $3$, equipped with the metrics 
\begin{equation} \label{metrics}
ds^2=\frac{dx^2+dy^2}{y^2} \text{\ and\ } ds^2=\frac{dx^2+dy^2+dt^2}{t^2},
\end{equation}
respectively.  Then the group $\opP\!B_+^\times=B_+^\times/Z(B_+^\times)=B_+^\times/F^\times$ with
\[ \iota(\opP\!B_+^\times) \subset \PGL^+_2(\R)^s \times \PSL_2(\C)^c \] 
acts naturally on the symmetric space $\calH=\calH_2^s \times \calH_3^c$ (on the left) by orientation-preserving isometries.  We will identify the group $\opP\!B_+^\times$ and its subgroups with their images via $\iota$.

Now let $\calO(1) \subseteq B$ be a maximal order, and let $\Gamma$ be a subgroup of $\opP\!B_+^\times$ commensurable with $\opP\!\calO(1)^\times=\calO(1)^\times/\Z_F^\times$; then $\iota(\Gamma) \subset \PGL^+_2(\R)^s \times \PSL_2(\C)^c$ is a discrete \defi{arithmetic} subgroup acting properly discontinuously on $\calH$; accordingly, the quotient $X=X(\Gamma)=\Gamma \backslash \calH$ has the structure of a Riemannian orbifold, a Hausdorff topological space equipped with an orbifold structure (locally modelled by the quotient of $\R^n$ by a finite group).  The orbifold $X$ has dimension $m=2s+3c$ and is compact except if $B \cong \M_2(F)$, in which case we write instead $X(\Gamma)=\Gamma \backslash \calH^*$ where $\calH^*=\calH \cup \PP^1(F)$.  If $\Gamma$ is torsion free, then $\Gamma$ acts freely on $\calH$ and $X=\Gamma \backslash \calH$ is a Riemannian manifold.

The heart of the construction of Vign\'eras is to compare such quotients arising from orders that are locally isomorphic but not isomorphic.  This situation is most compactly presented by employing adelic notation.  Let $\Zhat = \varprojlim_{n} \Z/n\Z$ be the profinite completion of $\Z$.  We endow adelic sets with hats, and where applicable we denote by $\widehat{\phantom{x}}$ the tensor product with $\Zhat$ over $\Z$.  For example, the ring of adeles of $F$ is $\Fhat=F \otimes_{\Q} \Qhat = F \otimes_{\Z} \Zhat$.  

The group $\Gamma$ is \defi{maximal} if it is maximal under inclusion.  In this paper we pursue the simplest examples and so we will be consider the class of maximal groups.  However, in this section, we find it natural to work with a larger class of groups: $\Gamma$ is a \defi{congruence} subgroup if $\Gamma$ contains a \defi{principal congruence subgroup} 
\[ \Gamma(\frakN)=\{\gamma \in \calO(1) : \nrd(\gamma)=1 \text{ and } \gamma \equiv 1 \psmod{\frakN}\} \leq \calO(1)^\times \]
for an ideal $\frakN \subseteq \Z_F$.  

Congruence groups are characterized by their idelic images as follows.  Let $\Gammahat$ be the closure of $\Gamma$ with respect to the topology on $\Bhat^\times$, where a fundamental system of neighborhoods of $1$ is given by the images of the principal congruence subgroups $\widehat{\Gamma(\frakN)}$ for an ideal $\frakN \subseteq \Z_F$.

\begin{lem} \label{lem:congclos}
$\Gammahat \cap B^\times = \Gamma$ if and only if $\Gamma$ is congruence.
\end{lem}

\begin{proof}
The group $\Gammahat$ is a closed subgroup of $\Bhat^\times$ commensurable with the compact open subgroup $\widehat{\Gamma(1)}$ and so $\Gammahat$ is also compact open.  Let $\overline{\Gamma}=\Gammahat \cap B^\times$.  Then by the topology on $\Bhat^\times$, $\overline{\Gamma}$ is the smallest congruence group containing $\Gamma$.  So $\Gamma$ is congruence if and only if $\Gamma=\overline{\Gamma}$.
\end{proof}

It follows from Lemma \ref{lem:congclos} that if $\Gamma$ is maximal, then $\Gamma$ is congruence.  Accordingly, for the rest of this section, we assume that $\Gamma$ is congruence.  (One can always replace a general group $\Gamma$ with its \defi{congruence closure} $\overline{\Gamma}=\widehat{\Gamma} \cap B^\times$.)

\begin{defn} \label{def:samegenus}
Two groups $\Gamma,\Gamma' \leq B^\times$ are \defi{everywhere locally conjugate} if $\Gammahat,\Gammahat'$ are conjugate in $\Bhat^\times$.  The \defi{genus} of $\Gamma \leq B^\times$ is the set of groups $\Gamma' \leq B^\times$ locally conjugate to $\Gamma$; we therefore will also say that $\Gamma,\Gamma'$ are \defi{in the same genus} if $\Gamma,\Gamma'$ are locally conjugate.
\end{defn}

Consider the double coset
\[ X_{\textup{Sh}}=X_{\textup{Sh}}(\Gamma)=B_+^\times \backslash (\calH \times \Bhat^\times / \Gammahat) \]
where $B_+^\times$ acts on $\calH$ via $\iota$ and on $\Bhat^\times/\Gammahat$ by left multiplication via the diagonal embedding.  We have a natural (continuous) projection map
\begin{equation} \label{projmap}
X_{\textup{Sh}} \to B_+^\times \backslash \Bhat^\times / \Gammahat.
\end{equation}
We suppose now and throughout that $s+c>0$, so $B$ has at least one split real or (necessarily split) complex place; this hypothesis is also known as the \defi{Eichler condition}.  (In the excluded case $s+c=0$, the field $F$ is totally real, $B$ is a totally definite quaternion algebra over $F$, and the associated variety $X_{\textup{Sh}}$ is a finite set of points.)  Under this hypothesis, by strong approximation \cite[Theor\`eme III.4.3]{vigneras-book}, the reduced norm gives a bijection
\begin{equation} \label{strongapprox}
\nrd:B_+^\times \backslash \Bhat^\times / \Gammahat \xrightarrow{\sim} F_+^\times \backslash \Fhat^\times/\!\nrd(\Gammahat) = \Cl_\Gamma
\end{equation}
where $\Cl_\Gamma$ is a class group of $F$ associated to $\Gamma$.  The group $\Cl_\Gamma$ is a finite abelian group that surjects onto the strict class group $\Cl^+ \Z_F$, the ray class group of $\Z_F$ with modulus equal to the product of all real (infinite) places of $F$; by class field theory, $\Cl^+ \Z_F$ is the Galois group of the maximal abelian extension of $F$ which is unramified at all finite places.

Thus, the space $X_{\textup{Sh}}$ is the disjoint union of Riemannian orbifolds of dimension $m=2s+3c$ indexed by $\Cl_\Gamma$, which we identify explicitly as follows.  Let the ideals $\frakb \subseteq \Z_F$ form a set of representatives for $\Cl_\Gamma$, let $\bhat = \frakb \otimes_{\Z} \Zhat$, and let $\betahat \in \ZFhat$ generate $\bhat$ so $\widehat{b} \ZFhat \cap \Z_F = \frakb$.  For simplicity, choose $\frakb=\Z_F$ and $\widehat{b}=\widehat{1}$ for the representatives of the trivial class.  
By strong approximation (\ref{strongapprox}), there exists $\betahat \in \Bhat^\times$ such that $\nrd(\betahat)=\widehat{b}$.  Therefore
\begin{equation} \label{breakup}
 X_{\textup{Sh}} = \bigsqcup_{[\frakb] \in \Cl_\Gamma} B_+^\times (\calH \times \betahat \Gammahat).
\end{equation}
We have a map
\begin{align*} 
B_+^\times(\calH \times \betahat \Gammahat) &\to X(\Gamma_{\frakb})=\Gamma_\frakb \backslash \calH \\
(z,\betahat \Gammahat) &\mapsto z
\end{align*}
where $\Gamma_\frakb=\betahat \Gammahat \betahat^{-1} \cap B$; now each $X(\Gamma_\frakb)$ is a connected orbifold of dimension $m$.  We abbreviate $X(\Gamma)=X(\Gamma_{(1)})$ for the trivial class.  Putting these together, we have an isomorphism of orbifolds
\begin{equation} \label{decomp}
X_{\textup{Sh}}(\Gamma) = \bigsqcup_{[\frakb] \in \Cl_\Gamma} X(\Gamma_{\frakb}),
\end{equation}
i.e., a homeomorphism of the underlying topological spaces that preserves the orbifold structure.  When $c=0$, i.e., $F$ is totally real, then $X_{\textup{Sh}}$ can be given the structure of an algebraic variety defined over a number field, and in this case we call $X_{\textup{Sh}}$ the \defi{quaternionic Shimura variety} associated to $B$ of \defi{level} $\Gamma$.  

Having made this decomposition, will see in Section 2, following Vign\'eras, that the connected orbifolds $X(\Gamma_\frakb)$ in (\ref{decomp}) are in some cases isospectral and nonisometric; and to some extent, it is exactly the fact that these connected orbifolds arise naturally together that makes them candidates to form an isospectral-nonisometric pair.  To that end, we now turn to a class of these orbifolds which will be the primary objects of interest.

\subsection*{Eichler orders and normalizers}

Let $\frakN$ be an ideal of $\Z_F$ coprime to the discriminant $\frakD$ and let $\calO=\calO_0(\frakN) \subseteq \calO_0(1)$ be an \defi{Eichler order of level} $\frakN$, consisting of elements which are upper triangular modulo $\frakN$ with respect to an embedding $\calO \hookrightarrow \calO \otimes_{\Z_F} \Z_{F,\frakN}$, where $\Z_{F,\frakN}$ is the completion of $\Z_F$ at $\frakN$.  Then the discriminant of $\calO$ is equal to $\frakd=\frakD\frakN$.  From the local description, we see that 
\[ \nrd(\widehat{\calO}^\times) = \widehat{\Z}_{F}^\times. \] 
Let 
\[ \calO_1^\times=\{\gamma \in \calO : \nrd(\gamma)=1\} \]
be the group of units of $\calO$ of reduced norm $1$; further, let
\[ N(\calO) = \{\alpha \in B^\times : \alpha^{-1} \calO \alpha \subseteq \calO \} \]
be the normalizer of $\calO$ in $B^\times$ and let $N(\calO)_+ = N(\calO) \cap B_+^\times$.  In this situation, we define three arithmetic groups and their quotients as follows:
\begin{align}\label{align:arithmeticgroups}
\Gamma^1=\Gamma_0^1(\frakN) &= \calO_1^\times & X^1 &= X_0^1(\frakN) = X(\Gamma^1) \notag  \\
\Gamma^+=\Gamma_0^+(\frakN) &= \calO_+^\times/\Z_F^\times & X^+ &= X_0^+(\frakN) = X(\Gamma^+) \\
\Gamma^*=\Gamma_0^*(\frakN) &= N(\calO)_+/F^\times & X^* &= X_0^*(\frakN) = X(\Gamma^*). \notag
\end{align}
We identify these groups with their images in $\PGL^+_2(\R)^s \times \PSL_2(\C)^c$ via $\iota$, as before.

We have the inclusions
\[ \Gamma^1 \leq \Gamma^+ \leq \Gamma^* \]
each with finite index, and hence finite surjective morphisms
\begin{equation} \label{finitemorph}
 X^1 \to X^+ \to X^*.
\end{equation}

\begin{rmk}
When $s > 1$, one can consider slightly larger orientation-preserving groups by requiring not that the elements $\gamma$ are totally positive but rather that $\N_{F/\Q}(\nrd \gamma)>0$; this class makes for more complicated notation, but can be defined analogously.  In this paper, we have $s \leq 1$, so this more nuanced class of groups does not arise.  Similarly, if $s>1$ or $c>1$, then the full (orientation-preserving) isometry group of $\calH$ contains maps permuting the factors; for the same reason, these groups do not arise here.  More general groups may still give interesting examples in other contexts.
\end{rmk}

An arithmetic Fuchsian group $\Gamma$ is \defi{maximal} if it is not properly contained in another (arithmetic) Fuchsian group.  Borel characterized the maximal arithmetic groups in terms of the above data as follows.

\begin{thm}[{Borel \cite[Proposition 4.4]{borel-commensurability}}] \label{Borel}
If an arithmetic subgroup of $\PSL_2(\R)^s \times \PSL_2(\C)^c$ is maximal, then it is isomorphic to a group of the form $\Gamma^*=\Gamma_0^*(\frakN)$ associated to an Eichler order of squarefree level $\frakN$.
\end{thm}

We will need to be quite explicit about the morphisms in (\ref{finitemorph}).  From now on, let $\calO \subset B$ be an Eichler order of level $\frakN$.  

\begin{defn}
An arithmetic group $\Gamma$ is \defi{unitive} if $\Gamma^1 \leq \Gamma \leq \Gamma^+$ for $\Gamma^1,\Gamma^+$ arising from an Eichler order.
\end{defn}

\begin{lem} \label{lem:gamma1*}
The subgroup $\Gamma^1 \trianglelefteq \Gamma^+$ is normal with quotient
\[ \Gamma^+/\Gamma^1 \cong \Z_{F,+}^\times/\Z_F^{\times 2} \]
an elementary abelian $2$-group; the rank of this group is at most $r+c$, and it is at most $r+c-1$ if $F$ has a real place.  If further $F$ is totally real, then 
\[ \Z_{F,+}^\times/\Z_F^{\times 2} \cong \Cl^+(\Z_F)/\Cl(\Z_F). \]
\end{lem}

\begin{proof}
By strong approximation, the map $\nrd:\calO_+^\times \to \Z_{F,+}^\times$ is surjective, so the map $\Gamma^+ \to \Z_{F,+}^\times/\Z_F^{\times 2}$ is surjective with kernel identified with $\Gamma^1$.
\end{proof}

The normalizer $N(\calO)$ and its corresponding arithmetic group $\Gamma^*$ are a bit more complicated but will feature prominently in our analysis of $2$-orbifolds.  An element $\alpha \in B^\times$ belongs to the normalizer $\alpha \in N(\calO)$ if and only if $\alpha \calO=\calO \alpha$ if and only if $\calO \alpha \calO$ is a principal two-sided fractional $\calO$-ideal.  Indeed, the map from the normalizer $N(\calO)$ to the set of principal two-sided $\calO$-ideals given by $\alpha \mapsto \calO \alpha \calO$ is surjective with kernel $\calO^\times$; the map restricted to elements with positive reduced norm has kernel $\calO_+^\times$.  
We describe the principal two-sided ideals by describing the locally principal ideals and identifying the corresponding ideal classes by strong approximation.  We write $\fraka \parallel \frakd$, and say $\fraka$ is a \defi{unitary divisor} of $\frakd$, if $\fraka \mid \frakd$ and $\fraka$ is prime to $\frakd\fraka^{-1}$.  

\begin{prop} \label{twosided}
Let $J$ be a locally principal two-sided fractional ideal.  Then $J$ is a product of an ideal $\frakc \calO$, with $\frakc$ a fractional ideal of $\Z_F$, and two-sided $\calO$-ideals of the form
\[ [\calO,\calO] + \frakp^e \calO \]
with $\frakp^e \parallel \frakd$ a prime power, where $[\,,\,]$ is the commutator.
\end{prop}

\begin{proof}
The proof is direct; see the work of Kirschmer and Voight \cite[\S 3]{KirschmerVoight}, for example.
\end{proof}

The set 
\[ \{ \fraka \parallel \frakd : [\fraka] \in (\Cl^+ \Z_F)^2\}  \]
of unitary divisors of $\frakd=\frakD\frakN$ which are squares in $\Cl^+ \Z_F$ can be given the structure of an elementary abelian $2$-group by considering it as a subgroup of $\bigoplus_{\frakp^e \parallel \frakd} \Z/2\Z$ in the natural way.  

We also define
\[ (\Cl \Z_F)[2]_+ = \{[\frakc] \in (\Cl \Z_F)[2] : \text{$[\frakc]$ lifts to an element of order $2$ in $\Cl^+ \Z_F$}\} \leq (\Cl \Z_F)[2]. \]
For example, if $\Cl^+ \Z_F = \Cl \Z_F$, then $(\Cl \Z_F)[2]_+ = (\Cl \Z_F)[2]$.  The subgroup $(\Cl \Z_F)[2]_+$ is the largest subgroup of $(\Cl \Z_F)[2]$ that splits as a direct summand in the natural projection map $\Cl^+ \Z_F \to \Cl \Z_F$.  

\begin{prop} \label{corgamma}
The subgroup $\Gamma^+ \trianglelefteq \Gamma^*$ is normal with quotient 
\[ \Gamma^*/\Gamma^+ \cong \{ \fraka \parallel \frakd : [\fraka] \in (\Cl^+ \Z_F)^2\} \times 
(\Cl \Z_F)[2]_+
\]
an elementary abelian $2$-group with rank at most $\omega(\frakd)+h_2$ where $\omega(\frakd)=\#\{\frakp : \frakp \mid 
\frakd\}$ and $h_2=\rk (\Cl \Z_F)[2]$.  

The quotient $\Gamma^*/\Gamma^1$ is an elementary abelian $2$-group of rank at most $r+c+\omega(\frakd) + h_2$ in general and $r-1+\omega(\frakd)+h_2$ if $F$ is totally real.
\end{prop}

\begin{proof}
Let $\alpha \in N(\calO)_+$.  From Proposition \ref{twosided} we have $(\calO \alpha \calO)^2 = \calO \alpha^2 \calO = a\calO$ with $a=\nrd(\alpha)$ so $\alpha^2 = a \gamma$ with $\gamma \in \calO^\times$; taking norms shows that $\gamma$ is totally positive.  For the second statement, we note further that $\nrd(\gamma)=\nrd(\alpha)^2/a^2=u^2$ so $\alpha^2=(au)(\gamma/u) \in F^\times \Gamma^1$ as claimed.  Thus $\Gamma^*/\Gamma^1$ and $\Gamma^*/\Gamma^+$ are elementary abelian $2$-groups, and we determine their ranks.

By Proposition \ref{twosided}, we have that $\calO\alpha \calO = \frakc J$ where $\nrd(J) = \fraka \parallel \frakd$.  From strong approximation, we have $\nrd(\alpha)\Z_F = \frakc^2 \fraka = a\Z_F$ with $a \in F_+^\times$.  Thus we have a homomorphism
\begin{equation} \label{NOa}
\begin{aligned}
N(\calO)_+ &\to \{\fraka \parallel \frakd : [\fraka] \in (\Cl^+ \Z_F)^2 \} \\
\alpha &\mapsto \fraka.
\end{aligned}
\end{equation}
The kernel of this map contains $\calO_+^*$.  The map is surjective, since if $\fraka \parallel \frakd$ with $[\fraka] \in (\Cl^+ \Z_F)^2$ then there exists $\frakc$ satisfying $[\frakc^2]=[\fraka^{-1}]$ in $\Cl^+ \Z_F$, and the two-sided ideal $\frakc J$ with $J=[\calO,\calO] + \fraka \calO$ has $[\nrd(\frakc J)]=[\frakc^2 \fraka]=[(1)]$ so by strong approximation there exists $\alpha \in \calO_+$ such that $\calO \alpha \calO = \frakc J$ and in particular $\alpha \in N(\calO)_+$.  

Recall $\Gamma^*=N(\calO)_+/F^\times$.  The kernel of (\ref{NOa}) is generated by $\alpha \in N(\calO)_+$ such that $\calO\alpha\calO = \frakc \calO$ with $[\frakc^2]=[(1)] \in \Cl^+ \Z_F$.  However, the image of $F^\times$ consists of all principal ideals, so the class of $\frakc$ is only well-defined in $\Cl \Z_F$, and thus belongs to $(\Cl \Z_F)[2]_+$, as claimed.
\end{proof}

\begin{rmk}
Proposition \ref{corgamma} corrects Exercise III.5.4 in Vign\'eras \cite{vigneras-book}.
\end{rmk}

In this situation, the decomposition (\ref{decomp}) has an alternate description via class numbers.  Two right fractional $\calO$-ideals $I,J$ are \defi{narrowly isomorphic} (resp.\ \defi{isomorphic}) if there exists $\alpha \in B_+^\times$ (resp.\ $\alpha \in B^\times$) such that $\alpha I = J$.  The set of narrow isomorphism classes of locally principal right fractional $\calO$-ideals is denoted $\Cl^+ \calO$ and called the \defi{narrow class set} of $\calO$; we have a canonical bijection
\[ B_+^\times \backslash \Bhat^\times / \calOhat^\times = \Cl^+(\calO). \]
We similarly define the \defi{class set} $B^\times \backslash \Bhat^\times / \calOhat^\times = \Cl(\calO)$.  Now the reduced norm (\ref{strongapprox}) gives a bijection
\[ \nrd: \Cl^+(\calO) = B_+^\times \backslash \Bhat^\times / \calOhat^\times \xrightarrow{\sim} F_+^\times \backslash \Fhat^\times / \Zhat_F^\times = \Cl^+(\Z_F) \]
to the narrow ideal class group of $F$ and similarly $\Cl(\calO) \xrightarrow{\sim} \Cl^{(+)}(\Z_F)$ is in bijection with $\Cl^{(+)}(\Z_F)$, the ideal class group of $F$ with modulus given by the set of all real places of $F$ ramified in $B$.  Note that $\Cl^+(\calO)=\Cl(\calO)$ if and only if $\Z_{F,(+)}^\times/\Z_F^{\times 2} = \Z_{F,+}^\times /\Z_F^{\times 2}$ where 
\[ \Z_{F,(+)}^\times=\{ u \in \Z_F^\times : v(u) > 0 \text{ for all real places $v$ ramified in $B$}\}. \]

Therefore, the decomposition (\ref{decomp}) in the case $\Gamma=\Gamma_0^+(\frakN)$ becomes:
\[ X_{\textup{Sh}}(\Gamma_0^+(\frakN)) = \bigsqcup_{[\frakb] \in \Cl^+(\Z_F)} X(\Gamma_0^+(\frakN)_\frakb) \]
where $\Gamma_0(\frakN)^+=\calO_{\frakb,+}/\Z_F^\times$.  A similar analysis (with another application of strong approximation) shows that $X^1$ is also a union indexed by $\Cl^+(\Z_F)$.

The orbifold associated to $\Gamma^*$ is instead described by type numbers.  For each representative $\betahat$ corresponding to an ideal $\frakb$ as in (\ref{breakup}), the order $\calO_\frakb = \betahat \calOhat \betahat^{-1} \cap B$ is again an Eichler order of level $\frakN$.  There is a bijection between the set of narrow isomorphism classes of orders everywhere locally isomorphic to $\calO$ and the set 
\[ T^+(\calO) = B_+^\times \backslash \Bhat^\times / N(\calOhat); \]
we call $T^+(\calO)$ the set of \defi{narrow types} of $\calO$ and $t^+(\calO)=\#T^+(\calO)$ the \defi{narrow type number}.  In a similar way, we define the set of \defi{types} $T(\calO)$ and the type number $t(\calO)=\#T(\calO)$.  

\begin{rmk}
The preceding results on type numbers and class numbers can also be phrased in the language of quadratic forms.  Indeed, there is a functorial, discriminant-preserving bijection between twisted similarity classes of ternary quadratic modules over $\Z_F$ and isomorphism classes of $\Z_F$-orders in quaternion algebras obtained from the even Clifford algebra \cite{Voight:charquat}.  Importing this language, we say also that the \defi{genus} of an order $\calO$ consists of all orders everywhere locally isomorphic to $\calO$, and the type set $T(\calO)$ represents the \defi{isometry classes} in the genus.
\end{rmk}

We have a surjective map from $\Cl^+(\calO) \to T^+(\calO)$ and $\Cl(\calO) \to T(\calO)$, so each isomorphism class of Eichler orders is represented at least once among the set of orders $\calO_\frakb$; strong approximation now gives a bijection 
\[ \nrd: B_+^\times \backslash \Bhat^\times / N(\calOhat) \xrightarrow{\sim} F_{+}^\times \backslash \Fhat^\times / \nrd(N(\calOhat)) = \Cl^+_{\Gamma^*}(\Z_F) \]
where $\Cl^+_{\Gamma^*}(\Z_F)$ is the quotient of $\Cl^+(\Z_F)$ by the subgroup generated by the squares $\Cl^+(\Z_F)^2$ and the set $\{[\fraka] : \fraka \parallel \frakd\}$ of classes of unitary divisors of the discriminant.  (Note the relationship of this statement to Proposition \ref{corgamma}.)  Similarly, $T(\calO)$ is in bijection the quotient of $\Cl^{(+)}(\Z_F)$ by the subgroup generated by $\Cl^{(+)}(\Z_F)^2$ and the set $\{[\fraka] : \fraka \parallel \frakd\}$.  In particular, the narrow type number and the type number are a power of $2$, and we have a decomposition
\[ X_{\textup{Sh}}(\Gamma_0^*(\frakN)) = \bigsqcup_{[\frakb] \in \Cl^+_{\Gamma^*}(\Z_F)} X(\Gamma_0^*(\frakN)_\frakb), \]
where $\Gamma_0^*(\frakN)_\frakb = N(\calO_{\frakb})_+/F^\times$ and of course $X(\Gamma_0^*(\frakN)_\frakb) = \Gamma_0^*(\frakN)_\frakb \bs \calH$.

We summarize this derivation in the following proposition.

\begin{prop} \label{classnotypeno}
Let $\calO$ be an Eichler order of discriminant $\frakd=\frakD\frakN$.  The set $\Cl^+(\calO)$  (resp.\ $\Cl(\calO)$) is in bijection with $\Cl^+(\Z_F)$ (resp.\ $\Cl^{(+)}(\Z_F)$).  

The set $T(\calO)$ of isomorphism classes of orders everywhere locally isomorphic to $\calO$ is in bijection with the group $\Cl_{\Gamma^*}^{(+)}(\Z_F)$ obtained as the quotient of $\Cl^{(+)}(\Z_F)$ by the subgroup generated by the squares and the set $\{[\fraka] : \fraka \parallel \frakd\}$ of classes of unitary divisors of the discriminant $\frakd$.

Similarly, the set $T^+(\calO)$ of narrow isomorphism classes of orders everywhere locally isomorphic to $\calO$ is in bijection with $\Cl^+_{\Gamma^*}(\Z_F)$ obtained as the quotient of $\Cl^+(\Z_F)$ by the squares and the classes of unitary divisors of the discriminant $\frakd$. 
\end{prop}

\subsection*{Volume formula}

The formula for the volume of arithmetic quotients will be a key ingredient in our proof.  The formula is due to Borel \cite[7.3]{borel-commensurability} (see also Maclachlan and Reid \cite[\S 11.1]{mac-reid-book}, Vign\'eras \cite[Corollaire IV.1.8]{vigneras-book} and, for the totally real case, Shimizu \cite[Appendix]{Shimizu}):
\begin{equation} \label{volumeformula}
\vol(X^1)=\vol(X_0^1(\frakN)_\frakb) = \frac{2(4\pi)^s}{(4\pi^2)^r(8\pi^2)^c} d_F^{3/2} \zeta_F(2) \Phi(\frakD)\Psi(\frakN) 
\end{equation}
where
\begin{equation} \label{PhiPsi}
\begin{aligned} 
\Phi(\frakD) &= \#(\Z_F/\frakD)^\times = N \frakD \prod_{\frakp \mid \frakD} \left(1-\frac{1}{N \frakp}\right) \\
\Psi(\frakN) &= [\Gamma_0^1(1) : \Gamma_0^1(\frakN)] = N\frakN \prod_{\frakp \mid \frakN} \left(1 + \frac{1}{N\frakp}\right)
\end{aligned}
\end{equation}
and the hyperbolic volume is normalized as in (\ref{metrics}).  Note that as the right-hand side of this formula is independent of the choice of Eichler order $\calO$, so every connected orbifold in the decomposition of $X_{\textup{Sh}}(\Gamma^1)$ has the same volume $\vol(X^1)$.  

The volume of an arithmetic Fuchsian group commensurable with $\Gamma^1$ can be computed using (\ref{volumeformula}), scaled by the index of the group with respect to $\Gamma^1$.  By Theorem \ref{Borel}, Proposition \ref{classnotypeno}, and the accompanying discussion, for the maximal arithmetic Fuchsian groups $\Gamma^*$ we have
\begin{equation} \label{volformnoteichler}
 \vol(X^*) = \frac{\vol(X^1)}{[\Z_{F,+}^\times:\Z_F^{\times 2}]\cdot \#\{\fraka \parallel \frakD\frakN : [\fraka] \in (\Cl^+ \Z_F)^2\} \cdot  \#(\Cl \Z_F)[2]_+}.
\end{equation}

\begin{rmk}
This expression (\ref{volformnoteichler}) is equivalent to the one given by Borel \cite[8.6]{borel-commensurability} (also given by Maclachlan and Reid \cite[Corollary 11.6.6]{mac-reid-book} and further adapted by Chinburg and Friedman \cite[\S 2]{chinburg-smallestorbifold}); we have simply factored the index $[\Gamma^*:\Gamma^1]$ in a different way which presents some advantages for our purposes (computational and conceptual).
\end{rmk}

Finally, in the case where the dimension $m=2$, so that $r=n-1$ and $s=0$, the volume (now area) is determined by discrete invariants of the Fuchsian group: by the Riemann-Hurwitz formula, we have \cite[Theorem 4.3.1]{Katok}
\begin{equation} \label{riemannhurwitz}
\vol(X(\Gamma))=2\pi\left( 2g-2 + \sum_{i=1}^k \left(1-\frac{1}{e_i}\right) + t\right)
\end{equation}
if $\Gamma$ has $k$ elliptic cycles of orders $e_i < \infty$ and $t$ parabolic cycles; and then  $\Gamma$ has \defi{signature} $(g;e_1,\ldots,e_k;t)$ and a presentation
\begin{equation} \label{presentation}
\begin{aligned}
\Gamma \cong &\la \alpha_1,\beta_1,\dots,\alpha_g,\beta_g,\ \gamma_1,\ldots,\gamma_k,\ \delta_1,\ldots,\delta_t \\
&\quad\quad \mid \gamma_1^{e_1} = \ldots = \gamma_k^{e_k} = [\alpha_1,\beta_1] \ldots [\alpha_g,\beta_g]\gamma_1 \cdots \gamma_k \delta_1 \ldots \delta_t = 1 \ra
\end{aligned}
\end{equation}
where $[\alpha_i,\beta_i]=\alpha_i\beta_i\alpha_i^{-1}\beta_i^{-1}$ is the commutator.

\section{Isospectrality and selectivity}\label{section:selectivitysection}

Continuing with the previous section, we relate the spectral properties of orbifolds obtained from quaternionic groups to the  arithmetic properties of the corresponding quaternion orders; in particular, we describe the phenomenon of selectivity.  

\subsection*{Representation equivalence}

Let $X$ be a connected, compact Riemannian manifold.  Associated to $X$ is the Laplace operator $\Delta$, defined by $\Delta(f)=-\opdiv(\grad(f))$ for $f \in L^2(X)$ a square-integrable function on $X$.  The eigenvalues of $\Delta$ on the space $L^2(X)$ form an infinite, discrete sequence of nonnegative real numbers 
\begin{equation} \label{eqn:spectrum}
0=\lambda_0 < \lambda_1 \leq \lambda_2 \leq \dots, 
\end{equation}
called the \defi{spectrum} of $X$.  We say that the Riemannian manifolds $X,X'$ are \defi{isospectral} if they have the same spectrum.

Let $G$ be a semisimple Lie group.  Then there exists a unique (up to scalar) bi-invariant Haar measure on $G$. For $g\in G$ we denote by $R_g$ the right translation operator on $\Gamma\bs G$.  Let $\Gamma \leq G$ be a discrete cocompact subgroup and let $L^2(\Gamma\bs G)$ be the space of square integrable functions on $\Gamma\bs G$ with respect to the projection of the Haar measure of $G$. We define the \defi{quasi-regular representation} $\rho_\Gamma$ of $G$ on $L^2(\Gamma\bs G)$ by $$\rho_\Gamma(g)f=f\circ R_g$$
for $g\in G$.  The quasi-regular representation $\rho_\Gamma$ is unitary.

Now let $\Gamma,\Gamma' \leq G$ be discrete cocompact subgroups.  We say that $\Gamma,\Gamma'$ are \defi{representation equivalent} if there exists a unitary isomorphism $T: L^2(\Gamma\bs G)\xrightarrow{\sim} L^2(\Gamma'\bs G)$ such that $T(\rho_{\Gamma}(g)f)=\rho_{\Gamma'}(g)T(f)$ for all $g\in G$ and $f\in L^2(\Gamma\bs G)$.

The following criterion for isospectrality was proven by DeTurck and Gordon \cite{DeturckGordon}.

\begin{theorem}[DeTurck-Gordon]\label{theorem:DeTurckGordon}
Let $G$ be a Lie group which acts on a Riemannian manifold $M$ (on the left) by isometries.  Suppose that $\Gamma,\Gamma' \leq G$ act properly discontinuously on $M$. If $\Gamma$ and $\Gamma'$ are representation equivalent in $G$, then $X=\Gamma\bs M$ and $X'=\Gamma'\bs M$ are isospectral.
\end{theorem}

\begin{rmk}
The conclusion of Theorem \ref{theorem:DeTurckGordon} can be strengthened to imply that $X$ and $X'$ are \defi{strongly isospectral}: the eigenvalue spectra of $X$ and $X'$ with respect to any natural, self-adjoint elliptic differential operator coincide, e.g., the Laplacian acting on $p$-forms.
\end{rmk}

All of the orbifolds which will be constructed in this paper will be shown to be isospectral by employing Theorem \ref{theorem:DeTurckGordon}.  As our construction is arithmetic in nature, it will be important to have an arithmetic means of proving representation equivalency; for this, we first translate representation equivalence into a statement about conjugacy classes in the group.

For $\gamma \in \Gamma$, we define the \defi{weight} of the conjugacy class $\gamma^\Gamma$ to be $w(\gamma^\Gamma) = \vol(C_\Gamma(\gamma)\bs C_G(\gamma))$ (with respect to the Haar measure), where $C_\Gamma(\gamma)$ denotes the centralizer of $\gamma$ in $\Gamma$.  Now for $g \in G$, we have
\begin{equation} \label{gGdecomp}
g^G \cap \Gamma = \bigsqcup_i \gamma_i^\Gamma
\end{equation}
as a (possibly empty) disjoint union of conjugacy classes in $\Gamma$; we call $\gamma_i^\Gamma$ a \defi{conjugacy class} of $g$ in $\Gamma$, and we accordingly define the \defi{weight} of $g^G$ over $\Gamma$ to be the multiset of weights $\{w(\gamma_i^\Gamma)\}_i$.

\begin{prop}[{Vign\'eras \cite[Corollaire 5]{vigneras-isospectral}}] \label{proposition:vignerasrepresentationequivalence}
Suppose that for all $g \in G$, the weight of $g^G$ over $\Gamma$ is equal to the weight of $g^G$ over $\Gamma'$.  Then $\Gamma$ and $\Gamma^\prime$ are representation equivalent.
\end{prop}

Proposition \ref{proposition:vignerasrepresentationequivalence} is a direct consequence of the Selberg trace formula.  

\begin{rmk} \label{rmk:repequiv}
In fact, in dimension $2$, a converse holds: the spectrum \eqref{eqn:spectrum} of a $2$-orbifold determines, and is determined by, the area, the number of elliptic points of each order, and the number of primitive closed geodesics of each length.  This result was proven by Huber \cite{HuberZur} for hyperbolic $2$-manifolds and generalized by Doyle and Rossetti \cite{DoyleRossetti0,DoyleRossetti} to non-orientable $2$-orbifolds.  Whether the converse is true in higher dimension remains open; however, from the point of view of the algebraic techniques brought to bear on this problem, we prefer to work with representation equivalence rather than Laplace isospectrality.
\end{rmk}

\subsection*{Conjugacy classes via double cosets}

We now recast the hypothesis of Proposition \ref{proposition:vignerasrepresentationequivalence} in our situation in terms of certain embedding numbers of quadratic rings into quaternion orders.  We retain the notation from Section 1.  (Compare Vign\'eras \cite[\S III.5]{vigneras-book}.)

Let $g \in B^\times$ represent a class in $\opP\!B^\times=B^\times/F^\times$; we study the conjugacy classes of $g$ in $\Gamma$.  Let $F[g]$ be the $F$-algebra generated by $g$; this algebra indeed depends only the conjugacy class of $g$ up to isomorphism.

First, if $g \in F^\times$ represents the trivial class, then $g$ is alone in its conjugacy class and
\[ w(g^\Gamma)=\{\vol(\Gamma \backslash \calH)\} \]
under a suitable normalization of the Haar measure.  So from now on suppose $g \not\in F^\times$.  Let $K=F[g]$.  

Suppose that $g^{B^\times} \cap \Gamma \neq \emptyset$.  Then without loss of generality we may assume $g \in \Gamma$.  Let $\gamma \in g^{B^\times} \cap \Gamma$.  Then $\gamma=\mu^{-1} g \mu$ with $\mu \in B^\times$.  Let
\[ E=\{\mu \in B^\times : \mu^{-1} g \mu \in \Gamma \} \]
be the set of such conjugating elements.  Then the map
\begin{equation} \label{eqn:firstbij}
\begin{aligned}
g^{B^\times} \cap \Gamma &\to K^\times \bs E / \Gamma \\
g^\Gamma &\mapsto K^\times \mu \Gamma
\end{aligned} 
\end{equation}
is well-defined and bijective by definition.  

Next, we compare this double coset with its adelization.  We consider the set
\[ \Ehat = \{ \muhat \in \Bhat^\times : \muhat^{-1} g \muhat \in \Gammahat \}. \]
Not every element of $\Ehat$ gives rise to an element of $E$; however, continuing with our assumption as in Section 1 that $B$ has at least one split infinite place and thus satisfies the Eichler condition, we will be able to characterize $E \subseteq \Ehat$ explicitly.

Let 
\[ F_{(+)}^\times=\{x \in F^\times : v(x) > 0 \text{ for all ramified real places $v$}\}. \]
The reduced norm
\[ \nrd: \Ehat/\Gammahat \to F_{(+)} \bs \Fhat^\times / \nrd(\Gammahat)=\Cl_\Gamma \]
partitions the set $\Ehat/\Gammahat$ into finitely many classes; for $[\frakb] \in \Cl_\Gamma$, write
\begin{equation} \label{eqn:Ehatb}
\Ehat_\frakb=\{\muhat \in \Ehat : [\nrd(\muhat\Gammahat)]=[\frakb] \in \Cl_\Gamma\} 
\end{equation}
so that the fibers of $\nrd$ are the cosets $\Ehat_\frakb/\Gammahat$.

\begin{lem} \label{lem:strapprox}
Let $\muhat \in \Bhat^\times$.  Then $\muhat \Gammahat =\mu \Gammahat$ for some $\mu \in B^\times$ if and only if $[\nrd(\muhat)] = 1 \in \Cl_\Gamma$, i.e., $\nrd(\muhat) \in F_{(+)}^\times\nrd(\Gammahat)$.
\end{lem}

\begin{proof}
Since $\Gammahat$ is an open subgroup of $\Bhat^\times$, by strong approximation (\ref{strongapprox}) the map
\[ \nrd:B^\times \bs \Bhat^\times / \Gammahat \xrightarrow{\sim} F_{(+)}^\times \bs \Fhat^\times / \nrd(\Gammahat)  \]
is a bijection, and so $\muhat\Gammahat  =  \mu\Gammahat$ with $\mu \in B^\times$ if and only if $\nrd(\muhat) \in F_{(+)}^\times\nrd(\Gammahat) $, as claimed.
\end{proof}

So by Lemma \ref{lem:strapprox}, the elements of $\Ehat^\times/\Gammahat$ that are representable by a class from $E$ are those that belong to the trivial fiber $\Ehat_{(1)}/\Gammahat$.

Many of the fibers are naturally in bijection, as follows.

\begin{lem} \label{lem:fibersinbij}
Suppose $\betahat \in \Khat^\times$, and let $[\nrd(\betahat\Gammahat)]=[\frakb]$.  Then there is a bijection
\begin{align*}
\Ehat_{(1)}/\Gammahat &\xrightarrow{\sim} \Ehat_\frakb \\
\muhat &\mapsto \betahat\muhat.
\end{align*}
\end{lem}

\begin{proof}
The element $\betahat \in \Khat^\times$ commutes with $g \in K$, so 
\[ (\betahat\muhat)^{-1} g (\betahat\muhat) = \muhat^{-1}g \muhat \in \Gammahat. \]
The inverse is given by left multiplication by $\betahat^{-1}$.
\end{proof}

By class field theory, the Artin map gives a surjection
\[ F_{(+)}^\times \bs \Fhat^\times / \nrd(\Khat^\times) \to \Gal(K/F). \]
Therefore, the map 
\begin{equation} \label{eq:nrdbs}
\nrd: K^\times \bs \Gammahat \Khat^\times / \Gammahat \to F_{(+)}^\times \bs \Fhat^\times / \nrd(\Gammahat)
\end{equation}
is either surjective or the image is half the size, and accordingly the fibers $\Ehat_{\frakb}$ are in bijection with $\Ehat_{(1)}$ as in Lemma \ref{lem:fibersinbij} for half or all of the fibers.

On the one hand, if $\widehat{\nu} \in \Ehat_{\frakb}$, then the map $\muhat \mapsto \widehat{\nu}\muhat$ induces a bijection
\begin{equation} \label{eqn:secondbij}
\begin{aligned}
\Ehat_{(1)}/\Gamma &\xrightarrow{\sim} \Ehat_\frakb/\Gamma \\
\muhat\Gammahat &\mapsto \widehat{\nu} \muhat\Gammahat.
\end{aligned}
\end{equation}
On the other hand, if $\Gammahat'=\betahat^{-1} \Gammahat \betahat$, and we define 
\[ \Ehat'=\{\muhat' \in B^\times : \muhat'^{-1} g \muhat' \in \Gammahat'\} \]
then the map $\muhat \mapsto \muhat\betahat$ induces a bijection
\begin{equation} \label{eqn:thirdbij}
\begin{aligned}
\Ehat/\Gamma &\xrightarrow{\sim} \Ehat'/\Gammahat' \\
\muhat\Gammahat &\mapsto \muhat\betahat\Gammahat' = \muhat\Gammahat\betahat^{-1}
\end{aligned}
\end{equation}
identifying the fibers $\Ehat_\frakc/\Gammahat$ and $\Ehat'_{\frakc\frakb}/\Gammahat'$, the latter defined in the obvious way.  

Putting together these bijections, we obtain the following result.  (Recall the definition of the genus of a group in terms of everywhere locally conjugate groups, given in Definition \ref{def:samegenus}.)

\begin{prop} \label{prop:equivifnonempty}
Let $g \in \Gamma$, let $\betahat \in \Khat^\times$, and let $\Gamma'=\betahat^{-1}\Gammahat\betahat \cap B^\times$.  Then the set $g^{B^\times} \cap \Gamma'$ is in bijection with $g^{B^\times} \cap \Gamma$ whenever it is nonempty.  
\end{prop}

\begin{proof}
To be completely explicit: we combine \eqref{eqn:firstbij} and Lemma \ref{lem:strapprox} to obtain bijections
\begin{align*}
g^{B^\times} \cap \Gamma &\xrightarrow{\sim}& K^\times \bs E^\times /\Gamma &\xrightarrow{\sim}& K^\times \bs \Ehat_{(1)}^\times / \Gammahat \\
(\mu^{-1} g \mu)^\Gamma &\mapsto& K^\times \mu \Gamma &\mapsto & K^\times \mu \Gammahat;
\end{align*}
now the bijections \eqref{eqn:secondbij}--\eqref{eqn:thirdbij} combine to give
\begin{align*}
K^\times \bs \Ehat_{(1)}^\times / \Gammahat &\xrightarrow{\sim}& K^\times \bs \Ehat_{\frakb^{-1}}^\times / \Gammahat  &\xrightarrow{\sim}& K^\times \bs \Ehat_{(1)}'^\times / \Gammahat' \\
K^\times \muhat \Gammahat &\mapsto& K^\times \widehat{\nu}\muhat\Gammahat  &\mapsto&  K^\times \widehat{\nu}\muhat \betahat \Gammahat'
\end{align*}
and one of the latter two sets is nonempty if and only the other is.  Reversing the first string of bijections for $\Gamma'$ instead, the result follows.  
\end{proof}



We summarize the preceding discussion with the following proposition.

\begin{prop} \label{prop:selectconst}
Let $\Gamma,\Gamma'$ be congruence groups in $B$ that are everywhere locally conjugate.  Let $g \in B^\times$.  Then the weight of $g^{B^\times}$ over $\Gamma$ is equal to the weight of $g^{B^\times}$ over $\Gamma'$ if and only if the following condition holds:
\[ g^{B^\times} \cap \Gamma \neq \emptyset\ \Leftrightarrow\ g^{B^\times} \cap \Gamma' \neq \emptyset. \]
\end{prop}

\begin{proof}
The statement about weights includes also a statement about volumes: but the adelic bijection identifies the two volumes as well \eqref{volumeformula}--\eqref{volformnoteichler}.
\end{proof}

Proposition \ref{prop:selectconst} motivates the following definition. 

\begin{defn}
Let $\Gamma,\Gamma' \leq B^\times/F^\times$ be congruence subgroups that are everywhere locally conjugate.  A conjugacy class $g^{B^\times}$ is \defi{selective} for $\Gamma,\Gamma'$ if 
\begin{center}
$g^{B^\times} \cap \Gamma \neq \emptyset$ and $g^{B^\times} \cap \Gamma' = \emptyset$
\end{center}
or vice versa.
\end{defn}

Combining Theorem \ref{theorem:DeTurckGordon} and Proposition \ref{prop:selectconst}, we have the following consequence.

\begin{theorem}\label{theorem:isospectraltheorem}
If no conjugacy class is selective for congruence groups $\Gamma,\Gamma'$ in the same genus, then $X(\Gamma),X(\Gamma')$ are isospectral.  
\end{theorem}

Theorem \ref{theorem:isospectraltheorem} reduces the problem of determining the isospectrality of orbifolds to ruling out selectivity, a problem we now turn to.  

\subsection*{Selectivity}

A problem of fundamental importance in the development of class field theory was the determination of the field extensions of $F$ that embed into a central simple $F$-algebra. This problem was elegantly solved by the theorem of Albert, Brauer, Hasse, and Noether, which we state only for the case of quaternion algebras where $n=2$.

\begin{theorem}[Albert--Brauer--Hasse--Noether]\label{theorem:abhn}
Let $L/F$ be a quadratic field extension. Then there exists an embedding of $F$-algebras $L\hookrightarrow B$ if and only if no prime of $F$ which ramifies in $B$ splits in $L/F$.
\end{theorem}

In our application, we require an integral refinement of this theorem.
It was first noted by Chevalley \cite{Chevalley-book} that in certain situations it was possible for $\Z_L$ to embed into some, but not all, maximal orders of $B$. 
Much later, Chinburg and Friedman \cite[Theorem 3.3]{Chinburg-Friedman} proved a generalization of Chevalley's theorem, considering arbitrary commutative, quadratic $\Z_F$-orders $R$ and broadening the class of quaternion algebras considered to those satisfying the Eichler condition. 
Their result was subsequently generalized to Eichler orders (independently) by Chan and Xu \cite{chan-xu} and Guo and Qin \cite{Guo-Qin}. (More recently, Linowitz has given a number of criteria \cite{Linowitz-selectivity} which imply that no quadratic $\Z_F$-order is selective with respect to the genus of a fixed order $\calO\subset B$; we refer to this work for further reference.)

\begin{theorem}[Chan--Xu, Guo--Qin]\label{theorem:eichlerorderselectivity}
Let $B$ be a quaternion algebra over $F$ that satisfies the Eichler condition.  Let $\calO$ be an Eichler order of level $\frakN$ and let $R \subset \calO$ be a quadratic $\Z_F$-order with field of fractions $L$.  Then every Eichler order of level $\frakN$ admits an embedding of $R$ unless both of the following conditions hold:

\begin{compactenum}
\item The algebra $B$ and the extension $L/F$ are both unramified at all finite places and ramify at exactly the same (possibly empty) set of real places; and
\item If $\frakp$ divides the discriminant $\disc(R/{\Z_L})$ and satisfies $\ord_\frakp(\frakN)\neq \ord_\frakp(\disc(R/{\Z_L}))$, then $\frakp$ splits in $L/F$.
\end{compactenum}

\noindent Furthermore, if (1) and (2) hold, then the Eichler orders of level $\frakN$ admitting an embedding of $R$ represent exactly one-half of the isomorphism classes of Eichler orders of level $\frakN$ in $B$. 
\end{theorem}

There is a further statement describing exactly when $R$ is selective for Eichler orders $\calO,\calO'$ in terms of a distance ideal; we will not make use of this refinement here.

Finally, we have further criteria for the existence of an embedding of a conjugacy class into a maximal arithmetic group.  This was provided in detailed work by Chinburg and Friedman \cite[Theorem 4.4]{Chinburg-Friedman}.  For an element $g \in B^\times$, we define $\disc(g)$ to be the discriminant of the reduced characteristic polynomial of $g$.

\begin{theorem} \label{theorem:chinburgfriedman_selectivity} 
Let $F$ be a number field and $B$ be a quaternion algebra defined over $F$ of discriminant $\frakD$. Assume that $B$ satisfies the Eichler condition. Let $\calO \subset B$ be an Eichler order of level $\frakN$ and let $\Gamma^*$ be the associated group as defined in \textup{(\ref{align:arithmeticgroups})}. Suppose that $g\in B^\times$. If a conjugate of the image $\overline{g}\in B^\times/F^\times$ of $g$ is contained in ${\Gamma^*}$ then the following three conditions hold:

\begin{enumerate}
\item[\textup{(a)}] $\disc(g)/\!\nr(g)\in\mathbb Z_F$.
\item[\textup{(b)}] If a prime $\frakp$ of $F$ appears to an odd power in the prime ideal factorization of $\nr(g)$ (that is, $g$ is \textit{odd} at $\frakp$) then $\frakp\mid \frakD\frakN$.
\item[\textup{(c)}] For each prime $\frakp$ dividing $\frakN$ at least one of the following hold:

\begin{enumerate}
\item[\textup{(i)}] $g\in F$;
\item[\textup{(ii)}] $g$ is odd at $\frakp$;
\item[\textup{(iii)}] $F[g]\otimes_F F_\frakp$ is not a field;
\item[\textup{(iv)}] $\frakp$ divides $\disc(g)/\!\nrd(g)$;
\end{enumerate}
\end{enumerate}

Conversely, if \textup{(a)}, \textup{(b)} and \textup{(c)} hold, then a conjugate of $\overline{g}$ lies in ${\Gamma^*}$ except possibly when the following three conditions hold:
\begin{enumerate}
\item[\textup{(d)}] $F[g]\subset B$ is a quadratic field extension of $F$.
\item[\textup{(e)}] The extension $F[g]/F$ and the algebra $B$ are unramified at all finite places and ramify at exactly the same (possibly empty) set of real places of k. Furthermore, all $\frakp\mid\frakN$ split in $F[g]/F$.
\item[\textup{(f)}] All primes dividing $\disc(g)/\!\nrd(g)$ split in $F[g]/F$.
\end{enumerate}

\noindent  Suppose now that \textup{(a)}--\textup{(f)} hold. In this case the type set of $\Gamma^*$ contains an even number of elements, exactly half of which contain a conjugate of $\overline{g}$.

\end{theorem}

\subsection*{Isometries}

To conclude this section, we prove two propositions that characterize isometry and isospectrality in this context, complementing Theorem \ref{theorem:isospectraltheorem}.  

\begin{prop} \label{prop:isospectralmeanslevel}
Let $\Gamma \leq B^\times/F^\times$ and $\Gamma^\prime \leq B^{\prime\times}/F^{\prime\times}$, and suppose that $X(\Gamma),X(\Gamma^\prime)$ are representation equivalent and $\dim(X(\Gamma))=\dim(X(\Gamma^\prime))=2,3$.  Then there exists a $\Q$-algebra isomorphism $\tau: B \xrightarrow{\sim} B^\prime$.
Moreover, if $\Gamma,\Gamma'$ each contain the units of norm $1$ in an Eichler order, then $\tau(\Gammahat)\cong \Gammahat^\prime$.  
\end{prop}

\begin{proof}
We will first show that $X(\Gamma)$ and $X(\Gamma^\prime)$ are commensurable: that is, we will show that $X(\Gamma)$ and $X(\Gamma^\prime)$ admit a common finite-sheeted covering.  Equivalently, we will show that $\Gamma$ and $\Gamma^\prime$ are commensurable as groups, i.e., $\Gamma\cap\Gamma^\prime$ has finite index inside both $\Gamma$ and $\Gamma^\prime$.  Upon proving this it will follow that $B$ and $B'$ are isomorphic as $\Q$-algebras: see Maclachlan and Reid \cite[Theorem 8.4.1, Theorem 8.4.6]{mac-reid-book} (the proof generalizes immediately to this context).

To show that $X(\Gamma)$ and $X(\Gamma^\prime)$ are commensurable, we will first show that they have the same rational length spectra. Let $\mathcal L(X(\Gamma)) \subseteq \R$ denote the set of lengths of closed geodesics of $X(\Gamma)$ (considered without multiplicity), and define the \defi{rational length spectrum} $\Q \calL(X(\Gamma))$ of $X(\Gamma)$ by 
\[ \Q \calL(X(\Gamma)) = \{r\ell : r\in\Q,\ \ell\in \mathcal L(X(\Gamma))\}.\] 
We claim that
\begin{equation} \label{eqn:samelength}
\Q\calL(X(\Gamma))=\Q\calL(X(\Gamma^\prime)). 
\end{equation}
By Theorem \ref{theorem:DeTurckGordon}, since $X(\Gamma)$ and $X(\Gamma^\prime)$ are representation equivalent, they are isospectral with respect to the Laplace operator.  When $\dim(X(\Gamma))=\dim(X(\Gamma^\prime))=2$, it follows that $X(\Gamma)$ and $X(\Gamma^\prime)$ are length isospectral by work of Dryden and Strohmaier \cite{Dryden} (see also Doyle and Rossetti \cite{DoyleRossetti0,DoyleRossetti}), so in particular $\mathcal L(X(\Gamma))=\calL(X(\Gamma^\prime))$.  When $\dim(X(\Gamma))=\dim(X(\Gamma^\prime))=3$, it follows from the trace formula of Duistermaat and Guillemin \cite{DuistermaatGuillemin} (see also Prasad and Rapinchuk \cite[Theorem 10.1]{Prasad-Rapinchuk}, and in the orbifold case Elstrodt, Grunewald, and Mennicke \cite[p. 203]{EGM}) that $X(\Gamma)$ and $X(\Gamma^\prime)$ have the same sets of lengths of closed geodesics, considered without multiplicity.  Therefore \eqref{eqn:samelength} holds in either case. 

Next, because the rational length spectrum of an orbifold coincides with the rational length spectrum of a finite degree manifold cover, our claim that $X(\Gamma)$ and $X(\Gamma^\prime)$ are commensurable follows from the fact that the commensurability class of a hyperbolic $2$- or $3$-manifold is determined by its rational length spectrum: this was proven by Prasad and Rapinchuk \cite{Prasad-Rapinchuk} for arithmetic hyperbolic manifolds of even dimension and by Chinburg, Hamilton, Long, and Reid \cite{chinburg-geodesics} for arithmetic hyperbolic $3$-manifolds.  Therefore, \eqref{eqn:samelength} implies that $X(\Gamma)$ and $X(\Gamma^\prime)$ are indeed commensurable.  

Now we prove the second statement.  From the first statement, we may suppose without loss of generality that $B=B^\prime$; we will show then that $\Gammahat \cong \Gammahat^\prime$.  We first identify the governing Eichler orders---we do this for $\Gamma$, with the same argument for $\Gamma^\prime$ applying primes.  Let $\calO$ be the sub-$\Z_F$-algebra of $B$ generated by the elements of $\Gamma$ of reduced norm $1$.  
An element $g \in \Gamma$ with $\nrd(g)=1$ is necessarily integral (i.e., has integral reduced trace) as $\Gamma$ is discrete, so every element of $\calO$ is integral and hence $\calO$ is a $\Z_F$-order.  By hypothesis, $\calO$ contains an Eichler order, and therefore $\calO$ itself is an Eichler order (any order that contains an Eichler order is also Eichler), and it is therefore the largest order $\calO$ such that the group of norm $1$ units is contained in $\Gamma$.  Let $\frakN$ be the level of $\calO$, and let $\Gamma^1$ be the subgroup of elements of reduced norm $1$.  Similarly define $\calO^\prime$, $\frakN^\prime$, and $\Gamma^{1\prime}$.  

Because $X(\Gamma)$ and $X(\Gamma^{\prime})$ are representation equivalent, the work of B{\'e}rard \cite{berard} (see also Gordon and Mao \cite[Lemma 1.1]{GordonMao}) shows that $X(\Gamma^1)$ and $X(\Gamma^{1\prime})$ are representation equivalent: we have restricted in both cases to precisely those conjugacy classes with determinant $1$.  That $X(\Gamma^1)$ and $X(\Gamma^{1\prime})$ are Laplace isospectral now follows from Theorem \ref{theorem:DeTurckGordon}.

We now show that in fact $\Gammahat^1 \cong \Gammahat^{1\prime}$, which is the same as showing that the Eichler orders $\calO,\calO'$ have the same level $\frakN=\frakN'$, as any two Eichler orders of the same level are everywhere locally conjugate.  Suppose that there exists a prime $\frakp$ such that $e=\ord_\frakp \frakN \neq \ord_\frakp \frakN' = e'$.  By local theory, a quadratic order $R=\Z_F[g] \supsetneq \Z_F$ (with $N(g)=1$) embeds in the completion $\calO_\frakp$ if and only if the minimal polynomial of $g$ has a root modulo $\frakp^e$.  Since $e \neq e'$, by local approximation there exists a quadratic order $R$ such that $S$ embeds in $\calO$ but $S$ does not embed in $\calO$; we may further suppose that there is no selectivity for $R$ by Theorem \ref{theorem:eichlerorderselectivity}, asking that $R$ is ramified at a finite place.  But if we now consider the associated conjugacy class of $g$, since $X(\Gamma),X(\Gamma')$ are Laplace isospectral we have a contradiction, as the conjugacy class of $g$ meets $\Gamma$ but not $\Gamma'$.

To conclude, we now know that $\Gammahat^1 \cong \Gammahat^{1\prime}$.  If $\Gamma,\Gamma'$ are maximal groups, then necessarily $\Gammahat \cong \Gammahat^{\prime}$ because adelic maximal groups are uniquely determined up to isomorphism by their level, by Borel (Theorem \ref{Borel}).  If instead $\Gamma,\Gamma'$ are unitive, then it suffices to show that a square class $u \in R^\times/R^{\times 2}$ occurs in $\nrd(\Gamma)$ if and only if it occurs in $\nrd(\Gamma^{\prime})$, and this follows as the determinant can be read off of the conjugacy class.
\end{proof}

\begin{rmk}
To prove the more general statement (without the hypothesis on the dimension), we would need to know an analogous statement that the rational length spectrum determines the commensurability class.  To our knowledge, this is open already in the case when $\Gamma \leq G=\PGL_2(\R) \times \PSL_2(\C)$.
\end{rmk}

\begin{prop} \label{prop:isometrytheorem}
Let $\Gamma \leq B^\times/F^\times$ and $\Gamma' \leq B'^\times/F'^\times$.  Let $\iota:B^\times/F^\times \hookrightarrow G$ and $\iota':B'^\times/F'^\times \hookrightarrow G$ be embeddings with $G=\PGL_2(\R)^s \times \PSL_2(\C)^c$.  Then $X(\Gamma),X(\Gamma')$ are isometric if and only if there exists a permutation $\sigma$ of the factors of $G$ and $\nu \in G$ such that $\sigma(\Gamma)=\nu \Gamma' \nu^{-1}$ and a $\Q$-algebra isomorphism $\tau:B \to B'$ such that the diagram
\[
\xymatrix{
B^\times \ar[d]^{\iota} \ar[rr]^\tau & & B'^\times \ar[d]^{\iota'} \\
G \ar[r]^{\sigma} & G \ar[r]^{\nu} & G 
} \]
commutes, where $\nu$ acts on $G$ by conjugation.  If further $F=F'$ and $B=B'$ and $\iota=\iota'$, then we may take $\nu \in \iota(B^\times)$.
\end{prop}

\begin{proof}
First, we argue that we may suppose without loss of generality that $B=B'$ and $F=F'$ by applying a $\Q$-algebra isomorphism.  In one direction, such an isomorphism is supplied, and in the other direction when $X(\Gamma),X(\Gamma')$ are isometric, then they are isospectral and we may apply the first statement in Proposition \ref{prop:isospectralmeanslevel}.  Under this assumption, the two embeddings $\iota,\iota'$ are conjugate by the Skolem--Noether theorem, so we may assume that $\iota=\iota'$.  

Now suppose that $X(\Gamma),X(\Gamma^\prime)$ are isometric.  This isometry preserves the product decomposition of $\calH$ up to permutation of factors, so after applying such a $\sigma$ we may assume that it does so; then we have $\Gamma=\nu \Gamma^\prime \nu^{-1}$ for some $\nu \in \iota(B^\times)$ again by the Skolem--Noether theorem, and the result is proved. 
\end{proof}

\section{Isospectral but not isometric hyperbolic 2-orbifolds: examples}\label{section:orbifoldsearch}

In this section, we exhibit isospectral but not isometric hyperbolic $2$-orbifolds with volume $23\pi/6$ and signature $(0;2^5,3,4)$.  As explained in the introduction, these orbifolds have covering groups which are maximal congruence Fuchsian groups and do not arise from Sunada's method.  Our computations are performed in \textsc{Magma} \cite{Magma}.

\subsection*{Pair 1}

Let $F=\Q(\sqrt{34})$ be the real quadratic field with discriminant $136$ and ring of integers $\Z_F=\Z[\sqrt{34}]$.  The fundamental unit $u=6\sqrt{34}+35$ of $\Z_F$ is totally positive, with norm $N(u)=1$.  The class group of $F$ is $\Cl \Z_F \cong \Z/2\Z$ and the narrow class group of $F$ is $\Cl^+ \Z_F \cong \Z/4\Z$; the nontrivial class in $\Cl \Z_F$ is represented by the ideal $\frakb=(3, \sqrt{34}+2)$ with $N\frakb=3$ and $\frakb^2=(\sqrt{34}+5)$, and the group $\Cl^+ \Z_F$ is generated by $\frakb$, since $\sqrt{34}+5$ has norm $N(\sqrt{34}+5)=-9<0$ and hence $\frakb$ is not generated by a totally positive element.

Let $B$ be the quaternion algebra over $F$ which is ramified at the unique prime $\frakp=(2,\sqrt{34})=(\sqrt{34}+6)$ of norm $2$, one of the two real places, and no other places.  Each choice of split real place will yield a pair of isospectral orbifolds, but it will turn out that these two choices yield the same pair of orbifolds up to isometry.  So we take the split real place to correspond to the embedding with positive square root $\sqrt{34} \mapsto 5.8309\ldots$.  Then $B$ has discriminant $\frakD=\frakp$ and we take $B=\quat{-1,\sqrt{34}-5}{F}$ (note that $\sqrt{34}-5>0$ but $-\sqrt{34}-5<0$), so that $B$ is the $F$-algebra generated by $i,j$ satisfying
\[ i^2 = -1, \quad j^2 = \sqrt{34}-5, \quad ji=-ij \]
as in (\ref{quateq}).
We embed
\begin{equation} \label{Bm2R}
\begin{aligned}
B &\hookrightarrow \M_2(\R) \\
i,j &\mapsto \begin{pmatrix} 0 & 1 \\ -1 & 0 \end{pmatrix}, 
\begin{pmatrix} \sqrt{\sqrt{34}-5} & 0 \\ 0 & -\sqrt{\sqrt{34}-5} \end{pmatrix}.
\end{aligned}
\end{equation}

We compute \cite{Voight-maxorder} a maximal order $\calO$ of $B$ as
\[ \calO=\Z_F \oplus \Z_F (\sqrt{34}+6)\frac{1+i}{2} \oplus \frakb \frac{j}{3} \oplus \frakb \frac{(\sqrt{34} + 1) + 3i + (\sqrt{34} + 1)j + ij}{6}. \]
The order $\calO$ has discriminant $\frakd=\frakp$.  The class group with modulus equal to the ramified real place of $B$ is $\Cl^{(+)} \Z_F = \Cl \Z_F \cong \Z/2\Z$; therefore by Proposition \ref{classnotypeno}, the class set $\Cl \calO$ is of cardinality $2$, with nontrivial class represented by a right $\calO$-ideal $I$ of reduced norm $\frakb=\frakp$.  We see also from Proposition \ref{classnotypeno} that the type number of $\calO$ is equal to $t(\calO)=2$; the isomorphism classes of orders are represented by $\calO_1=\calO$ and the left order $\calO_L(I)=\calO_2$ of $I$, with
\[ \calO_2 = \Z_F \oplus \frakb (\sqrt{34}+4)\frac{1+i}{6} \oplus \frakb \frac{j}{3} \oplus \Z_F (\sqrt{34}+5) \frac{ (\sqrt{34} - 5) - 3i + (\sqrt{34} + 1)j + ij}{18}. \]

As in (\ref{align:arithmeticgroups}), let $\Gamma_i^1=(\calO_i)_1^{\times}/\Z_F^\times$ for $i=1,2$ and similarly $\Gamma_i^+$ and $\Gamma_i^*$.  From equation (\ref{volumeformula}), we have 
\[ \vol(X(\Gamma_i^1)) = \frac{8\pi d_F^{3/2}\zeta_F(2)}{\left(4\pi^2\right)^2}(2-1) = \frac{46\pi}{3}. \]
From Lemma \ref{lem:gamma1*}, we have $\Gamma_i^1 < \Gamma_i^+$, with index 2: indeed, since the unit $u$ is totally positive we have $\Gamma_i^+/\Gamma_i^1 \cong \Z_{F,+}^\times/\Z_F^{\times 2} \cong \Z/2\Z$.  We compute \cite{Voight-shim} that the groups $\Gamma_i^1$ have signature $(3; 2^2, 3^4)$, and the groups $\Gamma_i^+$ have signature $(1;2^2,3^2,4^2)$ and area $23\pi/3$.  

Now we consider the maximal groups $\Gamma_i^*$.  Referring to Proposition \ref{corgamma}, we have that $\frakp = \frakd=(\sqrt{34}+6)$ is of the principal narrow class and $(\Cl \Z_F)[2]_+$ is trivial (the nontrivial class lifts to an element of order 4), so $\Gamma_i^+ < \Gamma_i^*$ with index $2$, and $\Gamma_i^*/\Gamma_i^+$ is generated by an element $\alpha_i \in N(\calO_i)_+$ such that $N(\alpha_i)=\sqrt{34}+6$: we find
\begin{align*}
\alpha_1 &= \frac{\sqrt{34}+6}{2}(-1+i+j+ij) \\
\alpha_2 &= \frac{\sqrt{34}+5}{9}(-3i+ij).
\end{align*}
We compute that $\Gamma_i^*$ has signature $(0;2^5,3,4)$ and area $23\pi/6$.  By Theorem \ref{Borel}, the groups $\Gamma_i^*$ are maximal arithmetic Fuchsian groups.  We have a decomposition (\ref{decomp})
\[ X(\Gamma^*)= B_+^\times \backslash (\calH \times \Bhat^\times / \Gammahat^*) = X(\Gamma^*_1) \sqcup X(\Gamma^*_2). \]

\begin{figure}
\begin{equation} \label{pic1} \notag
\includegraphics[scale=0.77]{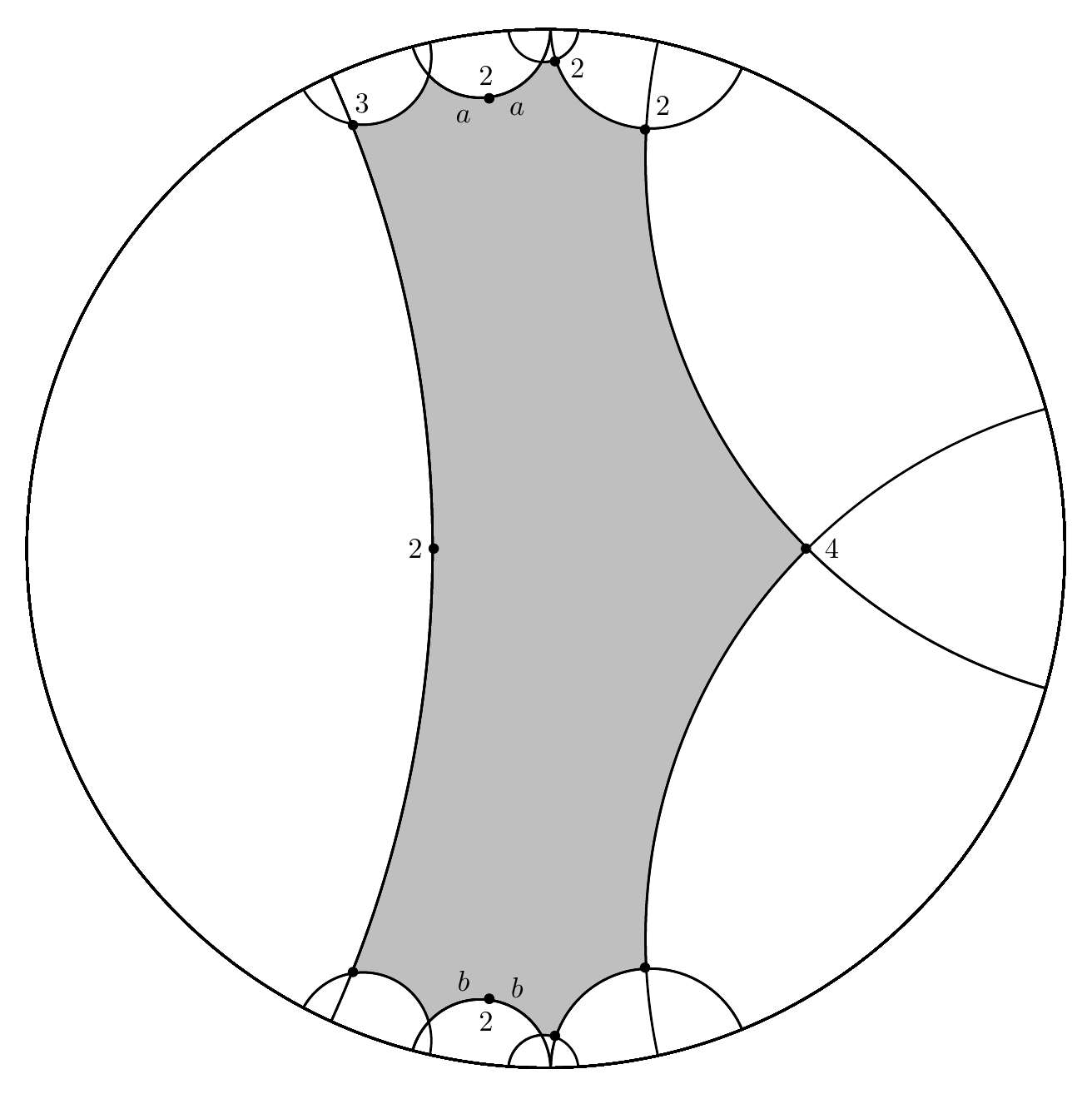} 
\end{equation}
\begin{equation*} 
\includegraphics[scale=0.77]{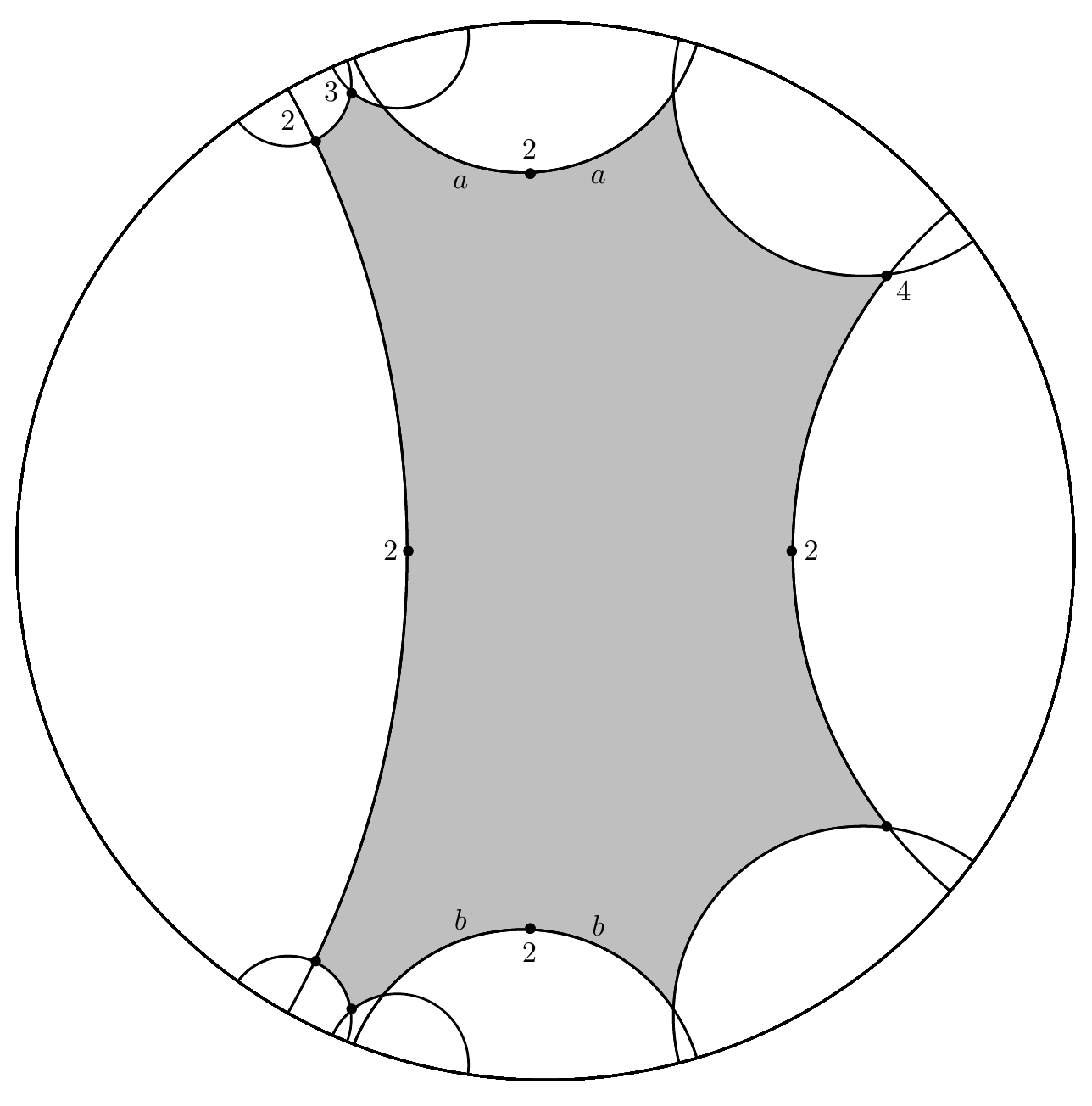} 
\end{equation*}
\centering
\textbf{Figure \ref{pic1}}: Fundamental domains for $(0;2^5,3,4)$-orbifolds $X(\Gamma_1^*)$,  $X(\Gamma_2^*)$ \\ over $F=\Q(\sqrt{34})$
\addtocounter{equation}{1}
\end{figure}

We compute  fundamental domains \cite{voight-funddom} for these groups in the unit disc, giving explicit presentations for the group and verifying the signature: they are given in Figure \ref{pic1}.  In these figures:
\begin{itemize}
\item a representative of each elliptic cycle is numbered with the size of its stabilizer group;
\item an unlabeled side is paired together with its image under complex conjugation;
\item labeled sides are paired together with labels $a,b,...$.
\end{itemize}
These domains are computed as Dirichlet domains with centers $p_1=\sqrt{-1}/3$ and $p_2=\sqrt{-1}/2$, respectively, corresponding to the origin in the unit disc.

By Theorem \ref{theorem:chinburgfriedman_selectivity}, the groups $\Gamma_i$ are not selective, since $B$ is ramified at $\frakp$.  We conclude using Theorem \ref{theorem:isospectraltheorem} that the hyperbolic $2$-orbifolds $X(\Gamma_1^*)$ and $X(\Gamma_2^*)$ are isospectral but not isometric.  

Finally, the same construction works with the other choice of split real place ($\sqrt{34} \mapsto -\sqrt{34}$).  However, if we apply the nontrivial Galois automorphism $\sigma:F \to F$ to the data above, then we find a pair of fundamental domains that are identical to the ones in Figure \ref{pic1}.  This is explained by the theory of canonical models of Shimura curves, due to Shimura \cite{Shimura} and Deligne \cite{Deligne}: the (disconnected) curve $X(\Gamma^*)$ is defined over $F$ (the two individual components are defined and conjugate over the Hilbert class field $H=F(\sqrt{u})$), but since the quaternion algebra $B$ has discriminant $\frakp$ and $\sigma(\frakp)=\frakp$, the curve $X(\Gamma^*)$ descends to $\Q$.  See work of Doi and Naganuma \cite{DoiNaganuma} for more discussion of this point.  Put another way, the commensurability class of an arithmetic Fuchsian group is determined by the isomorphism class of the quaternion algebra $B$ as an algebra over $\Q$, as in the proof of Proposition \ref{prop:isospectralmeanslevel} (see also Takeuchi \cite[Proposition 1]{Takeuchi}, and Maclachlan and Reid \cite[Theorem 8.4.7]{mac-reid-book}): in this case, the nontrivial Galois automorphism of $F$ identifies the two algebras, so the two groups are commensurable and consequently isomorphic.

\subsection*{Pairs 2 and 3}

Our second and third pairs come from the same Galois orbit, as follow.  

Let $F$ be the totally real quartic field of discriminant $21200=2^4 5^2 53^1$.  Then $F=\Q(w)$ with $f(w)=w^4-6w^2-2w+2=0$ and $\Z_F=\Z[w]$.  The Galois group of $F$ is $S_4$.  The class group of $F$ is trivial and the narrow class group of $F$ is $\Cl^+ \Z_F \cong \Z/2\Z$ with nontrivial class represented by $\frakb=(w-1)$ of norm $N\frakb=5$. The unit group of $\Z_F$ is generated by $\la -1, w+1, w^2+2w-1, 3w^2+w-1 \ra$ and $\Z_{F,+}^\times /\Z_F^{\times 2} \cong \Z/2\Z$, generated by $u=3w^2+w-1$.  

Let $B$ be the quaternion algebra over $F$ which is ramified at the unique prime $\frakp=(w)$ of norm $2$, and three of the four real places.  Specifically, we order the roots of $f$ as
\[ -2.1542\ldots, -0.8302\ldots, 0.4394\ldots, 2.5450\ldots, \]
and we label the real places of $F$ accordingly.  We take the split real place correspond to either the third or the fourth root, corresponding to the split places numbered $3$ or $4$.  We will show calculations for the third root; the other calculation follows similarly.  Then $B=\quat{-1,b}{F}$ where $b=w^3-w^2-5w+2$: we have $N(b)=-2$ and $b$ is positive in the split embedding and negative in the ramified embeddings.  We take an embedding $B \hookrightarrow \R$ analogous to (\ref{Bm2R}).

We compute a maximal order as
\begin{align*} 
\calO_1 = \calO = \Z_F &\oplus \Z_F \frac{(w^2+2w+2)(1+i)}{2} 
\oplus \Z_F \frac{(w^2+2w+2)( (w^2+2w+2) + w^2i + 2j)}{4} \\
&\oplus \Z_F \frac{(w^2+w+1)+(w+1)i+j+ij}{2}. 
\end{align*}
The class group with modulus equal to the product of the three ramified real places of $B$ is $\Cl^{(+)} \Z_F =\Cl^+\Z_F \cong \Z/2\Z$; this is true for split places $3$ and $4$ above, but not for the choices $1$ and $2$ (where the corresponding class group is trivial).  The second isomorphism class of orders is represented by
\begin{align*}
\calO_2 = \Z_F &\oplus \Z_F \frac{(w^2-4)(1+i)}{2} 
\oplus \Z_F \frac{(w^2+2w+2)( (w^2+2w+2) + (w^2-4)i + 2j)}{4} \\
&\oplus \Z_F \frac{(2w^3-w^2-9w+3)((w^2+w-1)+(w-13)i+5j+ij)}{10}.
\end{align*}

We compute again that $\vol(X(\Gamma_i))=46\pi/3$.  We find the containment $\Gamma_i^1 < \Gamma_i^+$ with index $2$, with signatures $(3;2^2,3^4)$ and $(1;2^2,3^2,4^2)$ as before.  Referring again to Proposition \ref{corgamma}, we have $\frakp=\frakd=(w)=(-w^3+11w^2+5w-4)$ in the narrow principal class and $(\Cl \Z_F)[2]_+$ is trivial, so we have $\Gamma_i^+ < \Gamma_i^*$ with index $2$ and the quotients are generated respectively by 
\begin{align*}
\alpha_1 &= w^2i + (-w^2 + 1)ij \\
\alpha_2 &= \frac{w^3 - w^2 - w}{2} + \frac{w^2 + w}{2}(i+j-ij)
\end{align*}
We compute that $\Gamma_i^*$ has signature $(0;2^5,3,4)$ and $\area(X(\Gamma_i^*))=23\pi/6$.  We compute a fundamental domain for these groups in Figure \ref{pic5} with centers $p_1=\sqrt{-1}/5$, $p_2=12\sqrt{-1}/11$  (and Figure \ref{pic6} for the other choice of split real place, with $p_1=2\sqrt{-1}/15$ and $p_2=\sqrt{-1}/2$).

\begin{figure}
\begin{equation} \label{pic5} \notag
\includegraphics[scale=0.77]{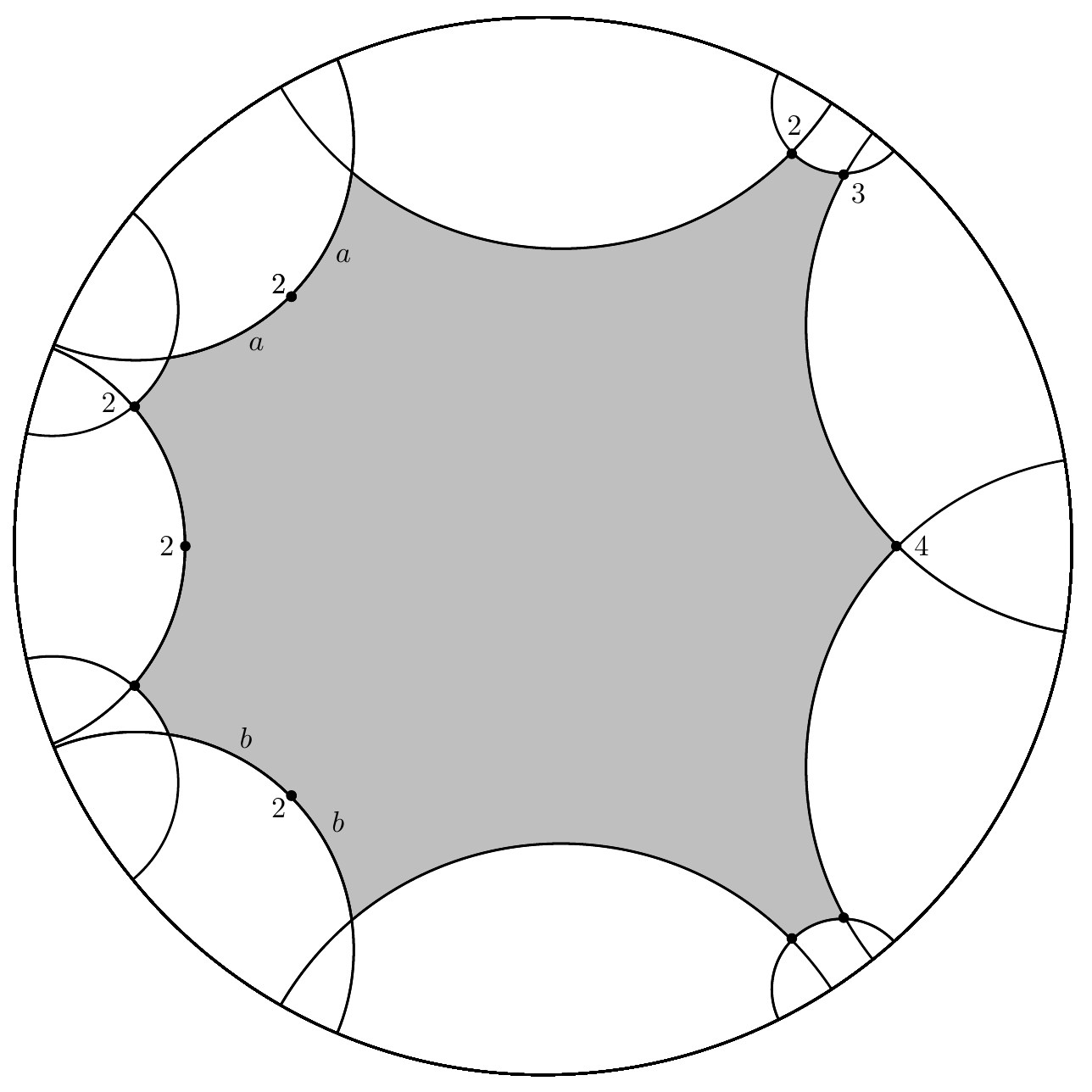} 
\end{equation}
\begin{equation*} 
\includegraphics[scale=0.77]{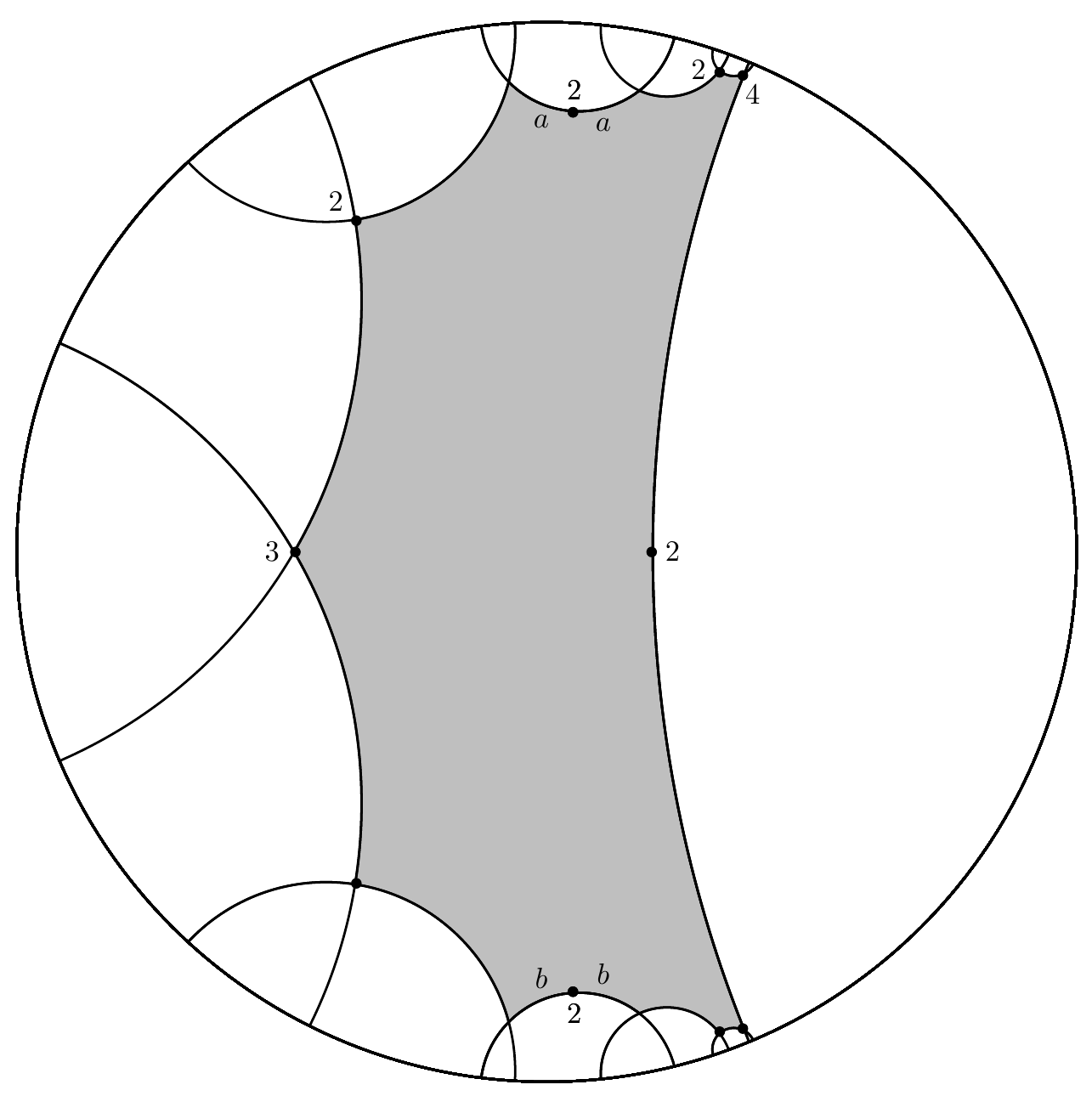} 
\end{equation*}
\centering
\textbf{Figure \ref{pic5}}: Fundamental domains for $(0;2^5,3,4)$-orbifolds $X(\Gamma_1^*)$,  $X(\Gamma_2^*)$ \\ over quartic $F$ (split place $3$)
\addtocounter{equation}{1}
\end{figure}

\begin{figure}
\begin{equation} \label{pic6} \notag
\includegraphics[scale=0.77]{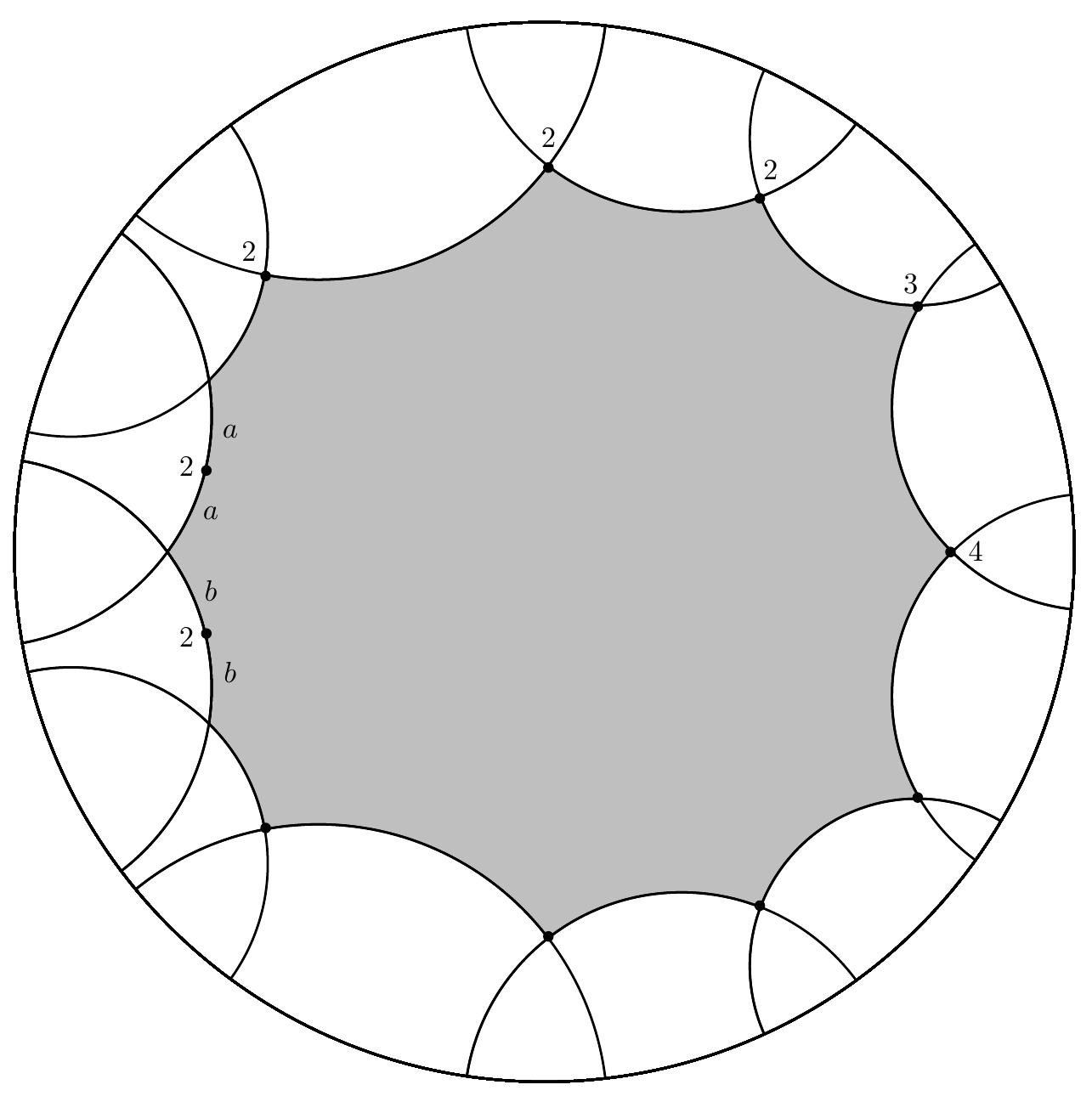} 
\end{equation}
\begin{equation*} 
\includegraphics[scale=0.77]{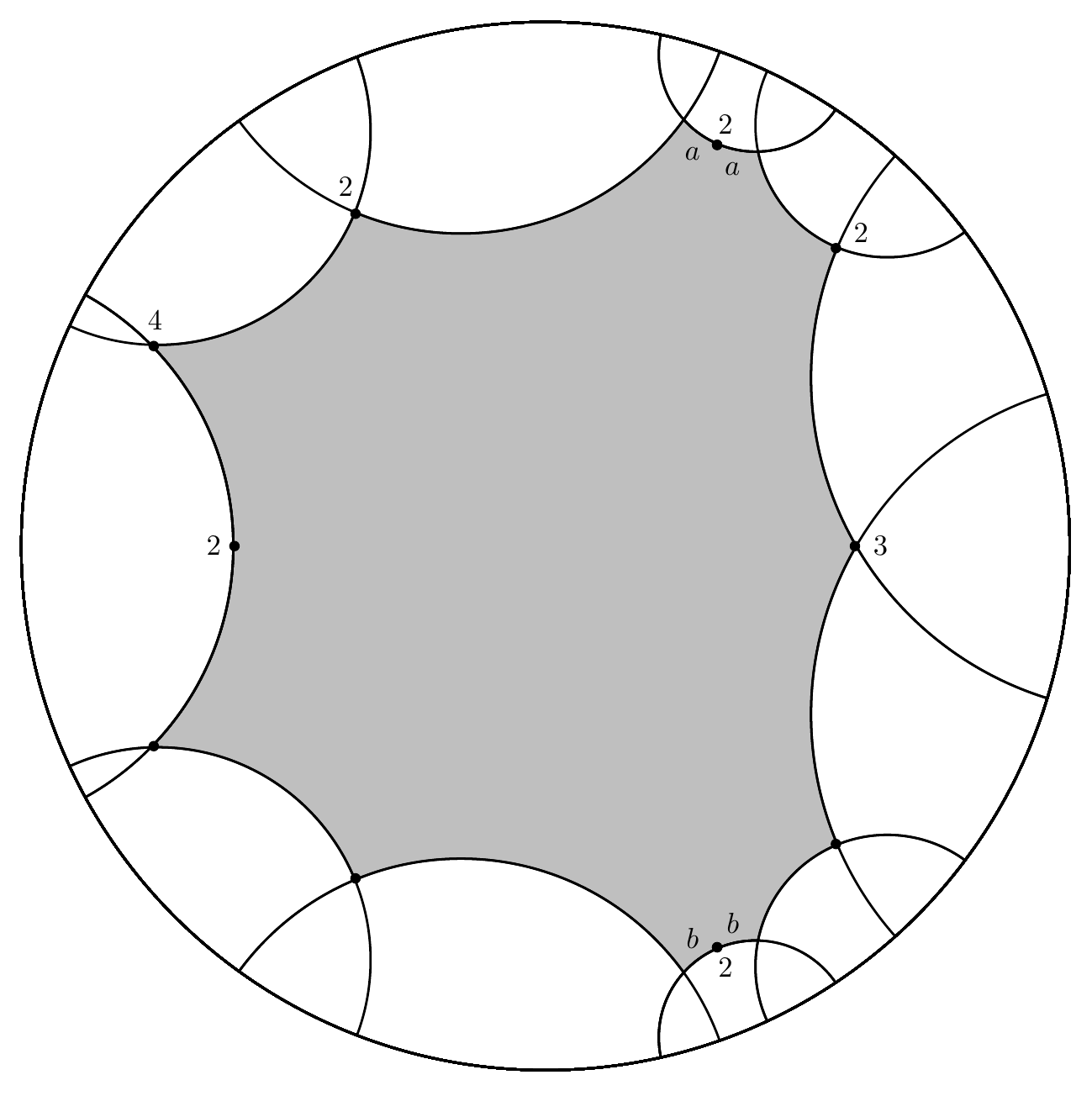} 
\end{equation*}
\centering
\textbf{Figure \ref{pic6}}: Fundamental domains for $(0;2^5,3,4)$-orbifolds $X(\Gamma_1^*)$,  $X(\Gamma_2^*)$ \\ over quartic $F$ (split place $4$)
\addtocounter{equation}{1}
\end{figure}

As with the first pair, by Theorem \ref{theorem:chinburgfriedman_selectivity}, the groups $\Gamma_i$ are not selective, since $B$ is ramified at $\frakp$.  We conclude using Theorem \ref{theorem:isospectraltheorem} that the hyperbolic $2$-orbifolds $X(\Gamma_1^*)$ and $X(\Gamma_2^*)$ are isospectral but not isometric.  

Here, there is no Galois descent of the curves: the quaternion algebras with split real places $3$ and $4$ over $\Q$ are not isomorphic.  Indeed, if they were, then a $\Q$-isomorphism would have to preserve the center $F$ and hence would be a Galois automorphism; but $F$ is not Galois, so this map is the identity, so the isomorphism is an $F$-isomorphism, and  this contradicts the fact that the algebras have distinct ramification sets at infinity.

\subsection*{A near-miss pair with slightly larger area}

We conclude this section by presenting an example with slightly larger area but a normalizer coming from the group $(\Cl \Z_F[2])_+$. 

\begin{exm}
Let $F=\mathbb Q(\sqrt{51})$ be the real quadratic field of discriminant $204$ and with ring of integers $\Z_F=\Z[\sqrt{51}]$.  The class number of $F$ is $\#\Cl \Z_F = 2$ with nontrivial ideal class represented by $\frakb=(3,\sqrt{51})$ with $\frakb^2=(3)$; we have $\Cl^+ \Z_F \cong \Z/2\Z \oplus \Z/2\Z$, generated by $\frakb$ and $\frakc=(\sqrt{51})$.  The fundamental unit $u=7\sqrt{51}+50$ is totally positive, with norm $N(u)=1$.  

Let $B$ be the quaternion algebra over $F$ ramified only at the prime $\frakp=(2,\sqrt{51}+1)=(\sqrt{51}-7)$ and the split real place corresponding to $\sqrt{51} \mapsto 7.141428\ldots$.  We take $B= \quat{-1,\sqrt{51}-7}{F}$.  We find a maximal order $\calO$ of $B$ as simply
\[ \calO_1=\calO = \Z_F \oplus \Z_F \frac{\sqrt{51}+i}{2} \oplus \Z_F j \oplus \Z_F \frac{\sqrt{51}j+ij}{2}. \]
The class group with modulus equal to the ramified real place of $B$ is equal to $\Cl^{(+)} \Z_F=\Cl \Z_F \cong \Z/2\Z$, so we again have type number $2$.  

We compute that $\Gamma_i^1$ has signature $(3;2^2,3^{12})$ and volume $26\pi$, that $\Gamma_i^+$ has index $2$ and signature $(1;2^2,3^6,4^2)$.  The groups $\Gamma_i^*$ have further index $2$: now, $\frakp$ does not belong to the narrow principal class, and so is not a square, but $\frakb$ lifts to an element of order $2$ in $\Cl^+ \Z_F$ so yields an element of $(\Cl \Z_F)[2]_+$: we find e.g.\ the normalizing element
\[ \alpha_1 = \frac{w-6}{2} + \frac{-2w+15}{2} i + (-3w+21)ij \]
of norm $3$.  We compute fundamental domains for these groups in Figure \ref{pic3}: they have signature $(0;2^3,3^2,4,6^2)$ and area $13\pi/2 > 23\pi/6$.  If the ideal $\frakp$ had belonged to the principal narrow class, then this would have given a pair with smaller area $13\pi/4 < 23\pi/6$.  

\begin{figure}
\begin{equation} \label{pic3} \notag
\includegraphics[scale=0.77]{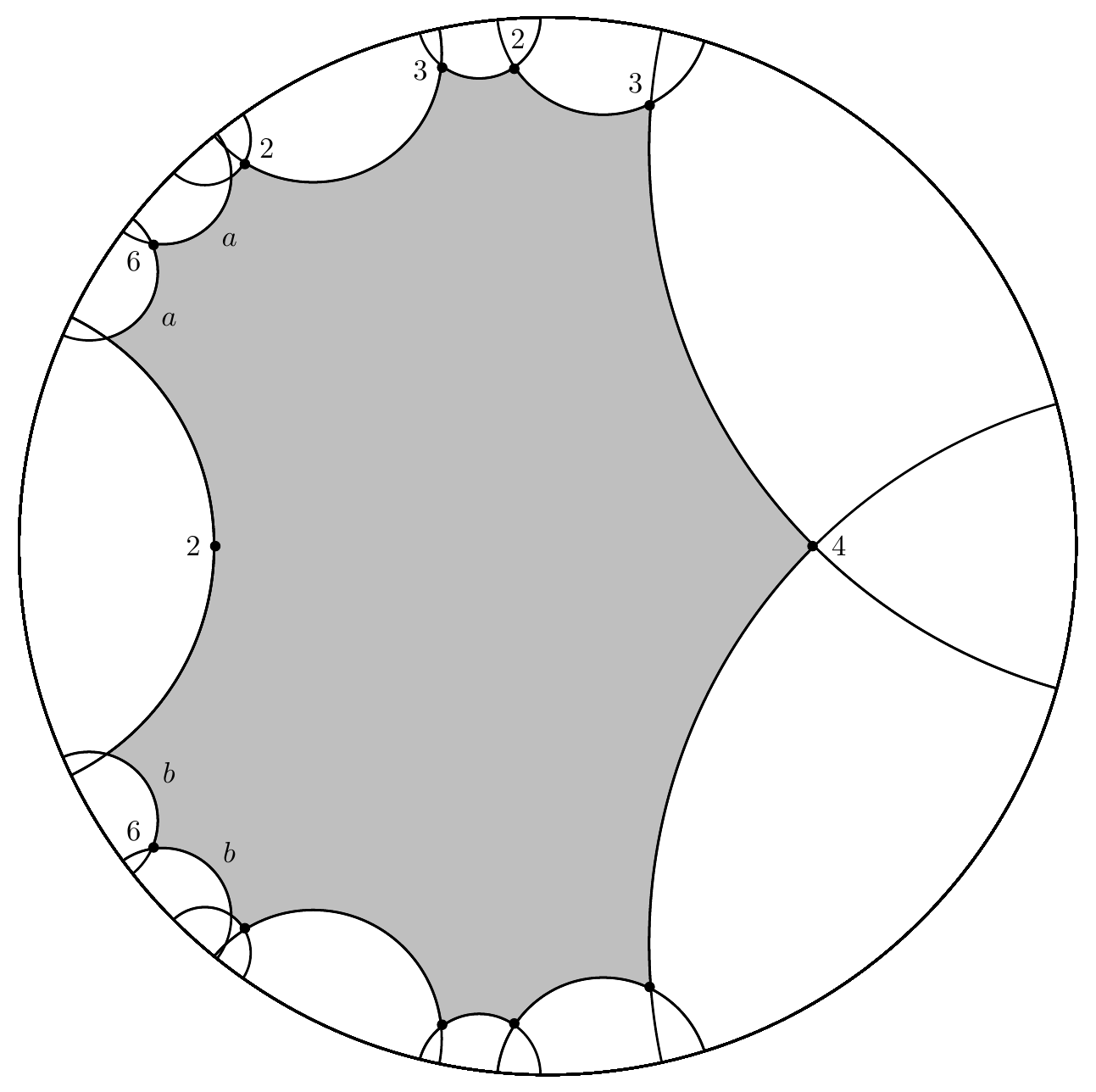} 
\end{equation}
\begin{equation*} 
\includegraphics[scale=0.77]{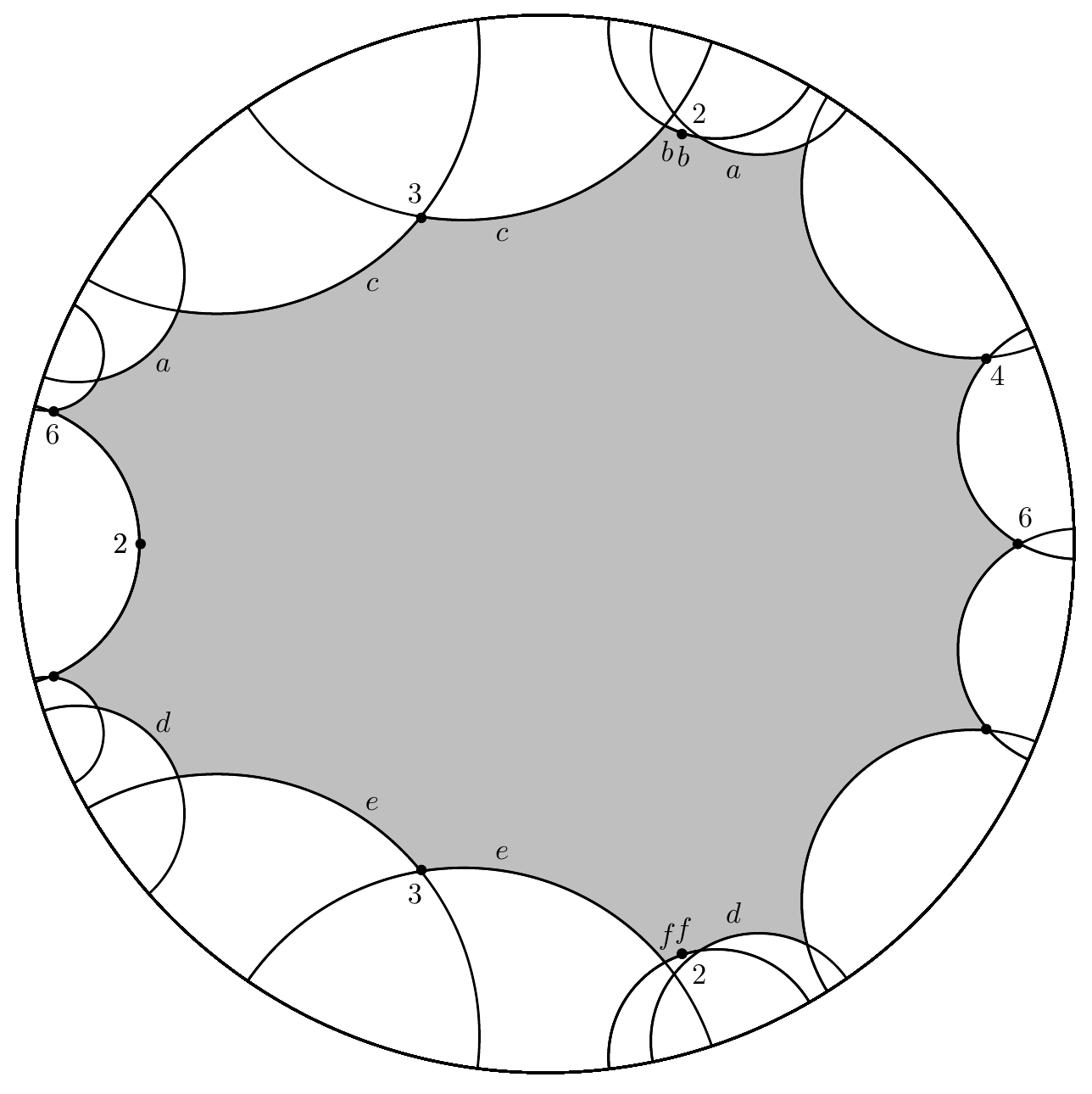} 
\end{equation*}
\centering
\textbf{Figure \ref{pic3}}: Fundamental domains for slightly larger $(0;2^3,3^2,4,6^2)$-orbifolds $X(\Gamma_1^*)$,  $X(\Gamma_2^*)$ \\ over $F=\Q(\sqrt{51})$
\addtocounter{equation}{1}
\end{figure}
\end{exm}

\subsection*{The smallest Sunada pairs?}

To conclude this section, we relax our requirement that the group be maximal and ask more broadly for the smallest arithmetic isospectral pair---here, we have some small examples that arises from the construction of Sunada \cite{Sunada}.  As mentioned in the introduction, Sunada's theorem reduces the construction of isospectral manifolds to a problem in finite group theory and has inspired an enormous amount of work within the field of inverse spectral geometry.  

It is not our intention to review these methods in any great detail, but in this section we connect our work to it: the reader interested in a more detailed description of Sunada's theorem and the work it has inspired is encouraged to consult the surveys of Gordon \cite{Gordon-IsospectralSurvey,Gordon-SunadaSurvey}.  In brief, we say that two subgroups $H_1,H_2 \leq G$ of a finite group $G$ are \defi{almost conjugate}  if each conjugacy class of $G$ meets $H_1$ and $H_2$ in the same number of elements.  We then have the following theorem.

\begin{thm}[Sunada] 
Let $G$ be a finite group of isometries acting on a Riemannian manifold $M$. Let $H_1,H_2 \leq G$ be almost conjugate in $G$. Then the orbit spaces $M/H_1$ and $M/H_2$ are isospectral orbifolds.
\end{thm}

Most isospectral, nonisometric orbifolds have been found using Sunada's method---but not all, and as noted above, the orbifolds (and manifolds) exhibited by Vign\'eras and featured in this article cannot arise from this construction \cite[Proposition 3]{Chen}.

A small isometric-nonisomorphic pair of $2$-orbifolds was found using Sunada's method by Doyle and Rossetti \cite[Section 2]{DoyleRossetti}: they claim ``this is likely the simplest pair of isospectral hyperbolic polygons, both in having the smallest number of vertices and having the smallest area''; their construction is a variation on the famous example of Gordon, Webb, and Wolpert \cite{GordonWebbWolpert}.  

The construction runs as follows.  They consider the triangle group 
\[ \Delta=\Delta(3,4,6)=\langle \delta_3,\delta_4,\delta_6 \mid \delta_3^3=\delta_4^4=\delta_6^6=\delta_3\delta_4\delta_6=1 \rangle; \]
this group is our group $\Gamma^+$ associated to a maximal order $\calO$ in a quaternion algebra $B$ over $F=\Q(\sqrt{6})$ of discriminant $\frakp=(2-\sqrt{6})$ (with $\N\frakp=2$) and ramified at one of the two real places.  There is a normal subgroup $\Pi \trianglelefteq \Delta$ with $\Delta/\Pi \cong \PSL_2(\F_7)=G$.  There are two almost conjugate but non-conjugate index $7$ groups in $G$, and their preimages yield two index $7$ subgroups $\Gamma,\Gamma' \leq \Delta$ that are almost conjugate but are not conjugate; by Sunada's theorem, they are isospectral, and Doyle and Rosetti further prove that they are not isometric.  The corresponding $2$-orbifolds have area $7\pi/2$ and signature $(0;2^3,3,4,6)$.  We have
\[ 7/4 = 1.75 < 1.91666\ldots = 23/12 \]
so the examples in our Theorem A are not the smallest example among all arithmetic groups---but they narrowly miss this!

\begin{lem}
The groups $\Gamma,\Gamma' \leq \Delta$ are not congruence.
\end{lem}

\begin{proof}
We see directly from the construction that the largest normal subgroup $\Pi \trianglelefteq \Delta$ contained in $\Gamma \cap \Gamma'$ corresponds to an (orbifold) Galois covering of $X(\Pi) \to X(\Delta)$ with Galois group $\PSL_2(\F_7)$, the simple group of order $168$.  Suppose that $\Gamma,\Gamma^\prime$ are congruence, containing the principal congruence subgroup $\Gamma(\frakN)$ for some ideal $\frakN \subseteq \Z_F=\Z[\sqrt{6}]$.  Since $\Gamma(\frakN)$ is normal in $\Delta$ it is normal in $\Gamma,\Gamma'$, so we have $\Pi \geq \Gamma(\frakN)$.  

If we write $\frakN=\frakM \frakp^e$ with $N\frakp=2$ as above, then we have
\[ \frac{\Delta}{\Gamma(\frakN)} \cong \PGL_2(\Z_F/\frakM) \times H \]
where $H$ is a solvable group, and so we have a surjective map
\[ \PGL_2(\Z_F/\frakM) \times H \cong \frac{\Delta}{\Gamma(\frakN)} \to \frac{\Delta}{\Pi} \cong \PSL_2(\F_7). \]
By the Chinese remainder theorem, we have $\GL_2(\Z_F/\frakM) \cong \prod_{\frakq^e \parallel \frakM} \GL_2(\Z_F/\frakq^e)$, and by the theory of finite simple Lie groups, one obtains $\PSL_2(\F_7)$ as a quotient of one of these factors if and only if it is a factor of $\GL_2(\Z_F/\frakq)$ for some $\frakq$ if and only if there is a prime $\frakq \subset \Z_F$ with $\N\!\frakq=7$.  Alas, there is not: $6 \equiv -1 \pmod{7}$ so $(7)$ is inert in $\Z_F$, and we have only a prime of norm $7^2=49$.
\end{proof}

We also find a small example of a torsion-free arithmetic isospectral nonisometric pair, due to Brooks and Tse \cite{BrooksTse}: again following Sunada, inside a group with signature $(1;7)$, they find torsion-free, index $7$ subgroups with signature $(4;-)$.  We verify that the signature $(1;7)$ is arithmetic (see e.g.\ the full treatment of $(1;e)$-groups by Sijsling \cite{Sijsling}).  

\begin{rmk}
We leave it as an open question if the example by Doyle and Rossetti is smallest among all isospectral-nonisometric pairs of arithmetic Fuchsian groups.  

Buser \cite[Chapter 12]{Buser} gives an infinite family of Riemann surfaces of genus $g \geq 4$ giving rise to nonisospectral $2$-manifolds.  We do not know if there are any isospectral-nonisometric pairs of $2$-manifolds in genus $3$ (arithmetic or not), and we have reason to believe that there are no such examples in genus $2$.
\end{rmk}

\section{Isospectral but not isometric hyperbolic 2-orbifolds: proofs}\label{section:orbifoldproofs}

We now prove that the three pairs with area $23\pi/6$ given in Section 3 are minimal (Theorem A). 

\begin{theorem} \label{theorem:bestorbifold}
The minimal area of an isospectral-nonisometric pair of $2$-orbifolds associated to maximal arithmetic Fuchsian groups is $23\pi/6$, and this bound is achieved by exactly three pairs, up to isomorphism.
\end{theorem}

The rest of this section will be devoted to the proof of Theorem \ref{theorem:bestorbifold}, consisting of three steps.

\begin{enumerate}
\item First, we show that there are only a finite number of totally real fields $F$ which can support a quaternion algebra $B$ and an Eichler order $\mathcal \calO\subset B$ with corresponding maximal Fuchsian group of coarea at most $23\pi/6$.  In particular, we obtain an explicit upper bound for the root discriminant $\delta_F=d_F^{1/n}$ of such a field $F$, and we enumerate all such fields.

\item Next, over each such field $F$, we bound the norm of the discriminant $\frakD$ and level $\frakN$, and then we enumerate all (finitely many) Eichler orders meeting these bounds (and in particular the volume bound). Recall that by Proposition \ref{prop:isospectralmeanslevel}, any isospectral-nonisometric pair of maximal arithmetic Fuchsian groups must arise from Eichler orders of a common level $\frakN$ inside of a common quaternion algebra of discriminant $\frakD$.

\item For each Eichler order $\calO$, we compute the type number $t(\calO)$ and keep $\calO$ if $t(\calO) > 1$; then we check for selectivity and keep $\calO$ if is not selective.  At the end of these computations, we find the list of three pairs of $2$-orbifolds announced in Theorem \ref{theorem:bestorbifold}.
\end{enumerate}

We now turn to each of these steps in turn.

\subsection*{Enumerating fields (1)}

To begin, we show that there are only a finite number of totally real fields $F$ which can support a quaternion algebra and an Eichler order with corresponding maximal Fuchsian group of coarea at most $23\pi/6$.  

We continue with the notation from Section 1.  In particular, let $F$ be a totally real field of degree $n=r$, let $B$ be a quaternion algebra over $F$ split at $s=1$ real place with discriminant $\frakD$, and let $\calO=\calO_0(\frakN) \subset B$ be an Eichler order of (squarefree) level $\frakN$ and discriminant $\frakd=\frakD\frakN$.  

The formula (\ref{volumeformula}) for the volume (area) in this case reads
\begin{equation} \label{X1vol}
\area(X^*) = \frac{2(4\pi)}{(4\pi^2)^n[\Gamma^*:\Gamma^1]} d_F^{3/2} \zeta_F(2) \Phi(\frakD)\Psi(\frakN).
\end{equation}
From (\ref{PhiPsi}) we have $\Phi(\frakD)\Psi(\frakN) \geq 1$, so 
\begin{equation} \label{equation:fullbound}
\area(X^*) \geq \frac{8\pi}{(4\pi^2)^n}\frac{1}{[\Gamma^*:\Gamma^1]} d_F^{3/2} \zeta_F(2).
\end{equation}

Suppose now that $\area(X^*) \leq 23\pi/6$.  (Our arguments work as well with a different bound on the area.)  Then the formula (\ref{riemannhurwitz}) 
for the co-area of a Fuchsian group in terms of its signature implies that $\Gamma^*$ is allowed to have one of only finitely many possible signatures and among these at most $\lfloor 2(23/12+2-2g) \rfloor =7-4g$ conjugacy classes of elliptic cycles of even order.  Therefore, from the presentation (\ref{presentation}), the maximal abelian quotient $(\Gamma^*)^{\textup{ab}}$ modulo squares has $2$-rank at most 
\[ (7-4g)+2g-1=6-2g \leq 6. \]  
But $\Gamma^*/\Gamma^1$ is an elementary abelian $2$-group of size $[\Gamma^*:\Gamma^1]$, so $[\Gamma^*:\Gamma^1] \leq 2^{6}$. 

\begin{rmk}
We do not expect that the bound $[\Gamma^*:\Gamma^1] \leq 2^{6}$ on the index to be sharp for groups of such small volume.  However, there are arithmetic Fuchsian groups with signature $(0;2^{6})$ and so we cannot rule it out theoretically.  This bound is enough for our purposes, but we believe it would still be profitable to consider what kind of bound on $[\Gamma^*:\Gamma^1]$ is implied by a bound on the area of $X(\Gamma^1)$ (for arithmetic groups).
\end{rmk}

From the trivial estimate $\zeta_F(2) \geq 1$, we find
\[ \frac{23\pi}{6} \geq \area(X^*) \geq \frac{8\pi}{(4\pi^2)^n 2^{6}} d_F^{3/2} \]
so
\begin{equation} \label{equation:odlyzko}
\begin{aligned}
\delta_F^{3/2} &\leq 4\pi^2 \left( \frac{92}{3}\right)^{1/n}  \\
\delta_F &\leq (11.596) \cdot (30.666)^{1/n} 
\end{aligned}
\end{equation}
where $\delta_F=d_F^{1/n}$ is the root discriminant of $F$.

But now by the (unconditional) Odlyzko bounds \cite{Odlyzko-bounds} (see also Martinet \cite{Martinet}), if $F$ has degree $n \geq 13$ then $\delta_F \geq 16.104$; but this contradicts (\ref{equation:odlyzko}), which insists that $\delta_F \leq 15.089$.  So we conclude that $n \leq 12$.  

We now refine this rough bound for these smaller degrees.  
Let $\Z_{F,+}^\times/\Z_F^{\times 2} \cong (\Z/2\Z)^m$ with $0 \leq m < n$, let  $\omega(\frakd)=\omega(\frakD) + \omega(\frakN)$ be the number of prime divisors of $\frakd$, and let $h_2=\rk_2 (\Cl \Z_F)[2]$.  Then by Lemma \ref{lem:gamma1*} and Proposition \ref{corgamma} we have 
\begin{equation} \label{Gamma1index}
[\Gamma^*:\Gamma^1] \leq 2^{m+\omega(\frakD)+\omega(\frakN)+h_2};
\end{equation}
we note that equality holds if and only if each of the prime divisors of $\frakD$ and $\frakN$ belong to a square class in $\Cl^+ \Z_F$ and further $(\Cl \Z_F)[2] = (\Cl \Z_F)[2]_+$.  
Substituting (\ref{Gamma1index}) back into (\ref{X1vol}), we have
\begin{equation} \label{equation:xstar2}
\area(X^*) \geq \frac{8\pi}{(4\pi^2)^n 2^{m+h_2}} d_F^{3/2} \zeta_F(2) \frac{\Phi(\frakD)}{2^{\omega(\frakD)}}\frac{\Psi(\frakN)}{2^{\omega(\frakN)}}. 
\end{equation}

We have $\Psi(\frakN)/2^{\omega(\frakN)} \geq 1$.  Next, we estimate $\Phi(\frakD)/2^{\omega(\frakD)}$ from below.  Let
\[ \omega_2(\frakD)=\#\{\frakp \mid \frakD: N\frakp=2\}. \]
By multiplicativity and the facts that $\Phi(\frakp)/2 = 1/2,1,1$ for $N\frakp=2,3,4$, respectively, and $\Phi(\frakp)/2 \geq 2$ for $N\frakp \geq 5$, we find that 
\begin{equation} \label{Phi2D}
\frac{\Phi(\frakD)}{2^{\omega(\frakD)}} \geq \frac{1}{2^{\omega_2(\frakD)}}.
\end{equation}

By the Euler product expansion
\[ \zeta_F(s)=\prod_{\frakp} \left(1-\frac{1}{N\frakp^{s}}\right)^{-1}, \]
convergent for $\repart s>1$, we have 
\[ \zeta_F(2) \geq \left(\frac{4}{3}\right)^{\omega_2(\frakD)}. \] 
Combining these with (\ref{equation:fullbound}) we get
\begin{equation}\label{equation:areabound}
\frac{23\pi}{6} \geq \area(X^*) \geq \frac{8\pi}{(4\pi^2)^n 2^{m+h_2}} d_F^{3/2} \left(\frac{2}{3}\right)^{\omega_2(\frakD)},
\end{equation}
hence
\begin{equation} \label{equation:betterbound}
\delta_F^{3/2} \leq 4\pi^2 \left(\frac{23}{48}\right)^{1/n} 2^{(k+h_2)/n} \left(\frac{3}{2}\right)^{\omega_2(\frakD)/n}.
\end{equation}

We now use (\ref{equation:betterbound}) together with two arguments to find tighter bounds.  The first of these arguments is the theorem of Armitage and Fr\"ohlich \cite{armitage-frolich}, giving the inequality
\[ h_2 \geq m - \lfloor n/2 \rfloor. \]
The second is a refinement of the Odlyzko bounds that utilizes the existence of primes of small norm, due to Poitou \cite{Poitou} and further developed by Brueggeman and Doud \cite{doud,doud-tables}.  

We show how these arguments can be used to show that $n\leq 9$.  

\begin{lem} \label{lem:nle9}
We have $n \leq 9$.
\end{lem}

\begin{proof}
Suppose that $n=10$; the proofs for $n=11,12$ are easier.  Our original bound (\ref{equation:odlyzko}) reads $\delta_F \leq 17.527$.  Suppose that $h_2 \geq 1$; then the Hilbert class field of $F$ has degree $\geq 2n=20$ and root discriminant $\delta_F \leq 17.527$, contradicting the Odlyzko bounds.  So $h_2=0$.  From Armitage-Fr\"ohlich, we have $m \leq 5$.  Consequently, from (\ref{equation:betterbound}), we obtain
\begin{equation} \label{equation:yup10}
\begin{aligned}
\delta_F^{3/2} &\leq (4\pi^2) (23/48)^{1/10} 2^{1/2} (3/2)^{\omega_2(\frakD)/10} \\
\delta_F &\leq (13.909) \cdot (1.028)^{\omega_2(\frakD)}
\end{aligned}
\end{equation}
If $\omega_2(\frakD)=0$ then this yields $\delta_F \leq 13.909$; however, Voight \cite{voight-fields} has shown that there is no totally real field of degree $10$ with root discriminant $\leq 14$, a contradiction.  If $\omega_2(\frakD)=1$, then by the Odlyzko-Poitou bounds, we find that $\delta_F \geq 15.093$ and this contradicts (\ref{equation:yup10}) which insists $\delta_F \leq 14.299$.  We similarly find contradictions for $\omega_2(\frakD)=2$ and then for $\omega_2(\frakD) \geq 3$.  
\end{proof}

We repeat this argument with the degrees $2 \leq n \leq 9$, extending the Odlyzko-Poitou bounds as needed \cite[Theorem 2.4]{doud}.  We illustrate the iterative derivation of these bounds with the case $n=5$.  

\begin{lem} \label{lem:neq5}
If $n=5$, then $\delta_F \leq 16.312$.
\end{lem}

\begin{proof}
Let $H_2$ be the subfield of the Hilbert class field of $F$ corresponding to $(\Cl \Z_F)[2]$, a totally real number field with $[H:F]=2^{h_2}$.  We begin with the bound (\ref{equation:odlyzko}) which provides $\delta_F \leq 26.490$.  If $h_2 \geq 3$, then $H_2$ has degree $\geq 40$ and root discriminant $\delta_{H_2}=\delta_F$, contradicting the Odlyzko bounds; thus $h_2 \leq 2$.  We have $m \leq n-1=4$ by Dirichlet's unit theorem.  We have $\omega_2(\frakD) \leq n = 5$; but if $\omega_2(\frakD)=5$ then (\ref{equation:betterbound}) provides $\delta_F \leq 23.981$ and the Poitou bounds give $\delta_F \geq 24.049$, a contradiction.  So $\omega_2(\frakD) \leq 4$.  Thus (\ref{equation:betterbound}) gives $\delta_F \leq 22.719$.  

Suppose that $h_2=2$.  Then there are at least $8$ primes of norm at most $4$ in $H_2$, a totally real field of degree $20$; but the Poitou bounds give $\delta_{H_2} = \delta_F \geq 29.517$, a contradiction.  So either $h_2=1$ or $\omega_2(\frakD) \leq 3$.  If $h_2=2$ and $\omega_2(\frakD) \leq 3$, then $\delta_{H_2}=\delta_F \leq 21.523$, contradicting the Odlyzko bounds.  So $h_2 \leq 1$.  But then by Armitage-Fr\"ohlich, we have $m \leq 3$.  If $\omega_2(\frakD)=4$ then $\delta_F \leq 18.885$ but $H_2$ has degree $10$ and at least $4$ primes of norm at most $4$ and the Poitou bounds give $\delta_H \geq 19.275$, a contradiction.  

So to summarize, we have $h_2 \leq 1$ and $m \leq 3$ and $\omega_2(\frakD)\leq 3$, and consequently the bound (\ref{equation:betterbound}) reads $\delta_F \leq 17.891$.  

Suppose that $h_2=1$.  Recall that equality holds in the bound (\ref{Gamma1index}) only if each of the prime divisors of $\frakD$ belong to the square of a class in $\Cl^+ \Z_F$, and hence projects to the trivial class in $(\Cl \Z_F)[2]$.  So if the inequality is tight, then all of the primes of norm $2$ which divide $\frakD$ split in $H_2$; so $H_2$ is a totally real field of degree $10$ with at least $2\omega_2(\frakD)$ primes of norm $2$, and the Poitou bounds now yield a contradiction unless $\omega_2(\frakD) \leq 1$, whence $\delta_F \leq 16.058$.  If the inequality is not tight, then at least one prime of norm $2$ does not belong to a square class: this reduces the size of the index $[\Gamma^*:\Gamma^1]$ by (at least) a factor $2$, and it improves the bound (\ref{equation:betterbound}) to $\delta_F \leq 16.312$.  

If $h_2=0$, then $m \leq 2$ and also we obtain $\delta_F \leq 16.312$.  Taking the maximum of these bounds, the proof is complete.
\end{proof}

The derivations in Lemma \ref{lem:nle9} and Lemma {lem:neq5} in each case can get involved, so we also checked these derivations algorithmically (the output is available upon request).  In this manner we obtain upper bounds for the root discriminant $\delta_F$ which we refer to as the Odlyzko-Poitou-Armitage-Fr\"ohlich (OPAF) bounds in (\ref{odlyzkotableyeah}).  For comparison, we list also the corresponding Odlyzko bounds (and, although we will not use them, the Odlyzko bounds assuming the GRH \cite{CD}).  We enumerate

\begin{equation} \label{odlyzkotableyeah}
\begin{array}{c||cc|c|c}
n & \textup{Odlyzko} & \textup{Odlyzko (GRH)} & \textup{OPAF} & \Delta \\
\hline
\rule{0pt}{2.5ex} 
2 & > 2.223 & (> 2.227) & < 23.778 & 30 \\
3 & 3.610 & (3.633) & 20.480 & 25 \\
4 & 5.067 & (5.127) & 19.007 & 20 \\
5 & 6.523 & (6.644) & 16.312 & 17 \\
6 & 7.941 & (8.148) & 15.410 & 16 \\
7 & 9.301 & (9.617) & 14.236 & 15.5 \\
8 & 10.596 & (11.042) & 12.661 & 15 \\
9 & 11.823 & (12.418) & 12.167 & 14.5 \\

\end{array}
\end{equation}

Voight \cite{voight-table} has enumerated all totally real fields of degree at most $10$ and small root discriminant, given by the bound $\Delta$ in (\ref{odlyzkotableyeah}).  Therefore, we obtain a finite (but long) list of fields to check.

\subsection*{Enumerating orders (2)}

Next, we enumerate all possible Eichler orders satisfying the bounds in the previous section.

To that end, let $F$ be a totally real field enumerated in the previous section and continue the assumption that $\area(X^*) \leq 23\pi/6$.  Equation (\ref{equation:xstar2}) implies
\begin{equation} \label{eqn:DNbound}
\frac{\Phi(\frakD)}{2^{\omega(\frakD)}} \leq \frac{\Phi(\frakD)}{2^{\omega(\frakD)}}\frac{\Psi(\frakN)}{2^{\omega(\frakN)}} \leq \frac{23}{48}(4\pi^2)^n \frac{2^{m+h_2}}{d_F^{3/2} \zeta_F(2)}.
\end{equation}
so $N\frakD$ and hence $N\frakN$ are bounded above by a factor that depends only on $F$.  The set of ideals with bounded norm can be enumerated using standard techniques \cite{Cohen}.  The quaternion algebra $B$ is characterized up to isomorphism by the finite set (of even cardinality) of places at which it ramifies.  Since our quaternion algebra $B$ must be ramified at all but one real place, it suffices to enumerate the collection of all possible discriminants $\frakD$ for quaternion algebras $B$ over $F$ which satisfy (\ref{eqn:DNbound}).  For each such $\frakD$, we compute the possible squarefree levels $\frakN$ again using (\ref{eqn:DNbound}).

\subsection*{Computing the type number and checking for selectivity (3)}

Finally, we check each order to see if we obtain an isospectral, nonisometric pair.

So let $B$ be a quaternion algebra of discriminant $\frakD$ over $F$ and $\calO \subset B$ an Eichler order of level $\frakN$ enumerated in the previous section.  We first check for the existence of a nonisometric pair.  From Propositions \ref{prop:isospectralmeanslevel} and \ref{prop:isometrytheorem}, it is necessary and sufficient for the type number $t(\calO)=\#T(\calO)$ to satisfy $t(\calO)>1$, where $T(\calO)$ is the collection of isomorphism classes of Eichler orders of level $\frakN$ as defined in section 1.  Indeed, if $t(\calO)=1$ then any two Eichler orders of level $\frakN$ will be conjugate, hence their normalizers are conjugate as well, and hence the associated surfaces will be isometric.  

By Proposition \ref{classnotypeno}, there is a bijection 
\begin{equation} \label{eqn:typeclassgroup}
 T(\calO) \xrightarrow{\sim} \frac{\Cl^{(+)}(\Z_F)}{\{[\fraka] : \fraka \parallel \frakd\}\Cl^{(+)}(\Z_F)^2 }
\end{equation}
where $\Cl^{(+)}(\Z_F)$ is the class group associated to the modulus given by the product of the real ramified places in $B$.  The group on the right in (\ref{eqn:typeclassgroup}) can be computed using computational class field theory.  We may also verify that these calculations are correct using the algorithms of Kirschmer and Voight \cite{KirschmerVoight}: we compute a complete set of right ideal classes of $\calO$ and compute the set of isomorphism classes amongst their left orders.

Let $T(B)$ denote the type set for the class of maximal orders in $B$ and $t(B) =\#T(B)$.  Then by (\ref{eqn:typeclassgroup}), we have $T(\calO) \leq T(B)$ for any Eichler order $\calO \subset B$.  This remark allows us to quickly rule many orders in one stroke, if $t(B)=1$.

As it turns out, in our list of orders we always have $t(B) \leq 2$.  Let $\calO$ be a group on our list with $t(\calO)=2$, and let $\calO_1=\calO$ and $\calO_2$ be representatives of $T(\calO)$, with associated surfaces $X_1^*$ and $X_2^*$.

\begin{rmk}
If we had found an order where $t(\calO) \geq 4$, we still obtain examples if there is no selectivity obstruction.  But even more is true: if for example $t(\calO)=4$ and the selective field $L$ in Theorem \ref{theorem:chinburgfriedman_selectivity} is unique, then there will still be a pair amongst the four surfaces which will be isospectral.  If there are multiple selective fields and large type number, then this quickly gets complicated.
\end{rmk}

It remains to check whether the surfaces $X_1^*,X_2^*$ are in fact isospectral.  According to Theorem \ref{theorem:isospectraltheorem}, we verify that $\Gamma_1,\Gamma_2$ are not selective for any conjugacy class.  The possible selective conjugacy classes, and their quotient fields $L$, are given in Theorem \ref{theorem:chinburgfriedman_selectivity}.


We check for a possible selective field $L$ in two ways.  The simplest check is to compute the sign matrix associated to a basis of units for $\Z_F^\times/\Z_F^{\times 2}$ and conclude whenever possible that the field $L$ cannot exist (see Example \ref{exmselect} below).  A second way is to use computational class field theory to compute the ray class field $H^{(+)}$ associated to $\Cl^{(+)} \Z_F$ by Artin reciprocity and check for each quadratic subfield $L$ if $L$ embeds in $B$ (verified easily by Theorem \ref{theorem:abhn}).  If no field $L$ exists, then we conclude that the surfaces $X_1^*$ and $X_2^*$ are isospectral but not isometric.  

\begin{exm} \label{exmselect}
We consider a small example where we can rule out selectivity.  (It is the smallest example we found considering only groups of the form $\Gamma^+$, leaving out normalizers.)  Let $F=\mathbb Q(w)$ where $w^5-7w^3-2w^2+11w+5=0$.  Then $F$ is a totally real field of degree $r=n=[F:\Q]=5$ with discriminant $220669=149\cdot 1481$, Galois group $S_5$, and ring of integers $\Z_F=\Z[w]$.  The class number of $F$ is $\#\Cl \Z_F = 1$ and the narrow class number of $F$ is $\# \Cl^+ \Z_F = 2$; the nontrivial class is represented by the ideal $\frakb=(w^2 - 3)$ with $N\frakb=2$.

Let $B$ be the quaternion algebra over $F$ which is unramified at all finite places of $F$, ramified at four of the five real places, and split at the remaining real place with $t \mapsto 1.557422\ldots$; then $B$ has discriminant $\frakD=(1)$.  We take $B=\quat{-1,b}{F}$ where $b=-w^4 + w^3 + 5w^2 - 2w - 6$ satisfies
$b^5 + 8b^4 + 14b^3 - 10b^2 - 8b - 1=0$.

As usual, we find a maximal order $\calO \subset B$.  The class number with modulus equal to the product of the four ramified real places of $F$ is $\#\Cl^{(+)} \Z_F = 2$; therefore by Proposition \ref{classnotypeno}, the class set $\Cl \calO$ is of cardinality $2$, with nontrivial class represented by a right $\calO$-ideal $I$ of reduced norm $\frakb$.  We see also from Proposition \ref{classnotypeno} that the type number of $\calO$ is equal to $t(\calO)=2$; let $\calO_1=\calO$ and $\calO_2$ be representatives of these classes.

We compute that $\Gamma_i^1 > \Gamma_i^+=\Gamma_i^*$, the first with index $2$; the $2$-orbifolds $X(\Gamma_i)$ have area equal to $19\pi/3$ and signature $(0;2^9,3)$.  

We verify that the groups $\Gamma_i^*$ are not selective.  We refer to Theorem \ref{theorem:chinburgfriedman_selectivity}.  Suppose that $g^{B^\times}$ is a selective conjugacy class, corresponding to a selective quadratic $\Z_F$-order $R=\Z_F[g] \subset L$.  Then the conductor of $L$ (as an abelian extension) is equal to the product of the four ramified places in $F$.  We organize this signs of units as follows.  Let $v_1,\dots,v_r$ be the real places of $F$.  Then we define the \defi{sign} map 
\begin{align*}
\sgn: \Z_F^\times &\to \{\pm 1\}^r \cong \F_2^r \\
u &\mapsto (\sgn(v_j(u))_{j=1,\dots,r}
\end{align*}
$\sgn$ is the usual sign of a real number.  Then there are units $u_1,\dots,u_5 \in \Z_F^\times$ that form a basis for $\Z_F^\times/\Z_F^{\times 2}$ with 
\begin{equation} \label{eqn:signmatrix}
 (\sgn(v_j(u_i))_{i,j=1}^{5} = 
\begin{pmatrix} 
1 & 0 & 0 & 0 & 0 \\
0 & 1 & 0 & 1 & 0 \\
0 & 0 & 1 & 0 & 0 \\
0 & 0 & 0 & 0 & 1 \\
0 & 0 & 0 & 0 & 0
\end{pmatrix}
\end{equation}
The size of the class group of modulus given by a subset of real places is equal to the corank of the matrix with the corresponding columns deleted; this corank remains $1$ for $\{v_1\}$, so no such $L$ exists.  We verify this using computational class field theory: the ray class field with this modulus is given by $F(\sqrt{u})$ where $u=w^4 - w^3 - 6w^2 + 3w + 8$ satisfies $u^5 - 4u^4 - 8u^3 + 4u^2 + 4u - 1=0$; however, $F(\sqrt{u})$ has conductor given by the product of just the real places $v_2v_4$, so $L \neq F(\sqrt{u})$ and hence $L$ does not exist.  

Finally, we conclude using Theorem \ref{theorem:isospectraltheorem} that the hyperbolic $2$-orbifolds $X(\Gamma_1)$ and $X(\Gamma_2)$ are isospectral but not isometric.  We obtain a total of three examples by appropriate choice of the split real place: the argument in the previous paragraph works also for the embeddings $t \mapsto -0.480942\ldots, -1.915645\dots$.
\end{exm}

In this way, we can verify that the orders are not selective.  If instead a quadratic subfield $L$ of $H^{(+)}$ exists and embeds into $B$ satisfying the above conditions, we compute further to verify that in fact the orders are selective as follows.  We search for a unit $\gamma \in \Z_L$ such that:
\begin{enumerate}
\item[(i)] $\gamma$ has totally positive norm to $F$; 
\item[(ii)] $R=\Z_F[\gamma]$ satisfies criteria (1) and (2) in Theorem \ref{theorem:eichlerorderselectivity}, so that $R$ embeds into either $\calO_1$ or $\calO_2$ (or both); and
\item[(iii)] $R=\Z_F[\gamma]$ satisfies criterion (2) in Theorem \ref{theorem:eichlerorderselectivity}, so that $R$ embeds into exactly one of $\calO_1,\calO_2$.  
\end{enumerate}
Our search technique is naive: we simply compute a (minimal) set of generators for $\Z_L^\times$ and take small powers of these units.  In each case, we are successful in finding such a unit $\gamma$ without looking too hard, thereby verifying that the associated orbifolds are indeed not isospectral.




Given that selectivity has proven to be a delicate issue in past calculations, to illustrate this computational method we consider a particular quaternion algebra of type number $2$ and prove that maximal orders of distinct type in this algebra do not give rise to isospectral surfaces.

\begin{exm}
Let $F=\Q(w)$ where $w^3-w^2-6w-1=0$; then $d_F=761$ (prime) and $F$ has $\#\Cl \Z_F=1$ and $\#\Cl^+\Z_F=2$.  Let $B$ be the quaternion algebra over $F$ which is unramified at all finite primes and ramified at two of the three real places, with split place $w \mapsto 3.064434\dots$.  Let $\calO$ be a maximal order; then $t(\calO)=2$.  Let $\calO_1=\calO$ and $\calO_2$ be maximal orders of distinct type.  Since $\frakD=\frakN=(1)$, we have $\Gamma_i^*=\Gamma_i$ for $i=1,2$.  We show that the corresponding orbifolds $X_1$ and $X_2$ are not isospectral.

For the record, we note that $\area(X_i)=5\pi/3$ and the groups $\Gamma_i$ have signature $(0;2,2,2,3,3)$ for $i=1,2$.  We compute that the ray class field $H^{(+)}$ with modulus equal to the product of the ramified places is $H^{(+)}=F(\sqrt{w})$, so we take $L=H^{(+)}$.  The field $L$ has conductor equal to this modulus; it follows from Theorem \ref{theorem:abhn} that there exists an embedding $L \hookrightarrow B$.  By searching as above, we find a (hyperbolic) unit $\gamma \in \Z_L^\times$ that satisfies
\[ \gamma^2 - (177228w^2 + 365877w + 57836)\gamma + 1. \]
By Theorem \ref{theorem:eichlerorderselectivity}, the order $R=\Z_F[\gamma]$ embeds in one of $\calO_1,\calO_2$ but not both.  Therefore the length of the closed geodesic associated to $\gamma$ lies in the length spectrum of one of $X_1,X_2$ but not both, and the two surfaces are not isospectral.


\end{exm}

In this manner, we have found the maximal arithmetic Fuchsian groups of minimal area which are isospectral but not isometric: they are the ones of area $23\pi/6$ exhibited in section 3.  This concludes the proof of Theorem \ref{theorem:bestorbifold}.

\section{Isospectral but not isometric hyperbolic 2-manifolds}

The search performed in Section \ref{section:orbifoldsearch} is easily adapted in order to find examples of isospectral but not isometric hyperbolic $2$-manifolds having small volume; indeed, $X(\Gamma)$ is a manifold if and only if $\Gamma$ is torsion free, so we need only add this additional hypothesis to our search, and thereby proving Theorem B.

We maintain the notation in Section 1.  In particular, let $\calO$ be an Eichler order of level $\frakN$ and let $\Gamma^1=\Gamma_0^1(\frakN)$.  
Because the maximal Fuchsian groups $\Gamma^*$ are more likely to have torsion, we consider more generally unitive groups $\Gamma^1 \leq \Gamma \leq \Gamma^+$ introduced in \eqref{align:arithmeticgroups}.  Indeed, in the notation of Section 2, if $\gamma \in \Gamma^* \in B^\times/F^\times$ has $\nrd(\gamma)=a$ and the order $R=\Z_F[x]/(x^2-a)$ embeds in $\calO$, then any any group containing $\Gamma$ will visibly have torsion.  This need not happen in every example, but for the cases of small area considered here (such as those arising from algebras of trivial discriminant), in practice this happens very often.  So already it is a bit generous to consider groups larger than $\Gamma^1$ when looking for isospectral-nonisometric pair of $2$-manifold groups.  

If the group $\Gamma^1$ has torsion, then so does any group containing it, so we first check for torsion in the group $\Gamma^1$.  The torsion in $\Gamma^1$ corresponds in a natural way to embeddings of cyclotomic quadratic fields; see e.g.\ work of Voight \cite[\S 2]{Voight-shim}.  An embedding $R \hookrightarrow \calO$ of a quadratic $\Z_F$-order is \defi{optimal} if $RF \cap \calO = R$.  

\begin{lem} \label{lemcycquad}
There is a bijection between conjugacy classes of elements in $\Gamma^1$ of finite order $q \geq 2$ and $\calO_1^\times$-conjugacy classes of optimal embeddings $R \hookrightarrow \calO$ with $\#R_{\textup{tors}}^\times=2q$.  
\end{lem}

In particular, it follows from Lemma \ref{lemcycquad} that if no cyclotomic field $K=F(\zeta_{2q})$ with $[K:F]=2$ embeds in $B$, then $\Gamma^1$ is torsion free.  (This statement becomes an ``if and only if'' when $\calO$ is maximal.)

To begin, we exhibit a pair of compact hyperbolic $2$-manifolds of genus $6$ which are isospectral but not isometric.

\begin{exm} \label{exm:manifoldexample}
Let $F=\Q(t)$ where $t^4-5t^2-2t+1=0$.  Then $F$ is a totally real quartic field with discriminant $d_F=5744=2^4 359$, Galois group $S_4$, class number $\#\Cl \Z_F=1$ and narrow class number $\#\Cl^+ \Z_F=2$.  Let $B$ be the quaternion algebra over $F$ which is ramified at the prime ideal $\frakp_{13}$ of norm $13$ generated by $b=t^3-t^2-4t$ and three of the four real places, with split place $t \mapsto -0.751024\ldots$: then $B=\quat{a,b}{F}$ where $a = t^3 - t^2 - 3t - 1$ satisfies $a^4+8a^3+12a^2-1=0$.  

A maximal order $\calO \subset B$ is given by
\[ \calO = \Z_F \oplus \Z_F i \oplus \Z_F \frac{(t^3+1)+t^2i + j}{2} \oplus \frac{(t+1)+(t^3+1)i+ij}{2}\Z_F; \]
$\calO$ has type number $t(\calO)=2$, so there exists two isomorphism classes of maximal orders $\calO_1=\calO$ and $\calO_2$.  Thus $X_1^1,X_2^1$ are not isometric, by Proposition \ref{prop:isometrytheorem}.  Since $B$ is ramified at a finite place, the orders are not selective by Theorem \ref{theorem:chinburgfriedman_selectivity}, thus $X_1^1,X_2^1$ are isospectral but not isometric.

In fact, $X_1^1$ and $X_2^1$ are $2$-manifolds.  We refer to Lemma \ref{lemcycquad}.  If $F(\zeta_{2q})$ is a cyclotomic quadratic extension of $F$ then $\Q(\zeta_{2q})^+ \subseteq F$; but $F$ is primitive, so the only cyclotomic quadratic extensions of $F$ are $K=F(\sqrt{-1})$ and $K=F(\sqrt{-3})$.  But as $\mathfrak p_{13}$ splits completely in $F(\sqrt{-1})$ and $F(\sqrt{-3})$, neither field embeds into $B$.  Thus by Lemma \ref{lemcycquad}, the groups $\Gamma_i$ are torsion free.

Finally, by (\ref{volumeformula}) we have $\area(X_i)=20\pi$, so $g(X_i)=6$ for $i=1,2$.  

Note that in this case whereas $X_1^+$ and $X_2^+$ have smaller area than $X_1^1$ and $X_2^1$ (Lemma \ref{lem:gamma1*} implies that they have area $10\pi$ ) and are isospectral-nonisometric, they do not give rise to a smaller pair of isospectral-nonisometric $2$-manifolds as the covering groups $\Gamma_1^+$ and $\Gamma_2^+$ are not torsion-free. Indeed, we compute their signatures to be $(3; 2,2)$.

We obtain a second example by choosing the split real place $t \mapsto -1.9202\ldots$, and since $F$ is not Galois, as in the case of the $2$-orbifold pairs 2 and 3, these are pairwise nonisometric.
\end{exm}

\begin{figure}
\begin{equation} \label{quartic1} \notag
\includegraphics[scale=0.75]{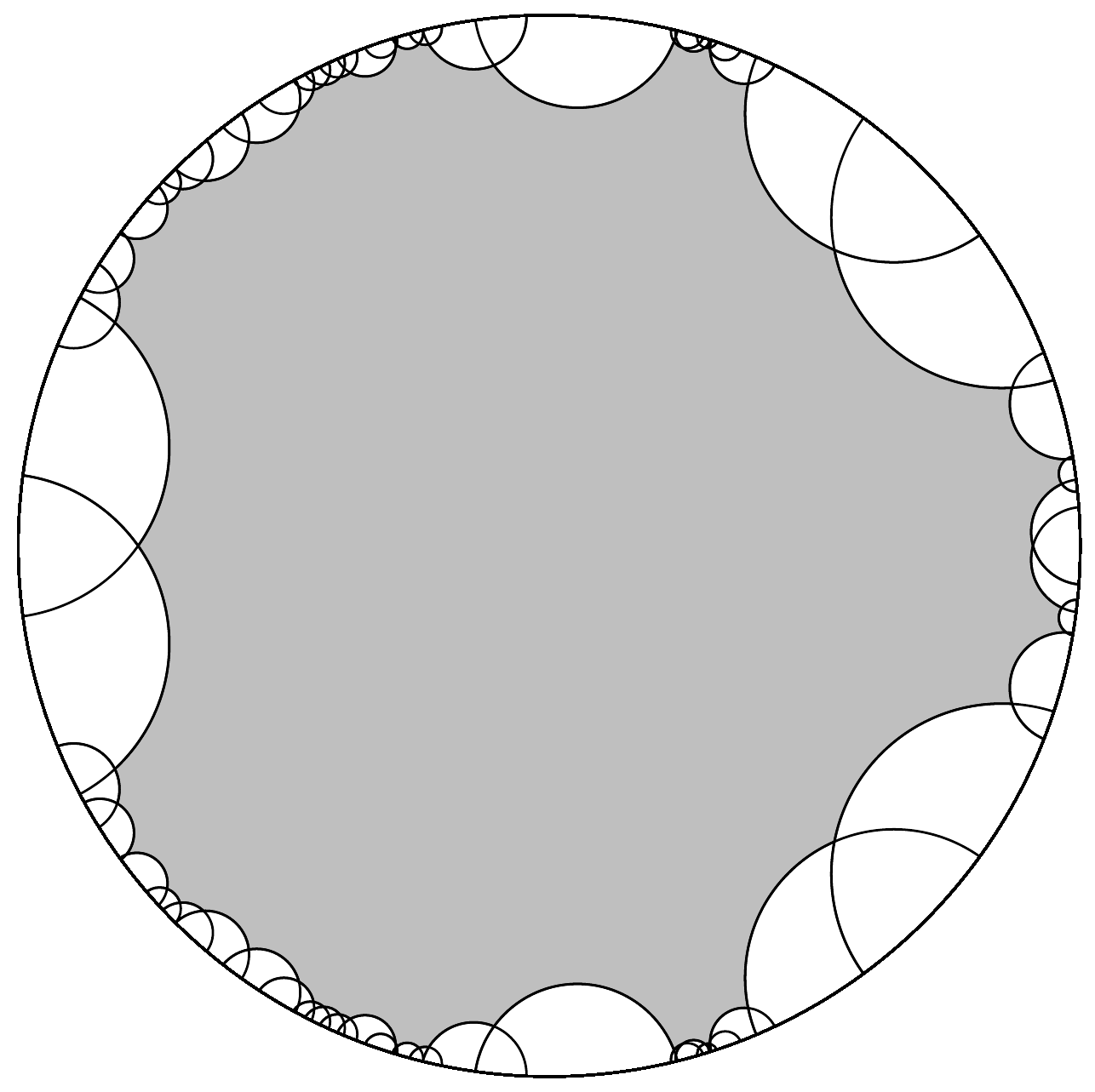} 
\end{equation}
\centering
\textbf{Figure \ref{quartic1}}: Fundamental domain for the genus $6$ manifold $X(\Gamma_1^1)$ over $F$ of discriminant $5744$
\addtocounter{equation}{1}

\begin{equation} \label{quartic2} \notag
\includegraphics[scale=0.75]{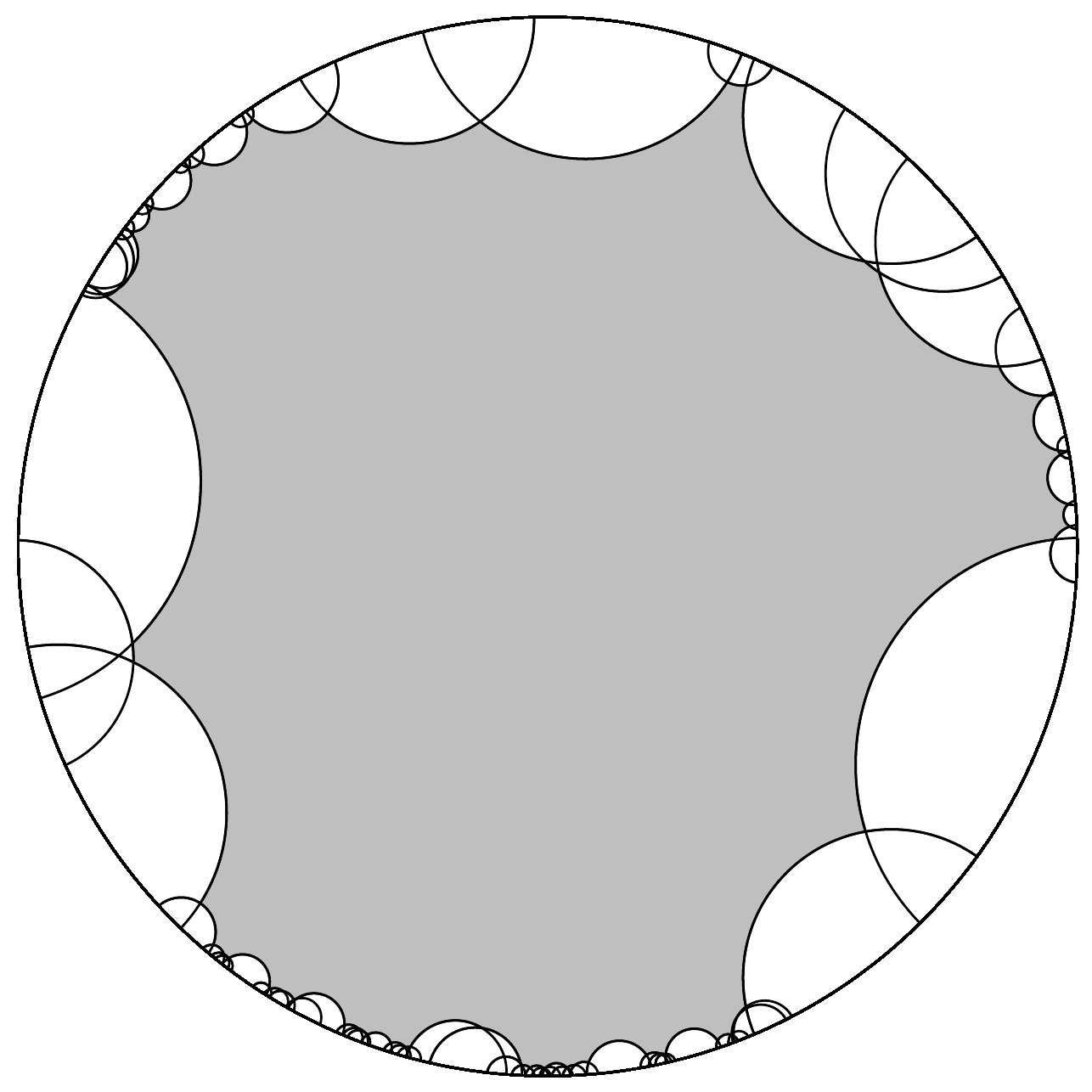} 
\end{equation}
\centering
\textbf{Figure \ref{quartic2}}: Fundamental domain for the genus $6$ manifold $X(\Gamma_2^1)$ over $F$ of discriminant $5744$
\addtocounter{equation}{1}
\end{figure}

\begin{theorem}\label{theorem:manifoldexamples}
Let $\Gamma,\Gamma'$ be torsion-free unitive Fuchsian groups such that the $2$-manifolds $X(\Gamma),X(\Gamma')$ are isospectral-nonisometric.  Then $g(X(\Gamma))=g(X(\Gamma')) \geq 6$ and this bound is achieved by the two pairs of surfaces in Example \textup{\ref{exm:manifoldexample}}, up to isomorphism.
\end{theorem}

The techniques used to prove Theorem \ref{theorem:manifoldexamples} are similar to those used in Section \ref{section:orbifoldsearch}, though a few additional complications arise due to the larger volume being considered and the fact that the quaternion algebras being considered may not contain any primitive $q$-th roots of unity for $q\geq 3$.

\subsection*{Enumerating the fields}

We begin by enumerating the totally real fields $F$ which can support a quaternion algebra $B$ and an Eichler order $\calO\subset B$ such that $\area(X^+)\leq 20\pi$.

Employing the same techniques that were used in Section \ref{section:orbifoldsearch} yields the following table (compare with (\ref{odlyzkotableyeah})).

\begin{equation} \label{manifoldodlyzkotableyeah}
\begin{array}{c|cc|c|c}
n  & \textup{Odlyzko} & \textup{Odlyzko (GRH)} & \textup{OPAF} & \Delta \\
\hline
\rule{0pt}{2.5ex} 
2 & > 2.223 & (> 2.227) & < 19.826 & 30 \\
3 & 3.610 & (3.633) & 19.341 & 300 \\
4 & 5.067 & (5.127) & 19.102 & 100 \\
5 & 6.523 & (6.644) & 17.287 & 35 \\
6 & 7.941 & (8.148) & 17.469 & 28 \\
7 & 9.301 & (9.617) & 15.423 & 15.5 \\
8 & 10.596 & (11.042) & 15.767 & 17 \\
9 & 11.823 & (12.418) & 14.917 & 15 \\
\end{array}
\end{equation}
We extend the tables of totally real fields \cite{voight-fields} up to root discriminant $\Delta$ to cover the spread in low degree.

\subsection*{Enumerating the orders, computing the type numbers and checking for torsion}

Given a totally real field $F$ satisfying the root discriminant bounds in (\ref{manifoldodlyzkotableyeah}) we compute all possible quaternion algebras $B$ over $F$ which could contain an Eichler order $\calO$ for which $\area(X^+)\leq 20\pi$. Because of the extremely large number of totally real fields $F$ which need to be considered, we first pare down our list by computing the narrow class number of each field. A large proportion of these fields have narrow class number one, and in fact it is a consequence of an extension of the Cohen--Lenstra heuristics that when ordered by absolute value of discriminant, a positive proportion of all totally real fields will have narrow class number one. The derivation of Proposition \ref{classnotypeno} shows that no isospectral-nonisometric unitive groups can arise from a quaternion algebra defined over a totally real field with trivial narrow class number. (The corresponding Eichler orders will be conjugate causing the corresponding surfaces to be isometric.)

As noted in Section \ref{section:orbifoldsearch}, $N\frakD$ and $N\frakN$ can be bounded above by factors which depend only on $F$, hence this reduces to enumerating certain square-free ideals of bounded norm. This latter enumeration can be done using standard techniques. After performing this enumeration we see that every Eichler order $\calO$ for which $\area(X^+)\leq 20\pi$ and $t(\calO)\geq 2$ satisfies $\frakN=(1)$ and is thus a maximal order. Because conjugate maximal orders will produce isometric surfaces, we eliminate all quaternion algebras $B$ of discriminant $\frakD$ which are unramified a unique real place and which have type number $1$.

As we are searching for torsion-free unitive groups we must ensure that $\Gamma^1$ is torsion-free. We do so by employing Lemma \ref{lemcycquad}; that is, we enumerate all quadratic cyclotomic extensions of $F$ and check to see if any embed into $B$ using the Albert-Brauer-Hasse-Noether theorem. As was noted above, the group $\Gamma^1$ will be torsion-free if and only if none of these cyclotomic extensions embed into $B$. 

After ruling out all unitive groups containing a subgroup $\Gamma^1$ which contains elements of finite order we see that the only groups remaining are the ones which were considered in Example \ref{exm:manifoldexample}. This concludes the proof of Theorem \ref{theorem:manifoldexamples}.

\section{Isospectral but not isometric hyperbolic 3-manifolds and orbifolds}

In this section we exhibit examples of hyperbolic $3$-orbifolds and $3$-manifolds that are isospectral but not isometric. Note that by the Mostow-Prasad Rigidity Theorem these orbifolds will have non-isomorphic fundamental groups. Our method follows in the same vein as the previous two sections, but requires greater computational effort.

\subsection*{Minimal examples}




We begin by exhibiting the minimal isospectral $3$-orbifolds and manifolds.

\begin{exm} \label{example:3orbifolds}
The best 3-orbifold example we found is as follows. Let $F=\Q(w)$ where $w^6 - 2w^5 - w^4 + 4w^3 - 4w^2 + 1=0$. Then $F$ is number field of degree $6$ with exactly $4$ real places, discriminant $-974528=-2^6\cdot 15227$, Galois group $S_6$ and ring of integers $\Z_F=\Z[w]$. The class number of $F$ is $\#\Cl(\Z_F)=1$ and the narrow class number of $F$ is $\Cl^+(\Z_F)=2$; the nontrivial class is represented by the ideal $\frakb=(-4w^5 + 6w^4 + 5w^3 - 8w^2 + 11w + 1)$ with $N\frakb=12281$.

Let $B$ be the quaternion algebra over $F$ which is unramified at all finite places of $F$ and at the four real places of $F$. Then $\frakD=(1)$ and we see that $B=\quat{-1,-1}{F}$.

Let $\calO\subset B$ be a maximal order. The class number with modulus equal to the product of the four ramified real places of $F$ is $\#\Cl^{(+)}(\Z_F)=2$, hence the type number $t(\calO)$ of $\calO$ is $2$ by Proposition \ref{classnotypeno}. Let $\calO_1$ and $\calO_2$ be representatives of the isomorphism classes of maximal orders of $B$.

We compute that $[\Gamma_i^*:\Gamma_i^1]=[\Gamma_i^+:\Gamma_i^1]=4$, hence the $3$-orbifolds $X(\Gamma_1^+)$ and $X(\Gamma_2^+)$ have volume $2.83366\ldots$. The orbifolds $X(\Gamma_1^+)$ and $X(\Gamma_2^+)$ are not isometric by Proposition \ref{prop:isometrytheorem}. In order to verify that they are isospectral we must verify that there is no quadratic $\Z_F$-order which is selective for $\calO_1$ and $\calO_2$. We do this by means of Theorem \ref{theorem:eichlerorderselectivity}. The ray class field with modulus equal to the product of the four real places of $F$ is $L=F(\sqrt{u})$ where $u=w^4 - w^2 + 2w - 1 $. The field $L$ does not embed into $B$ however, as the first real place of $F$ ramifies in $B$ yet splits in $L/F$. Theorem \ref{theorem:eichlerorderselectivity} therefore implies that no quadratic order is selective for $\calO_1$ and $\calO_2$. We conclude that $X(\Gamma_1^+)$ and $X(\Gamma_2^+)$ are isospectral-nonisometric.


We find that in each of the groups $\Gamma_1^*,\Gamma_2^*$ there are $16$ conjugacy classes of elements of order $2$ and one class each of orders $3$ and $4$.  We also find that 
\[ \Gamma_1^+/(\Gamma_1^+)^{\textup{ab}} \cong \Gamma_2^{+}/(\Gamma_2^{+})^{\textup{ab}} \cong (\Z/2\Z)^3 \]
and that the groups $\Gamma_1^*,\Gamma_2^*$ can be generated by $4$ elements: one has relations
\begin{align*}     
&\gamma_1^2 = \gamma_2^2 = (\gamma_1  \gamma_4^{-1})^2 = \gamma_3^4 = 
(\gamma_2  \gamma_3  \gamma_4^{-1})^2 = \gamma_4^{-1}  \gamma_1  \gamma_4^2  \gamma_3^{-1}  \gamma_2  \gamma_1  \gamma_4^3  \gamma_3^{-1}  \gamma_2 = 1 \\
&    \gamma_3^{-1}  \gamma_2  \gamma_1  \gamma_3^{-1}  \gamma_2  \gamma_1  \gamma_4^{-1}  \gamma_2  \gamma_4^{-1}  \gamma_2  \gamma_3 
     \gamma_1  \gamma_2  \gamma_3  \gamma_1  \gamma_2 = 
    \gamma_2  \gamma_3  \gamma_4^{-1}  \gamma_1  \gamma_4  \gamma_2  \gamma_3  \gamma_4^{-1}  \gamma_3  \gamma_1  \gamma_4^3  
    \gamma_3^{-1}  \gamma_2  \gamma_3  \gamma_1  \gamma_4^3  \gamma_3^{-1} = 1 \\
&    \gamma_2  \gamma_1  \gamma_4^{-1}  \gamma_2  \gamma_3  \gamma_1  \gamma_2  \gamma_3  \gamma_1  \gamma_3^{-1}  \gamma_2  \gamma_4 
     \gamma_1  \gamma_2  \gamma_3  \gamma_1  \gamma_3^{-1}  \gamma_2  \gamma_4  \gamma_1  \gamma_2  \gamma_3  \gamma_1 = 1 
\end{align*} 
and the other
\begin{align*}
&    \gamma_4^2 = \gamma_2^3 = (\gamma_3^{-1}  \gamma_2^{-1})^2 = (\gamma_4  \gamma_1)^2 = (\gamma_4  \gamma_3^{-1})^2 = \gamma_4  \gamma_3^{-1}  \gamma_2^{-1}  \gamma_4  \gamma_2  \gamma_3  \gamma_4  \gamma_2^{-1}  \gamma_4  \gamma_2 = 1 \\
&    (\gamma_3  \gamma_1^{-2})^4 = (\gamma_2  \gamma_1^{-3}  \gamma_3  \gamma_1  \gamma_3)^2 = \gamma_1^{-3}  \gamma_3  \gamma_1  \gamma_4  \gamma_3^{-1}  \gamma_1^3  \gamma_3^{-1}  \gamma_1^{-1}  \gamma_4  \gamma_3 = 1
\end{align*} 
but we warn the reader that there is nothing canonical about these presentations---we have merely simplified the presentation as much as we can.  

Fundamental domains for these two groups are exhibited in Figure \ref{poly1}.  We thank Aurel Page for computing the group presentations and fundamental domains for these examples using his \textsc{Magma} package \cite{Page}.
\end{exm}

\begin{figure}
\begin{equation} \label{poly1} \notag
\includegraphics[scale=0.40]{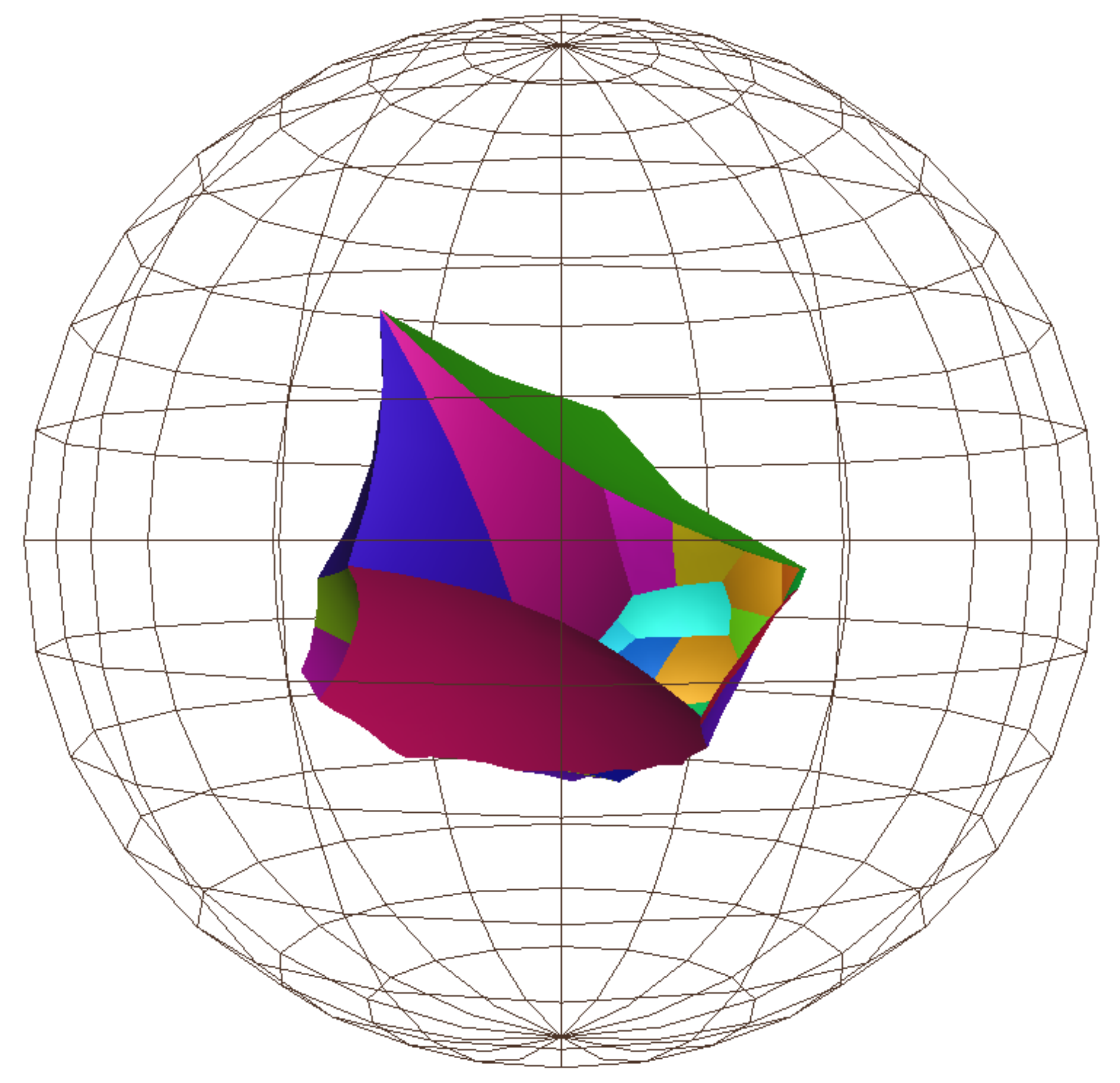}
\end{equation}
\begin{equation*} 
\includegraphics[scale=0.40]{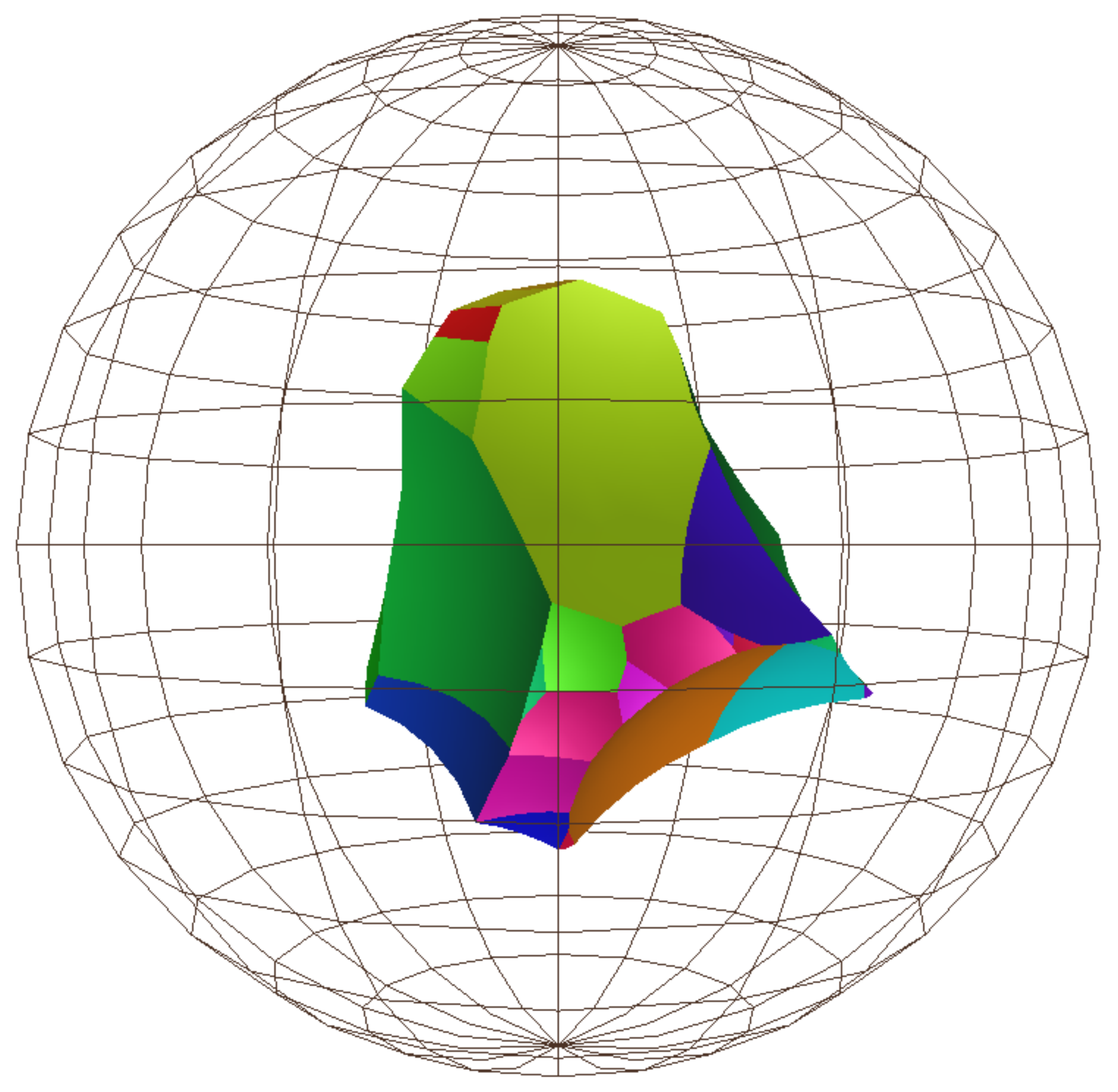} 
\end{equation*}
\centering
\textbf{Figure \ref{poly1}}: Fundamental domains for the orbifolds $X(\Gamma_1^*),X(\Gamma_2^*)$ over sextic $F$
\addtocounter{equation}{1}
\end{figure}

\begin{exm} \label{example:3manifolds}
The best 3-manifold example we found is as follows. Let $F=\Q(w)$ where $w^3 - w^2 + 3w - 4=0$. Then $F$ is a complex cubic number field with discriminant $-331$, Galois group $S_3$ and ring of integers $\Z_F=\Z[w]$. The class number of $F$ is $\#\Cl(\Z_F)=2$; the nontrivial class is represented by the ideal $\frakb=(2,w)$ with $N\frakb=2$. The narrow class number of $F$ is $\#\Cl^+(\Z_F)=2$, hence the class group and narrow class group coincide.

Let $B$ be the quaternion algebra over $F$ which is ramified at the prime $\frakp=(w^2-1)$ and the unique real place of $F$. We may take $B=\quat{-3w^2 + 6w - 10,-13}{F}$.

As in Example \ref{example:3orbifolds} we easily check that the type number of $B$ is equal to $2$. Let $\calO_1$ and $\calO_2$ be representatives of the isomorphism classes of maximal orders of $B$. We check that $[\Gamma_i^+:\Gamma_i^1]=2$, hence $X(\Gamma_1^+)$ and $X(\Gamma_2^+)$ have volume $39.2406\ldots$ and are nonisometric by Proposition \ref{prop:isometrytheorem}. We claim that $X(\Gamma_1^+)$ and $X(\Gamma_2^+)$ are $3$-manifolds. We show this via Lemma \ref{lemcycquad}.  If $F(\zeta_{2q})$ is a cyclotomic quadratic extension of $F$ then $\Q(\zeta_{2q})^+ \subseteq F$. Because $F$ is primitive, the only cyclotomic extensions of $F$ which are quadratic are $K=F(\sqrt{-1})$ and $K=F(\sqrt{-3})$. Neither of these extensions embed into $B$ however, as the prime $\frakp$ ramifies in $B$ and splits in both extensions. To verify that the manifolds $X(\Gamma_1^+)$ and $X(\Gamma_2^+)$ are isospectral we must show that there does not exist a quadratic $\Z_F$-order which is selective for the maximal orders $\calO_1,\calO_2$. This follows immediately from Theorem \ref{theorem:eichlerorderselectivity} though, as no quaternion algebra which ramifies at a finite prime can exhibit selectivity. This proves that the manifolds 
$X(\Gamma_1^+)$ and $X(\Gamma_2^+)$ are isospectral, finishing the proof of Theorem D. \end{exm}

We now prove Theorem C, which restate for convenience.

\begin{theorem}\label{theorem:best3orbifold}
The smallest volume of an isospectral-nonisometric pair of $3$-orbifolds associated to unitive Kleinian groups is $2.8340\ldots.$  This bound is attained uniquely by the pair of orbifolds in Example \ref{example:3orbifolds}, up to isomorphism.
\end{theorem}

\begin{proof}
Let $F$ be a number field of degree $n$ which possesses a unique complex place. As in \cite{Darmon-Book} we will call such a field an \defi{almost totally real (ATR)} field. Let  $B$ be a quaternion algebra over $F$ of discriminant $\frakD$ which is ramified at all real places of $F$. 

Lemma \ref{lem:gamma1*} and (\ref{volumeformula}), along with the estimates $\zeta_F(2)\geq 1$ and $\omega_2(\frakD)\geq 0$ show that if $\vol(X(\Gamma^+))\leq 2.835$ then

\begin{equation}\label{equation:initial3orbifoldvolumebound}
\delta_F \leq (11.595)(0.172)^{1/n}(1.588)^{m/n}\Phi(\frakD)^{-2/3n}
\end{equation}

We claim that the degree $n$ of $F$ satisfies $n\leq 8$. Suppose for instance that $n=9$. By combining the estimates $m\leq n$ and $\Phi(\frakD)\geq 1$ with the equation (\ref{equation:initial3orbifoldvolumebound}) we see that $\delta_F\leq 15.142$. We now apply the Odlyzko bounds to the narrow class field of $F$, a number field of degree at least $2^{m-1}hn$ (and some signature), where $h$ is the class number of $F$. In this case the Odlyzko bounds imply that a number field of degree at least $2^3n$ has root discriminant at least $15.356$. This implies that $m\leq 3$, hence $\delta_F\leq 11.125$ by (\ref{equation:initial3orbifoldvolumebound}). Repeating this argument shows that $m\leq 2$ and $\delta_F\leq 10.568$. Because the algebra $B$ is ramified at all $7$ real places of $F$, it follows that $B$ must ramify at some finite prime of $F$. The discriminant bounds of Odlyzko and Poitou show that if $F$ contains a prime of norm less than $9$ then $\delta_F\geq 10.568$, hence $B$ is ramified at some prime of norm at least $9$. This implies that $\Phi(\frakD)\geq 8$ so that another application of (\ref{equation:initial3orbifoldvolumebound}) shows that $\delta_F\leq 9.059$. This is a contradiction as the Odlyzko bounds show that $F$ must have root discriminant $\delta_F\geq 10.138$.

In this manner we also obtain upper bounds for the root discriminant $\delta_F$ in the cases in which $2\leq n \leq 8$, which we refer to as the Odlyzko-Poitou-Armitage-Fr\"ohlich (OPAF) bounds in (\ref{odlyzkotabledim3prelim}). For the convenience of the reader we also list the Odlyzko bound (both unconditional and the bound conditional on GRH) for an ATR field of degree $n$.

\begin{equation} \label{odlyzkotabledim3prelim}
\begin{array}{c||cc|c|c|c}
n & \textup{Odlyzko} & \textup{Odlyzko (GRH)} & \textup{OPAF} & \Delta & \textup{CPU time} \\
\hline
\rule{0pt}{2.5ex} 
2 & 0.468 & (1.722) & 7.637 & - \\
3 & 1.736 & (2.821) & 10.240 & - \\
4 & 3.344 & (3.963) & 10.564 & 12 & \text{2 min} \\
5 & 4.954 & (5.325) & 10.762 & 11 & \text{20 min} \\
6 & 6.438 & (6.524) & 10.897 & 11 & \text{2 days} \\
7 & 7.763 & (7.963) & 10.291 & 10.5 & \text{2 weeks} \\
8 & 8.973 & (9.270) & 10.446 & 10.5 & \text{1.5 years} \\
\end{array}
\end{equation}

ATR fields of degree $2$ are easy to enumerate (corresponding to negative fundamental discriminants); cubic ATR fields can be very efficiently enumerated by an algorithm due to Belabas \cite{Belabas}.  For degrees $n=4,5$, existing tables \cite{lmfdb} would be sufficient; but for completeness we enumerate fields in degrees $4 \leq n \leq 8$ with root discriminant $\delta \leq \Delta$ as in the above table.  

Techniques for enumerating number fields are standard, relying on the geometry of numbers (Hunter's theorem): for a description and many references, see the book by Cohen \cite[\S 9.3]{Cohen2} and further work of Kl\"uners and Malle \cite{KlunersMalle}.  We define the Minkowski norm on a number field $F$ by $T_2(\alpha)=\sum_{i=1}^n |\alpha_i|^2$ for $\alpha \in F$, where $\alpha_1,\alpha_2,\dots,\alpha_n$ are the conjugates of $\alpha$ in $\C$.  The norm $T_2$ gives $\Z_F$ the structure of a lattice of rank $n$.  In this lattice, the element $1$ is a shortest vector, and an application of the geometry of numbers to the quotient lattice $\Z_F/\Z$ yields the following result.

\begin{lem}[Hunter] 
There exists $\alpha \in \Z_F \setminus \Z$ such that $0 \leq \Tr(\alpha) \leq n/2$ and
\[ T_2(\alpha) \leq \frac{\Tr(\alpha)^2}{n} + \gamma_{n-1}\left(\frac{|d_F|}{n}\right)^{1/(n-1)} \]
where $\gamma_{n-1}$ is the $(n-1)$th Hermite constant.
\end{lem}

Therefore, if we want to enumerate all number fields $F$ of degree $n$ with $\delta_F \leq \Delta$, an application of Hunter's theorem yields $\alpha \in \Z_F \setminus \Z$ such that $T_2(\alpha) \leq C$ for some $C \in \R_{>0}$ depending only on $n,\Delta$.  We thus obtain bounds on the power sums
\[ |S_k(\alpha)|=\biggl|\sum_{i=1}^n \alpha_i^k\biggr| \leq T_k(\alpha)=\sum_{i=1}^n |\alpha_i|^k \leq nC^{k/2}, \]
and hence bounds on the coefficients $a_i \in \Z$ of the characteristic polynomial 
\[ f(x)=\prod_{i=1}^n (x-\alpha_i) = x^n + a_{n-1} x^{n-1} + \dots + a_0 \] 
of $\alpha$ by Newton's relations:
\begin{equation}
S_k+\sum_{i=1}^{k-1}a_{n-1}S_{k-i} + ka_{n-k}=0.
\label{Newton}
\end{equation}
This then yields a finite set of polynomials $f(x) \in \Z[x]$ such that every $F$ is represented as $\Q[x]/(f(x))$ for some $f(x)$.  Additional complications arise from the fact that $\alpha$ as given by Hunter's theorem may only generate a subfield $\Q\subset \Q(\alpha) \subsetneq F$ if $F$ is imprimitive.  The size of the set of resulting polynomials is $O(\Delta^{n(n+2)/4})$ (see Cohen \cite[\S 9.4]{Cohen2}).

We highlight three techniques that bring this method within range of practical computation.  The first is due to Pohst \cite{Pohst}: choosing a possible value of (nonzero) $a_0 \in \Z$ with $|a_0|^2 \leq (C/n)^n$ (using the arithmetic-geometric mean inequality), we use an efficient form of the method of Lagrange multipliers to obtain further bounds on the coefficients $a_i$.  

Second, and more significantly (especially so in larger degree), we use a modified version of the method involving Rolle's theorem \cite[\S 2.2]{voight-fields}, adapted for use with ATR fields.  Given values $a_{n-1},a_{n-2},\dots,a_{n-k}$ for the coefficients of $f(x)$ for some $k \geq 2$, we deduce bounds for $a_{n-k-1}$.  For simplicity, we explain the method for $k=n-1$; the method works more generally whenever a derivative $f^{(k)}$ is ATR.  Since $f(x)$ is ATR and separable, by Rolle's theorem it follows that there is a real root of $f'(x)$ between any two consecutive real roots of $f(x)$, so there are at least $n-3$ real roots of $f'(x)$; but the converse holds as well: if there are exactly $n-3$ real roots $\beta_1<\dots<\beta_{n-3}$ of $f'(x)$, then the $n-2$ real roots of $f(x)$ must be interlaced with the real roots of $f'(x)$.  Since $f=g(x)+a_0$ is monic, it follows that 
\[ f(\beta_{n-3}) < 0, \quad f(\beta_{n-4})=g(\beta_{n-4})+a_0 > 0, \text{etc.} \]
and knowing $g(x)$ and the roots $\beta_i$ gives bounds on $a_0$.  This simple trick (which can be implemented quite efficiently) eliminates consideration of a large range of coefficients.  

Finally, using the fact that the number of real roots is a continuous function of the coefficients and changes in a (real) one-parameter family only when the discriminant changes sign, we can quickly advance out of coefficient regions which are not ATR.

The computations were performed on a standard desktop machine except for $n=8$, where the computations were hosted graciously at the Vermont Advanced Computing Center (VACC); approximate timings are given in Table \ref{odlyzkotabledim3prelim}.

To conclude the proof of Theorem \ref{theorem:best3orbifold} we proceed exactly as in the proofs of Theorems \ref{theorem:bestorbifold} and \ref{theorem:manifoldexamples}: we enumerate all unitive Kleinian groups $\Gamma$ with $\covol(\Gamma)\leq 2.835$ (via (\ref{volumeformula}) and Lemma \ref{lem:gamma1*}) and confirm that in each case the type number of the associated quaternion algebra $B$ is equal to one or else employ Theorem \ref{theorem:eichlerorderselectivity} to exhibit a selectivity obstruction which precludes the possibility of isospectrality.
\end{proof}

\begin{rmk} \label{rmk:nonunitive3orbifold}
We briefly comment on what would be needed to extend Theorem \ref{theorem:best3orbifold} to maximal arithmetic Kleinian groups. The difficulty in carrying out such an extension is due to Proposition \ref{corgamma}: one must determine an upper bound for what is essentially $\#(\Cl \Z_F)[2]$. One approach is to observe that $\#(\Cl \Z_F)[2]$  is less than the class number $\#\Cl \Z_F$ of $F$, which can be bounded above by means of the class number formula by using a lower bound for the regulator of $F$ (for instance that of Friedman \cite{Friedman}) along with an upper bound for the residue at $s=1$ of the Dedekind zeta function of $F$ (for instance that of Louboutin \cite{Louboutin-classnumberbounds}). Even if one assumes the GRH and Dedekind's conjecture in order to take advantage of the strongest version of the aforementioned bounds as well as the Odlzyko-Poitou bounds, we still find that all ATR fields of degree $8$ and root discriminant $\delta_F\leq 11.845$ and degree $9$ and root discriminant $\delta_F\leq 11.467$ must be enumerated. Unfortunately carrying out such large enumerations is not currently feasible. See Remark \ref{rmk:nonunitive3manifold} below.

We note that one way to circumvent many of the aforementioned difficulties and obtain an unconditional result would be to obtain good bounds for the dimension over $\F_2$ of the homology group $H_1(M;\F_2)$, where $M$ is a closed, orientable hyperbolic $3$-orbifold. When $M$ is a closed, orientable, hyperbolic $3$-manifold with volume less than $3.08$, Culler and Shalen \cite{culler-shalen} have proven that $\dim_{\F_2} H_1(M;\F_2)\leq 5$; several similar results of this kind are also known \cite{CS,ACS}.  Shalen \cite{shalen} is currently working on extending this result to $3$-orbifolds and hopes to show bounds of the same type for a closed, orientable hyperbolic $3$-\emph{orbifold}.  Such a bound would allow for a short proof that there are no representation equivalent-nonisometric maximal arithmetic Kleinian groups with volume less than the bound provided in such a theorem.
\end{rmk}

\subsection*{Some remarks on Theorem D}

Although our searches were extensive only in low ranges, it would not be unreasonable to hope that the isospectral-nonisometric $3$-manifolds described in Example \ref{example:3manifolds} are in fact the smallest volume isospectral-nonisometric $3$-manifolds arising from unitive Kleinian groups.  Indeed, we have searched through all ATR fields of degree at most $7$ and absolute value of discriminant at most $10^6$ and have been unable to find any smaller volume examples. In this section we will provide root discriminant bounds on the field of definition of a pair of isospectral-nonisometric $3$-manifolds arising from arithmetic unitive Kleinian groups having volume less than $39.241$. Due to the difficulty of actually enumerating these ATR fields however, we are unable to prove that our example is in fact the best possible.

As for $3$-orbifolds, we let $F$ be an ATR field of degree $n$ and $B$ be a quaternion algebra defined over $F$ of discriminant $\frakD$ which ramifies at all real places of $F$. If $\Gamma^+$ arises from $B$ and satisfies $\vol(X(\Gamma^+))\leq 39.241$ then the same estimates used in the derivation of (\ref{equation:initial3orbifoldvolumebound}) show
\begin{equation}\label{equation:3manifoldequation}
\delta_F \leq (11.595)(0.996)^{1/n}(1.588)^{m/n}.
\end{equation}
Using the same ideas that were used in the proof of Theorem \ref{theorem:best3orbifold} to show that the defining field had degree at most $8$, we easily see that in our current situation we have $n\leq 10$. When $2\leq n \leq 10$ we have the following root discriminant bounds for $F$:

\begin{equation} \label{odlyzkotabledim3manifold}
\begin{array}{c||cc|cc|c}
n & \textup{Odlyzko} & \textup{Odlyzko (GRH)} & \textup{OPAF} \\
\hline
\rule{0pt}{2.5ex} 
2 & 0.468 & (1.722) & 18.376 \\
3 & 1.736 & (2.821) & 13.510 \\
4 & 3.344 & (3.963) & 14.597 \\
5 & 4.954 & (5.325) & 15.291 \\
6 & 6.438 & (6.524) & 15.772 \\
7 & 7.763 & (7.963) & 15.094 \\
8 & 8.973 & (9.270) & 13.784 \\
9 & 10.138 & (10.529) & 13.522 \\
10 & 11.258 & (11.664) & 13.316 \\
\end{array}
\end{equation}

\begin{rmk} \label{rmk:nonunitive3manifold}
A Hunter search for number fields of degree $n$ with root discriminant at most $\Delta$ currently require checking $O(\Delta^{n(n+2)/4})$ polynomials, and it would be very interesting to find a method in moderate degree $n > 6$ which is an improvement on this upper bound (work of Ellenberg-Venkatesh \cite{EV} gives some theoretical evidence that this may be possible)---the many existing tricks seem only to cut away at the constant.  (Note that this bound is sharp for $n=2$.)  The slight increase in the bounds in Table \ref{odlyzkotabledim3manifold} over those in Table \ref{odlyzkotabledim3prelim} are in fact quite prohibitive to compute in practice: for example, in degree $8$ we expect to have to work $(13.784/10.5)^{20} \approx 231$ as hard, and for degrees $9,10$ we expect truly significant efforts will be required.  We hope instead that recent techniques in the understanding of $3$-manifolds can be brought to bear on this problem.
\end{rmk}

\end{document}